\documentclass[10pt,reqno]{amsart}
\usepackage{setspace}
\usepackage[utf8]{inputenc}
\usepackage{graphicx}
\graphicspath{ {./images/} }
\usepackage[pagewise]{lineno}
\usepackage{amssymb,amstext, amsthm,amsmath}
\usepackage[foot]{amsaddr}
\usepackage[all,arc]{xy}
\usepackage{enumerate}
\usepackage{mathrsfs}
\usepackage[margin=1in]{geometry} 
\usepackage{amsmath,amsthm}
\usepackage{mathtools}
\usepackage[T1]{fontenc}
\usepackage{fourier}
\usepackage{stmaryrd}
\usepackage{amsfonts}
\usepackage{cite}

\usepackage[backref=page]{hyperref}
\usepackage[capitalise,nameinlink]{cleveref}

\newcommand *\w{^\wedge}

\newcommand{\fint}{{\vbox{\hbox{--}}\kern-0.8em}\int}
\newtheorem{thm}{Theorem}[section]
\theoremstyle{definition}

\newtheorem{example}[thm]{Example}

\newtheorem{theorem}{Theorem}[section]
\newtheorem{lemma}[theorem]{Lemma}
\newtheorem{proposition}[theorem]{Proposition}
\newtheorem{corollary}[theorem]{Corollary}
\newtheorem{definition}[theorem]{Definition}
\newtheorem{remark}[theorem]{Remark}

\usepackage{xcolor}
\hypersetup{
	colorlinks,
	linkcolor={blue!100!black},
	citecolor={red!100!black},
	urlcolor={green!100!black}
}
\DeclarePairedDelimiterX{\inp}[2]{\langle}{\rangle}{#1, #2}

\makeatletter
\def\@tocline#1#2#3#4#5#6#7{\relax
	\ifnum #1>\c@tocdepth 
	\else
	\par \addpenalty\@secpenalty\addvspace{#2}%
	\begingroup \hyphenpenalty\@M
	\@ifempty{#4}{%
		\@tempdima\csname r@tocindent\number#1\endcsname\relax
	}{%
		\@tempdima#4\relax
	}%
	\parindent\z@ \leftskip#3\relax \advance\leftskip\@tempdima\relax
	\rightskip\@pnumwidth plus4em \parfillskip-\@pnumwidth
	#5\leavevmode\hskip-\@tempdima
	\ifcase #1
	\or\or \hskip 1em \or \hskip 2em \else \hskip 3em \fi%
	#6\nobreak\relax
	\dotfill\hbox to\@pnumwidth{\@tocpagenum{#7}}\par
	\nobreak
	\endgroup
	\fi}
\makeatother

\usepackage{hyperref}

\numberwithin{equation}{section}

\usepackage{epigraph}

\title{Global $C^{1,\alpha}$-Regularity for 
Musielak-Orlicz Equations in Divergence Form}

\author{Hlel Missaoui$^{1}$}
\address{$^{1}$ Mathematics Department,  Higher Institute of Computer Science of Mahdia, University of Monastir, 5111 Mahdia, Tunisia.}
\email{$^{1}$ \tt hlel.missaoui@fsm.rnu.tn}
\date{}

\begin{document}
\begin{abstract}
In this paper, we establish global $C^{1,\alpha}$-regularity for bounded generalized solutions of elliptic equations in divergence form with Musielak-Orlicz growth and subject to Dirichlet or Neumann boundary conditions. The structural assumptions on the growth function $G$ are those under which $C^{0,\alpha}$-regularity was obtained in \cite{HBO}, with the additional requirement that the lower coefficients in the sandwich condition do not degenerate; the additional ingredient with respect to \cite{HBO} is the H\"older continuity of $A$ in $(x,u)$, with a logarithmic factor in the gradient variable, needed for the comparison with the frozen problem. Our findings extend and generalize several important regularity results in cases of special attention such as variable exponent spaces, Orlicz spaces, and some $(p,q)$ situations. We also point out new conditions in the analysis that focus on the interplay between non-standard growth conditions and the boundary behaviour in such generalized examples. Degenerate (vanishing) coefficients are not covered: the Colombo--Mingione gap condition, which is needed for the $C^{1,\alpha}$-theory in that case, is not addressed here.
\end{abstract}
\maketitle

\noindent\textbf{\textup{2020} Mathematics Subject Classification.} {Primary: 
35B65, 
35J60; 
Secondary: 
35J25, 35J70, 46E30
}

\smallskip
	
\tableofcontents

\section{Introduction}
\label{sec:intro}
\subsection{Overview}
The regularity theory for elliptic partial differential equations has been a pillar of modern analysis for more than fifty years, and one of its profound results is \( C^{1,\alpha} \)-regularity, which denotes the existence of Hölder continuous first derivatives of weak or generalized solutions. The results display regularity inherent in the solutions, and provide important information about the qualitative behavior of solutions associated with elliptic nonlinear equations. The current paper carries on the line of investigation that was first explored in \cite{HBO}, where the authors established \( C^{0,\alpha} \) (Hölder) continuity of bounded weak solutions under Musielak--Orlicz growth conditions. We are now able to prove global \( C^{1,\alpha} \)-regularity for bounded generalized solutions of elliptic equations in divergence form
\begin{equation}\label{eqs}
-\mathrm{div}\, A(x,u,Du) = B(x,u,Du),
\end{equation}
for both Dirichlet and Neumann boundary problems, in a single analytical framework.\\

Elliptic regularity theory originated with the work of De Giorgi \cite{deGiorgi1957}, Nash \cite{Nash1958}, and Moser \cite{Moser1960,Moser1961}, who independently established the Hölder continuity of weak solutions to linear elliptic equations with bounded measurable coefficients. This famous result (the De Giorgi–Nash–Moser theorem) formed the foundation for everything else concerning regularity in nonlinear elliptic equations. For equations with standard polynomial growth, Uraltseva \cite{Uraltseva1968}, Uhlenbeck \cite{Uhlenbeck1977} and DiBenedetto \cite{4} provided \( C^{1,\alpha} \)-regularity results for the \( p \)-Laplacian and quasilinear equations. Later, Lieberman \cite{Li1988} established fundamental regularity theory for quasilinear elliptic equations with polynomial growth of order $p>1$, under the structural assumption
\[
\lambda |\xi|^{p} \leq A(x,u,\xi)\cdot\xi \leq \Lambda |\xi|^{p},
\]
which ensures uniform ellipticity and coercivity.\\
The Orlicz framework further generalized this theory by replacing the power-type growth $|\xi|^p$ with an $N$-function $G(t)$, leading to structure conditions
\[
\lambda G(|\xi|) \leq A(x,u,\xi)\cdot\xi \leq \Lambda G(|\xi|),
\]
which accommodate a broad class of nonlinearities while preserving the essential features of elliptic regularity \cite{DonaldsonTrudinger1971,Chlebicka2021,Li1991}.

In the late 1980s and early 1990s, Marcellini \cite{marc1,marc2} developed the first systematic theory of elliptic equations with nonstandard $(p,q)$-growth, given by
\[
|\xi|^{p} \lesssim A(x,\xi)\cdot\xi \lesssim |\xi|^{q}, \qquad 1<p\le q.
\]
He showed that the ratio $\frac{q}{p}$ is crucial for regularity, since a large gap between the exponents may lead to discontinuous solutions \cite{Mar1}. Although related ideas had appeared earlier in homogenization theory, particularly in the work of Zhikov, Marcellini’s results laid the foundations for the modern study of elliptic problems with nonuniform growth.

At the turn of the century, Fan and Zhao initiated a systematic investigation into the regularity properties of solutions to quasilinear elliptic equations with non-standard growth conditions, with particular focus on equations of the form
\[
    -\mathrm{div}\big(|D u|^{p(x)-2}D u\big) = f(x,u,D u),
\]
where the exponent $p: \Omega \to \mathbb{R}$ is a measurable function rather than a constant. Their research program unfolded in two complementary directions. First, they developed the fundamental functional-analytic framework by establishing critical embedding theorems and compactness results for variable exponent Lebesgue spaces $L^{p(x)}(\Omega)$ and Sobolev spaces $W^{1,p(x)}(\Omega)$; see~\cite{Fan2012a}. These foundational results provided the necessary infrastructure for applying variational methods and fixed point theory to nonlinear problems with variable exponents.

Building upon these functional-analytic foundations, Fan and Zhao subsequently derived precise regularity estimates for weak solutions of the $p(x)$-Laplacian equation. A key contribution was their identification of the log-Hölder continuity condition
\[
|p(x) - p(y)| \leq \frac{C}{\ln\left(e + \frac{1}{|x-y|}\right)} \quad \text{for all } x,y \in \Omega,
\]
as both necessary and sufficient for obtaining optimal regularity results; see~\cite{FanZhao1999, Fan2007}.

More recently, the study of double-phase functionals of the form
\[
G(x,\xi) = |\xi|^p + a(x)|\xi|^q, \quad 1 < p \leq q, \ a(x) \geq 0,
\]
initiated by Marcellini \cite{marc1} and systematically developed by Zhikov, Baroni, Colombo and Mingione \cite{Baroni2016,Colombo2015,Colombo2015b,Zhikov1987,Zhikov1995}, has revealed subtle regularity phenomena depending on the Hölder continuity of the coefficient $a(x)$. These results demonstrated that even mild oscillations in $a(x)$ can influence the regularity exponent $\alpha$, providing a deeper understanding of mixed growth problems.

In this context, Cristina De Filippis and Jehan Oh have made remarkable contributions, developing a unified framework for double-phase and multi-phase functionals \cite{DF-multiphase,DFS2023}. They established optimal regularity results, improved higher integrability, and derived sharp criteria connecting the Hölder continuity of the coefficients with the resulting smoothness of solutions. In particular, the notion of multi-phase growth, where several distinct energy densities interact,
\[
G(x,\xi) = \sum_{i=1}^{k} a_i(x) |\xi|^{p_i},
\]
provides a natural setting for models arising in electrorheological fluids and composite materials.

Further progress was recently achieved by Sumiya Baasandorj and Sun-Sig Byun in their article \cite{SB2023}, where they extended the double-phase approach to general Orlicz growth structures, providing local $C^{1,\alpha}$-regularity results under controlled modular conditions. Their work demonstrates that many regularity phenomena persist even when the growth is driven by a general Orlicz function, bridging the gap between polynomial, variable exponent, and multi-phase theories.\\

Due to the vastness and rapid development of modern regularity theory, it is clearly impossible to make a complete list of all contributions made to the field. As a result, there are many valuable contributions made by others not specifically addressed, and for this we thank them all. Readers interested in getting more detailed and complete information about this area of research should look at one of the many different types of survey articles or bibliography available, such as those referenced in \cite{mr3, Mingione2006}. In addition, when looking for detailed information on particular aspects of Regularity, we have provided several classic references along with their extensive lists of articles and papers. With regards to the classical Laplacian and linear elliptic theory, the fundamental works that laid the groundwork for this area include \cite{deGiorgi1957, Nash1958, Moser1960, Moser1961}. The results related to the $p$-Laplace operator and operators that exhibit $(p,q)$-growth rates can be found in a number of papers including ~\cite{ GmeinederKristensen2024, Uraltseva1968,  Li1988, Uhlenbeck1977,  DeFilippisKochKristensen2024, KristensenMingione2010, Manfredi1986, Manfredi1988, marc1, marc2, Ac1, Ac2, mg3, MG3,24,19,12,EDI,DF2019,MG2, NF, LE, MH, Pucci}. In the Orlicz setting and with other types of non-standard growth, readers should consult ~\cite{HsK1,HsK2,ACAN,Li1991,HHL21,Simon1 ,Harjulehto2017, Cupini2017, Eleuteri2016, Chlebicka2021, Chlebicka2022, mgg}. For those working on double-phase problems, a good number of developments can be found in ~\cite{Colombo2015, Colombo2015b, Baroni2016, defi,Simon1}. For those doing work on variable exponent $p(x)$-Laplacians, you can look to the following references: ~\cite{Fan1999, FanZhao1999, Acerbi2007, Coscia1999}.\\

The Musielak--Orlicz framework provides a unified and far-reaching extension of many classical growth models by allowing the integrand \( G(x,t) \) to depend simultaneously on the space variable and the gradient magnitude. It encompasses, as particular cases, the variable exponent model \( G(x,t)=t^{p(x)} \), the Orlicz case \( G(x,t)=G(t) \), and the double-phase structure \( G(x,t)=t^{p}+a(x)t^{q} \). Owing to this high level of generality, the Musielak--Orlicz setting offers a natural and flexible analytical framework for describing heterogeneous media and spatially dependent nonlinear phenomena within a single variational and PDE structure. Within this general context, Hästö and Ok \cite{HsK1} established local \( C^{1,\alpha} \)-regularity for local minimizers of Musielak--Orlicz type functionals under remarkably weak and essentially optimal structural assumptions on \( G \). From the variational perspective, their results are therefore strictly stronger than those obtained in the present work.

The objective of this paper is, however, fundamentally different. Rather than studying regularity properties of minimizers, we address the \( C^{1,\alpha} \)-regularity of weak solutions to elliptic equations in divergence form with Musielak--Orlicz growth. This PDE-oriented framework covers a substantially broader class of problems, including equations that do not arise from an underlying energy functional. As a consequence, the greater generality of the equations under consideration inevitably requires stronger structural assumptions on the nonlinearities.

In this sense, the two contributions should be regarded as complementary. While Hästö and Ok obtain optimal regularity results for variational problems under minimal assumptions within the Musielak--Orlicz framework, the present work extends \( C^{1,\alpha} \)-regularity theory to a wide family of elliptic equations beyond the variational setting.

In this work, we consider elliptic equations of divergence form \eqref{eqs}, where the nonlinearities are controlled by a Musielak--Orlicz function \(G(x,t)\) satisfying assumptions \eqref{D22}--\eqref{GG6b}. Under these conditions, we establish global \(C^{1,\alpha}\)-regularity for bounded generalized solutions. The argument combines Gehring-type higher integrability in Musielak--Orlicz spaces, comparison with frozen-coefficient problems, Campanato--Morrey estimates, and a boundary treatment covering both Dirichlet and Neumann conditions. This improves the \(C^{0,\alpha}\)-regularity result from \cite{HBO} and advances the development of a regularity theory for elliptic equations with nonstandard growth. Moreover, our results extend the existing regularity theory in several directions. In contrast to the variable exponent setting \cite{Fan2007}, our framework is not restricted to power-type growth. Compared with the double-phase model \cite{Colombo2015,Baroni2016}, the spatial dependence influences the entire nonlinearity rather than only a coefficient. Relative to the Orlicz case \cite{Chlebicka2021}, our approach bridges the theories of Orlicz and variable exponent growth by allowing explicit dependence on \(x\) in the growth function.

\subsection{Main Results}\label{subsec:main}
We now present our assumptions and main results in a detailed manner, including the structural conditions, boundary settings, and the regularity conclusions for Musielak--Orlicz elliptic equations.\\

Let $\Omega$ be a bounded Lipschitz domain in $\mathbb{R}^n$ with \( n \geq 2 \), and let 
\[
G(x,t) = \int_0^{\lvert t\rvert} g(x,s)\, \mathrm{d}s
\]
be a generalized \textnormal{N}-function (see Section~\ref{sec2} for precise definitions) and consider the following structural conditions:

There exist constants $1 < g^- \leq g^+ < n$ such that, for all $x \in \overline\Omega$ and $t > 0$,
\begin{equation}\label{D22}
1 < g^- \leq \frac{g(x,t)\,t}{G(x,t)} \leq g^+ . \tag{G0}
\end{equation}
This is exactly condition (G0) of \cite{HBO}. In addition, for the frozen problems treated by means of Lieberman's regularity theory (Lemmas~\ref{Lem1.2} and~\ref{Lem2.21}) we assume that $g(x,\cdot)\in C^1((0,\infty))$ and
\begin{equation}\label{D22diff}
g^--1\le \frac{t\,\partial_tg(x,t)}{g(x,t)}\le g^+-1\qquad\forall x\in\overline\Omega,\ t>0. \tag{G0$'$}
\end{equation}
Integrating the differential inequality (G0$'$) over $(0,t)$ gives (G0) with the same exponents $g^\pm$, so (G0$'$) is stronger than (G0) and introduces no new constant; it is used only through Lieberman's theory for the frozen problems and through the monotonicity of $t\mapsto g(x,t)\,t^{1-g^\pm}$ (Proposition~\ref{Prop2.2}, Lemma~\ref{lem:orlicz-ineq}).

There exists $F \geq 1$ such that for almost every $x \in \mathbb{R}^n$:
\begin{equation}\label{GG1}
F^{-1} \leq G(x,1) \leq F. \tag{G1}
\end{equation}

There exists $\mu \in (0,1]$ such that for every ball $B \subset \mathbb{R}^n$ with $|B| \leq 1$, for every $t \in [1, |B|^{-1}]$, and for almost all $x,y \in B$:
\begin{equation}\label{GG2}
\mu G^{-1}(x,t) \leq G^{-1}(y,t). \tag{G2}
\end{equation}

Let $\delta\ge1$, $\eta\ge0$, $\delta+\eta>1$ and $\ell(t):=\log(\delta+\eta t)$; thus $\ell$ is nondecreasing, $\ell\ge0$ on $[0,\infty)$ and $\ell(t)\ge\ell(1)=\log(\delta+\eta)>0$ for $t\ge1$. There exist $\Theta\ge1$, $\bar\alpha\in(0,1]$, $\bar a>0$ and, for $i=1,2$, functions $p_i,q_i,p_i',q_i':\overline\Omega\to(1,n)$ and $a_i,b_i,a_i',b_i':\overline\Omega\to[0,\bar a]$ such that:
\begin{enumerate}[(i)]
\item for all $x\in \overline{\Omega}$ and $t\ge1$,
\begin{equation}\label{GG3}
\sum_{i=1}^{2}\left[a_{i}^{\prime}(x)t^{p_{i}^{\prime}(x)} + b_{i}^{\prime}(x)t^{q_{i}^{\prime}(x)}\ell (t)\right]\leq G(x,t)\leq \sum_{i = 1}^{2}\left[a_{i}(x)t^{p_{i}(x)} + b_{i}(x)t^{q_{i}(x)}\ell (t)\right]; \tag{G3}
\end{equation}
\item
\[
\begin{gathered}
p_{i}\leq q_{i},\quad p_{i}^{\prime}\leq q_{i}^{\prime}\quad\text{on }\overline{\Omega},\qquad p^{-} = \min_{i}\inf_{\overline{\Omega}}p_{i}^{\prime} > 1,\\
p^{+}:= \max_{i}\sup_{\overline{\Omega}}p_{i},\qquad q^{+}:= \max_{i}\max \left\{\sup_{\overline{\Omega}}q_{i},\sup_{\overline{\Omega}}q_{i}^{\prime}\right\} < n;
\end{gathered}
\]
\item $a_{i}^{\prime}(x) + b_{i}^{\prime}(x)\geq d > 0$ for all $x$ and each $i = 1,2$;
\item  there are $L_{0} > 0$ and $R_{0}\in (0,\tfrac{1}{2})$ such that for all $x_{0}\in \overline{\Omega}$, $R< R_{0}$, $i = 1,2$,
\[
\sup_{B_{R}(x_{0})}p_{i} - \inf_{B_{R}(x_{0})}p_{i}^{\prime}\leq \frac{L_{0}}{|\ln(R)|},\qquad \sup_{B_{R}(x_{0})}q_{i} - \inf_{B_{R}(x_{0})}q_{i}^{\prime}\leq \frac{L_{0}}{|\ln(R)|};
\]
\item $a_{i}^{\prime},b_{i}^{\prime}\in C^{0,\bar\alpha}(\overline{\Omega})$ with seminorms $[a_{i}^{\prime}]_{0,\bar\alpha},[b_{i}^{\prime}]_{0,\bar\alpha}$ and $b_{i}\leq \Theta b_{i}^{\prime}$ on $\overline{\Omega}$;
\item non-degeneracy, which is the alternative (G3b) of \cite{HBO} (the alternative (G3a) is not used here):
\begin{equation}\label{GG6b}
\beta_{0}:= \min_{i}\inf_{\overline{\Omega}}b_{i}^{\prime} > 0. \tag{G3b}
\end{equation}
\end{enumerate}
For all $x \in \overline{\Omega}$,
\begin{equation}\label{GG4}
\int_0^1 \frac{G^{-1}(x,t)}{t^{\frac{n+1}{n}}}\, \mathrm{d}t < +\infty. \tag{G4}
\end{equation}

\begin{align}\label{GG5}
    \begin{split}
      G_* \in &C^{1-0}(\operatorname{Dom}(G_*))\text{ and there exist positive constants } \delta_0 < \dfrac{1}{n},\ C_0\\
      &\text{and }0 < t_0 < \min_{x\in \mathbb{R}^n}T(x)=+\infty\text{ such that }\\
      & |\nabla_x G_*(x,t)| \leq C_0 \bigl(G_*(x,t)\bigr)^{1+\delta_0},\\
      &\text{for } x \in \mathbb{R}^n\text{ and } t \in [t_0,+\infty) \text{ provided } \nabla_x G_*(x,t)\text{ exists},
    \end{split}\tag{G5}
\end{align}
where 
\[
G_*(x,t) = \int_0^t g_*(x,s)\, \mathrm{d}s
\]
is the  Fan's Sobolev conjugate of \( G(x,t) \) (see Definition \ref{DefF}) and $T(x)$ is defined in \eqref{Tx}.\\
We write $L:=e^{L_0}$ and $\omega(R):=\frac{L_0}{|\ln R|}$ for $R<R_0$; the other quantities derived from \eqref{GG3} ($d_*$, $d_\ell$, $H_*$, $[b]_*$, $\mu_*$, $R_*$) are those of \cite{HBO} and are recalled in Subsection~\ref{subsec:notation}. For the definition of $G^{-1}$ see Definition \ref{definv}.

\begin{remark}\label{rem:G3b}
The $C^{0,\alpha}$-theory of \cite{HBO} uses conditions (G0)--(G3)(i)--(v) plus either the natural-gap alternative (G3a),
\[
C_*:=\max_i\sup_{x\in\overline\Omega,\,t\ge1}t^{q_i(x)-p_i(x)-\bar\alpha}\ell(t)<\infty,
\]
or the non-degeneracy (G3b), $\beta_0>0$. Here only the non-degeneracy (G3b) is used. This is required by the method: the gradient is not a priori bounded, and for degenerate coefficients the ratio $\frac{\sup_{B_R}G(\cdot,t)}{\inf_{B_R}G(\cdot,t)}$ grows like a power of $Rt$, not like $t^{\omega(R)}$ (Lemma~\ref{lem:comp-factor} and Remark~\ref{rem:degenerate} below).
\end{remark}

\begin{definition}[see Sect.~1, Chap.~I of \cite{OL}]\label{Definition 2.4}
We say that the boundary $\partial \Omega$ of $\Omega$ satisfies condition \eqref{AO} if there exist positive constants $m_0$ and $\theta_0$ such that for any ball $B_\rho$ with center on $\partial \Omega$ and radius $\rho \leq m_0$, and for any connected component $\Omega'_\rho$ of $B_\rho \cap \Omega$, the following inequality holds:
\begin{equation}\label{AO}
|\Omega'_\rho| \leq (1 - \theta_0)|B_\rho|. \tag{A}
\end{equation}
\end{definition}

\begin{remark}\label{rem:density}
If $\partial\Omega$ is Lipschitz, in particular if $\partial\Omega\in C^{1,\beta_3}$, then \eqref{AO} holds with constants $m_0,\theta_0$ depending only on $n$ and on the Lipschitz character of $\partial\Omega$. For $C^{1,\beta_3}$ domains we always choose $m_0\in(0,1]$ and $\theta_0$ so that \eqref{AO} holds; this is the only property of the exterior of $\Omega$ used through Proposition~\ref{prop1.1}(2) (the Sobolev--Poincar\'e steps near $\partial\Omega$ in Section~\ref{sec:proofs} are carried out after flattening the boundary, where the exterior density is explicit).
\end{remark}

Consider the elliptic equation in divergence form:
\begin{equation}\label{P}
\operatorname{div} A(x, u, Du) + B(x, u, Du) = 0, \quad x \in \Omega. \tag{$\mathcal{P}$}
\end{equation}

With Dirichlet boundary condition:
\begin{equation}\label{PD}
\begin{cases}
\operatorname{div} A(x, u, Du) + B(x, u, Du) = 0, & x \in \Omega,\\
u = \phi, & \text{on } \partial\Omega,
\end{cases} \tag{$\mathcal{PD}$}
\end{equation}

or Neumann boundary condition:
\begin{equation}\label{PN}
\begin{cases}
\operatorname{div} A(x, u, Du) + B(x, u, Du) = 0, & x \in \Omega,\\
A(x, u, Du) \cdot \nu = C(x,u), & \text{on } \partial\Omega,
\end{cases} \tag{$\mathcal{PN}$}
\end{equation}
where $\nu(x)$ denotes the outward unit normal vector to $\partial\Omega$ at $x \in \partial\Omega$.

The functions
\[
A : \Omega \times \mathbb{R} \times \mathbb{R}^n \rightarrow \mathbb{R}^n,\quad
B : \Omega \times \mathbb{R} \times \mathbb{R}^n \rightarrow \mathbb{R},\quad
C : \partial\Omega \times \mathbb{R} \rightarrow \mathbb{R},\quad
\phi : \partial\Omega  \rightarrow \mathbb{R}
\]
 satisfy the following $G(x,t)$-structure conditions:

\textbf{Assumptions on $A$ :} $A = (A_1, A_2, \ldots , A_n) \in C(\overline{\Omega} \times \mathbb{R} \times \mathbb{R}^n, \mathbb{R}^n)$. For every $(x, u) \in \overline{\Omega} \times \mathbb{R}$, $A(x, u, \cdot) \in C^1(\mathbb{R}^n \setminus \{0\}, \mathbb{R}^n)$, and there exist a non-increasing continuous function $\lambda : [0, \infty) \rightarrow (0, \infty)$ and a non-decreasing continuous function $\Lambda : [0, \infty) \rightarrow (0, \infty)$ such that for all $x, x_1, x_2 \in \overline{\Omega}, u, u_1, u_2 \in \mathbb{R}, \eta \in \mathbb{R}^n \setminus \{0\}$ and $\xi = (\xi_1, \xi_2, \ldots , \xi_n) \in \mathbb{R}^n$, the following conditions are satisfied:

\begin{align}
\sum_{i,j=1}^n \frac{\partial A^i(x,u,\xi)}{\partial \xi_j} \eta_i \eta_j &\geq \lambda(|u|) \frac{g(x,|\xi|)}{|\xi|} |\eta|^2, \tag{$\mathcal{A}_1$}\label{6.31}\\
\sum_{i,j=1}^n \left| \frac{\partial A^i(x,u,\xi)}{\partial \xi_j} \right| &\leq \Lambda(|u|) \frac{g(x,|\xi|)}{|\xi|}, \tag{$\mathcal{A}_2$}\label{6.41}\\
A(x, u, 0) &=0,\tag{$\mathcal{A}_3$}\label{7.411}\\
|A(x_1, u_1, \eta)-A(x_2, u_2, \eta)| &\leq \Lambda\big(\max\{|u_1|,|u_2|\}\big)
\left(|x_1-x_2|^{\beta_1}+|u_1-u_2|^{\beta_2}\right)\tag{$\mathcal{A}_4$}\label{7.41}\\
&\qquad\times \left(1+\ln(e+|\eta|)\right)\left(1+\max\left\lbrace g(x_1,|\eta|),g(x_2,|\eta|)\right\rbrace\right),\notag
\end{align}
where $\beta_1,\beta_2\in (0,1)$. The logarithmic factor in \eqref{7.41} is needed for variable exponents, as in \cite{Fan2007}, and is harmless otherwise.

\textbf{Assumptions on $B$:} As in \cite{HBO}, let $H_1$ be a generalized \textnormal{N}-function with
\begin{align}
F_1^{-1} \leq H_1(x,1) \leq F_1, &\qquad \text{for all } x \in \overline{\Omega}, \tag{$\mathcal{HDN}_1$} \label{H111}\\
g^+ < h_1^- \leq \frac{h_1(x,t)\, t}{H_1(x,t)} \leq h_1^+\leq g_*^-:=\frac{ng^-}{n-g^-}, &\qquad \text{for all } x \in \overline{\Omega}\text{ and } t > 0, \tag{$\mathcal{HDN}_2$} \label{H211}
\end{align}
\begin{equation}
 H_1\mathop{\prec}_{\overline{\Omega}}  G^*, \tag{$\mathcal{HD}_1$}\label{YoungTriple1}
\end{equation}
\begin{equation}
\int_{\Omega} H_1(x,t) \, \mathrm{d}x < \infty,\ \text{for } t>0, \tag{$\mathcal{HD}_2$}\label{YoungTriple4}
\end{equation}
and assume, with $M(x,t):=\widetilde{H_1}^{-1}(x,G(x,t))$,
\begin{equation}\label{6.51}
|B(x, u, \eta)| \le b_1M(x,|\eta|)+b_2 h_1(x,|u|)+b_3, \quad \forall (x, u, \eta) \in \overline{\Omega} \times \mathbb{R} \times \mathbb{R}^n, \tag{$\mathcal{B}$}
\end{equation}
where $b_1,b_2,b_3>0$.

\textbf{Assumptions on $\phi$ (Dirichlet condition) :} For some $\beta_3\in (0,1)$, we assume that
\begin{equation}\label{H12}
\phi\in C^{1,\beta_3}(\partial\Omega), \tag{$\mathcal{H}$}
\end{equation}
and we fix $M_\phi\ge\|\phi\|_{C^{1,\beta_3}(\partial\Omega)}$. Since $\partial\Omega$ is of class $C^{1,\beta_3}$ in the Dirichlet problem (Theorem~\ref{Thm1.2}), $\phi$ admits an extension to $\mathbb{R}^n$, still denoted $\phi$, with $\|\phi\|_{W^{1,\infty}(\mathbb{R}^n)}\le M_\phi$, after enlarging $M_\phi$ by a constant depending only on the $C^{1,\beta_3}$-character of $\partial\Omega$; this is the extension used in Subsection~\ref{subsec:proof-bdry}.

\textbf{Assumptions on $C$ (Neumann condition):} $C(\cdot,\cdot) \in C(\partial\Omega \times \mathbb{R}, \mathbb{R})$ and there is a generalized \textnormal{N}-function $H_2$ with
\begin{align}
F_2^{-1} \leq H_2(x,1) \leq F_2, &\qquad \text{for all } x \in \partial\Omega, \tag{$\mathcal{HN}_1$} \label{NH111}\\
g^+ < h_2^- \leq \frac{h_2(x,t)\, t}{H_2(x,t)} \leq h_2^+\leq g_*^-:=\frac{ng^-}{n-g^-}, &\qquad \text{for all }x\in \partial\Omega  \text{ and } t>1,\tag{$\mathcal{HN}_2$} \label{NH211}
\end{align}
\begin{equation}
 H_1\mathop{\prec}_{\overline{\Omega}} G_*   \quad \text{and}\quad H_2\mathop{\prec}_{\partial\Omega}\left[G_*\right]^{\frac{n-1}{n}}, \tag{$\mathcal{HN}_3$}\label{NYoungTriple4}
\end{equation}
such that for all $x, x_1, x_2 \in \partial\Omega$ and $u, u_1, u_2 \in \mathbb{R}$,
\begin{align}
|C(x, u)| &\le c_1 h_2(x,|u|) + c_2, \tag{$\mathcal{C}$}\label{6.61}\\
|C(x_1, u_1) - C(x_2, u_2)| &\le \Lambda\left(\max\left\{|u_1|, |u_2|\right\}\right)\left(|x_1 - x_2|^{\beta_1} + |u_1 - u_2|^{\beta_2}\right),\tag{$\widetilde{\mathcal{C}}$}\label{C1}
\end{align}
where $c_1,c_2>0$ and $h_2(x,t)=\partial_t H_2(x,t)$. For the Neumann problem we also use the conditions \eqref{GG4}--\eqref{GG5} stated above.

\begin{definition}[bounded generalized solution]\label{def:sol}
\begin{enumerate}
    \item[$(1)$] $u \in W^{1,G(x,t)}(\Omega)$ is called a bounded generalized solution of problem \eqref{P} if $u \in L^\infty(\Omega)$ and
\begin{equation}\label{FDV}
\int_{\Omega} A(x, u, Du) \cdot Dv \, \mathrm{d}x = \int_{\Omega} B(x, u, Du) v \, \mathrm{d}x, \quad \forall v \in W_0^{1,G(x,t)}(\Omega). \tag{$\mathcal{FDV}$}
\end{equation}

\item[$(2)$] $u \in W^{1,G(x,t)}(\Omega)$ is called a bounded generalized solution of the boundary value problem \eqref{PD} if $u \in L^\infty(\Omega)$, $u-\phi \in W_0^{1,G(x,t)}(\Omega)$ and \eqref{FDV} holds.

\item[$(3)$] $u \in W^{1,G(x,t)}(\Omega)$ is called a bounded generalized solution of the boundary value problem \eqref{PN} if $u \in L^\infty(\Omega)$ and
\begin{equation}\label{FNV}
\int_{\Omega} A(x, u, Du) \cdot Dv \, \mathrm{d}x = \int_{\Omega} B(x, u, Du) v \, \mathrm{d}x + \int_{\partial \Omega} C(x, u) v \, \mathrm{d}\sigma, \quad \forall v \in W^{1,G(x,t)}(\Omega), \tag{$\mathcal{FNV}$}
\end{equation}
where $\mathrm{d}\sigma$ denotes the surface measure on $\partial\Omega$.
\end{enumerate}
\end{definition}

Our main results are Theorems \ref{Thm1.1}--\ref{Thm1.3}. In all of them we assume \eqref{6.31}--\eqref{7.41}, \eqref{H111}--\eqref{YoungTriple4}, \eqref{6.51} and \eqref{D22}--\eqref{GG6b}; in addition, \eqref{H12} is assumed for the Dirichlet problem, and \eqref{6.61}--\eqref{C1}, \eqref{NH111}--\eqref{NYoungTriple4} and \eqref{GG4}--\eqref{GG5} for the Neumann problem. We also assume that there exists a constant $M\ge1$ such that the bounded generalized solution $u$ mentioned in Theorems \ref{Thm1.1}--\ref{Thm1.3} satisfies
\begin{equation}\label{M}
\max_{\Omega} |u(x)| := \mathrm{ess\,sup}_{\Omega} |u(x)| \leq M. \tag{$\mathcal{M}$} 
\end{equation}

\begin{theorem}\label{Thm1.1}
Under assumptions \eqref{6.31}--\eqref{7.41}, \eqref{H111}--\eqref{YoungTriple4}, \eqref{6.51} and \eqref{D22}--\eqref{GG6b}, if $u \in W^{1,G(x,t)}(\Omega) \cap L^{\infty}(\Omega)$ is a bounded generalized solution of problem \eqref{P} satisfying \eqref{M}, then $u \in C^{1,\alpha}_{\mathrm{loc}}(\Omega)$. The Hölder exponent $\alpha$ depends only on $\mathrm{data}$, and for any $\Omega_0 \Subset \Omega$ the norm $\Vert u\Vert_{C^{1,\alpha}(\overline{\Omega_0})}$ depends only on $\mathrm{data}$ and on $\mathrm{dist}(\Omega_0, \partial\Omega)$.
\end{theorem}

\begin{theorem}\label{Thm1.2}
Under assumptions \eqref{6.31}--\eqref{7.41}, \eqref{H111}--\eqref{YoungTriple4}, \eqref{6.51}, \eqref{D22}--\eqref{GG6b} and \eqref{H12}, and assuming that the boundary $\partial\Omega$ is of class $C^{1,\beta_3}$, if $u \in W^{1,G(x,t)}(\Omega) \cap L^{\infty}(\Omega)$ is a bounded generalized solution of the boundary value problem \eqref{PD} satisfying \eqref{M}, then $u \in C^{1,\alpha}(\overline{\Omega})$. The exponent $\alpha$ depends only on $\mathrm{data}_D$, i.e. on $\mathrm{data}$, on $M_\phi\ge\|\phi\|_{C^{1,\beta_3}(\partial\Omega)}$ and on $\mathrm{data}_\Omega$ (the exponent $\beta_3$, the constants $m_0,\theta_0$ of \eqref{AO} and the $C^{1,\beta_3}$ character of $\partial\Omega$); the norm $\Vert u\Vert_{C^{1,\alpha}(\overline{\Omega})}$ depends in addition on $\Omega$.
\end{theorem}

\begin{theorem}\label{Thm1.3}
Assume that $\partial\Omega$ is of class $C^{1,\beta_3}$ and that the structural conditions \eqref{6.31}--\eqref{7.41}, \eqref{H111}--\eqref{YoungTriple4}, \eqref{6.51}, \eqref{NH111}--\eqref{NYoungTriple4}, \eqref{6.61}--\eqref{C1}, \eqref{D22}--\eqref{GG6b} and \eqref{GG4}--\eqref{GG5} are satisfied. If $u \in W^{1,G(x,t)}(\Omega) \cap L^{\infty}(\Omega)$ is a bounded generalized solution of the Neumann boundary value problem \eqref{PN} satisfying the bound \eqref{M}, then $u \in C^{1,\alpha}(\overline{\Omega})$.
Moreover, the Hölder exponent $\alpha$ depends only on $\mathrm{data}_N$, i.e. on $\mathrm{data}$, $\mathrm{data}_C$ and $\mathrm{data}_\Omega$; the norm $\|u\|_{C^{1,\alpha}(\overline{\Omega})}$ depends in addition on $\Omega$.
\end{theorem}

\begin{remark}
In Theorems \ref{Thm1.2} and \ref{Thm1.3} the Hölder exponent $\alpha$ depends on the domain only through $\mathrm{data}_\Omega$ (the exponent $\beta_3$, the constants $m_0,\theta_0$ of \eqref{AO} and the $C^{1,\beta_3}$ character of $\partial\Omega$): it is the same for all domains with the same $\mathrm{data}_\Omega$. The norm $\Vert u\Vert_{C^{1,\alpha}(\overline{\Omega})}$ depends on $\Omega$ itself.
\end{remark}

\begin{remark}\label{rem:sol}
The identities \eqref{FDV} and \eqref{FNV} are those of \cite[Definitions~4.1 and 4.1w]{HBO}; hence a bounded generalized solution is exactly a weak solution of \cite{HBO} which is bounded (see Definition~\ref{def:weak} below). Since $|u|\le M$, all the integrals converge: $A(x,u,Du)\cdot Dv\in L^1(\Omega)$ by \eqref{eqc1}, Young's inequality \eqref{Yi} and \eqref{L1}; $B(x,u,Du)\,v\in L^1(\Omega)$ by \eqref{6.51}, Young's inequality for $(H_1,\widetilde{H_1})$ and the embedding $W^{1,G(x,t)}(\Omega)\hookrightarrow L^{H_1(x,t)}(\Omega)$, which holds by \eqref{YoungTriple1}--\eqref{YoungTriple4} and Propositions~\ref{embb}, \ref{fembb}; and $C(x,u)\,v\in L^1(\partial\Omega)$ by \eqref{6.61}, \eqref{NYoungTriple4} and Proposition~\ref{ffembd}. Consequently, by truncation ($T_kv:=\max\{-k,\min\{v,k\}\}\to v$ in $W^{1,G(x,t)}(\Omega)$ as $k\to\infty$) and dominated convergence, it is equivalent to test \eqref{FDV} and \eqref{FNV} only with bounded functions $v\in W_0^{1,G(x,t)}(\Omega)\cap L^\infty(\Omega)$, respectively $v\in W^{1,G(x,t)}(\Omega)\cap L^\infty(\Omega)$; the proofs below use only such test functions.
\end{remark}

\begin{example}
\label{ex:typical_examples}

We present below several typical examples of operators $A$ satisfying
\eqref{6.31}--\eqref{7.41}, together with generalized $N$-functions
$G$ satisfying \eqref{D22}--\eqref{GG6b}. In all the examples,
$A(x,u,0)=0$. Moreover, conditions \eqref{6.31}--\eqref{6.41}
follow from the eigenvalue structure of $\partial_\xi A$.
Indeed, for an operator of the form
\[
A(\xi)=\frac{h(|\xi|)}{|\xi|}\xi,
\]
the eigenvalues of $\partial_\xi A$ are
$\frac{h(s)}{s}$, with multiplicity $n-1$, and $h'(s)$, where
$s=|\xi|$. By \eqref{D22diff}, these eigenvalues are comparable
to $\frac{g(x,s)}{s}$.

\begin{enumerate}

\item 
Let
\[
A(x,u,\eta)=|\eta|^{p-2}\eta.
\]
The corresponding function is
\[
G(x,t)=|t|^p,
\qquad 1<p<n.
\]
Condition \eqref{GG3} holds with
\[
\ell\equiv1,\quad \delta=e,\quad \eta=0,
\]
and, for $i=1,2$,
\[
a_i=a_i'=0,\quad
b_i=b_i'=\frac{1}{2},\quad
p_i=q_i=p_i'=q_i'=p.
\]

\item 
Let
\[
A(x,u,\eta)=\frac{g(|\eta|)}{|\eta|}\eta.
\]
 The corresponding function is $G(x,t)=G(t)$, which is independent of $x$, with $g\in C^1$, and
\[
g^--1\le \frac{tg'(t)}{g(t)}\le g^+-1,
\qquad g^+<n.
\]
Assume, in addition, that there exist constants
$1<p\le q<n$ and $C\ge1$ (we take $\bar a\ge C$ in \eqref{GG3}) such that
\[
C^{-1}(t^p+t^q)\le G(t)\le C(t^p+t^q),
\qquad t\ge1.
\]
Then \eqref{GG3} holds with
\[
\ell\equiv1,\qquad
b_1'=b_2'=C^{-1},\qquad
b_1=b_2=C,\qquad \Theta=C^2,
\]
and
\[
p_1=q_1=p_1'=q_1'=p,\qquad
p_2=q_2=p_2'=q_2'=q.
\]

It is worth noting that a general \textnormal{N}-function with
$g^-<g^+$ need not satisfy \eqref{GG3}. Indeed, for an $x$-independent $G$, conditions \eqref{GG3}(i) and \eqref{GG3}(iv) (with $R\to0$) force $G(t)$, $t\ge1$, to be comparable to a sum of at most two terms of the form $t^{p_i}$ or $t^{q_i}\ell(t)$; an N-function whose growth oscillates between two powers at arbitrarily large scales does not satisfy \eqref{GG3}.

\item 
Let
\[
A(x,u,\eta)=|\eta|^{p(x)-2}\eta.
\]
The corresponding function is
\[
G(x,t)=|t|^{p(x)},
\]
where
\[
1<p^-\le p(x)\le p^+<n,
\qquad
p\in C^{0,\gamma}(\overline\Omega),\qquad \gamma\in(0,1).
\]
In particular, the Hölder continuity of $p$ implies the
log-Hölder continuity required in \eqref{GG2}--\eqref{GG3}(iv).

To verify \eqref{7.41}, let $t=|\eta|$. If $t\ge1$, then
\[
\begin{aligned}
|t^{p(x_1)-1}-t^{p(x_2)-1}|
&\le |p(x_1)-p(x_2)|\,t^{\max\{p(x_1),p(x_2)\}-1}\ln t\\
&\le \frac{[p]_{0,\gamma}}{p^-}\,|x_1-x_2|^\gamma\ln t\,
\max\{g(x_1,t),g(x_2,t)\},
\end{aligned}
\]
where we used that $s\mapsto t^s$ is increasing for $t\ge1$ and that
\[
\max\{g(x_1,t),g(x_2,t)\}\ge p^-\,t^{\max\{p(x_1),p(x_2)\}-1}.
\]
For $0<t<1$, we similarly have
\[
|t^{p(x_1)-1}-t^{p(x_2)-1}|
\le C|p(x_1)-p(x_2)|.
\]
Consequently, \eqref{7.41} holds with $\beta_1=\gamma$.
Notice that log-Hölder continuity of $p$ alone does not
suffice to guarantee \eqref{7.41}.

Condition \eqref{GG3} holds as in Example 1, with
\[
p_i=q_i=p_i'=q_i'=p(x),\quad\text{for }i=1,2.
\]

\item 
Let
\[
A(x,u,\eta)
=a(x,u)|\eta|^{p(x)-2}\eta.
\]
The corresponding function is
\[
G(x,t)=|t|^{p(x)},
\]
where $p$ satisfies the assumptions of Example 3 and
$a$ is Hölder continuous with respect to $(x,u)$,
satisfies $a(x,u)\ge a_0>0$, and is bounded on
$\overline\Omega\times[-M,M]$.

Condition \eqref{7.41} follows by writing
\[
\begin{aligned}
A(x_1,u_1,\eta)-A(x_2,u_2,\eta)
&=\big(a(x_1,u_1)-a(x_2,u_2)\big)|\eta|^{p(x_1)-2}\eta\\
&\quad+a(x_2,u_2)\big(|\eta|^{p(x_1)-2}-|\eta|^{p(x_2)-2}\big)\eta,
\end{aligned}
\]
and applying the Hölder continuity of $a$ and the
estimates established in Example 3.

\item 
Let
\[
A(x,u,\eta)
=|\eta|^{p(x)-2}\ln(e+|\eta|)\eta.
\]
The corresponding function is
\[
G(x,t)=|t|^{p(x)}\ln(e+|t|),
\]
where $p$ satisfies the assumptions of Example 3 and
$p^++1<n$. In this case, we take
\[
\ell(t)=\ln(e+t),\qquad \delta=e,\qquad \eta=1,
\]
and, for $i=1,2$,
\[
b_i=b_i'=\frac{1}{2},\qquad
a_i=a_i'=0,\qquad
p_i=q_i=p(x).
\]
Conditions \eqref{D22} and \eqref{D22diff} hold with
\[
g^-=p^-,\quad g^+=p^++1,
\quad F=\ln(e+1).
\]
Condition \eqref{7.41} is verified as in Example 3: the additional factor $\ln(e+t)$ in $A$ is absorbed, because
\[
\max\{g(x_1,t),g(x_2,t)\}\ge p^-\,t^{\max\{p(x_1),p(x_2)\}-1}\ln(e+t);
\]
hence \eqref{7.41} holds with $\beta_1=\gamma$.

\item 
Let
\[
A(x,u,\eta)
=|\eta|^{p(x)-2}\eta
+|\eta|^{q(x)-2}\eta.
\]
The corresponding function is
\[
G(x,t)=|t|^{p(x)}+|t|^{q(x)},
\]
where $p,q\in C^{0,1}(\overline\Omega)$ and
\[
1<p(x)\le q(x)<n.
\]
Here,
\[
g^-=p^-,
\qquad g^+=q^+.
\]
Condition \eqref{GG3} holds with
\[
\ell\equiv1,\qquad
b_1=b_1'=b_2=b_2'=1,
\]
and
\[
p_1=q_1=p(x),\qquad
p_2=q_2=q(x).
\]

Furthermore, \eqref{GG2} follows from
\[
G^{-1}(x,t)\approx t^{\frac{1}{q(x)}}
\]
and the Lipschitz continuity of $q$.
We emphasize that the gap condition
\[
\frac{q^+}{p^-}<1+\frac{1}{n}
\]
is not needed for any of the assumptions \eqref{D22}--\eqref{GG5}: \eqref{GG4} holds because $G^{-1}(x,t)\approx t^{\frac{1}{p(x)}}$ for $t\le1$ and $p^+<n$, and \eqref{GG5} follows from the Lipschitz continuity of $p$ and $q$ (a direct computation gives $|\nabla_xG_*(x,t)|\lesssim G_*(x,t)\ln G_*(x,t)$ for large $t$).

\item 
Let
\[
A(x,u,\eta)
=|\eta|^{p-2}\eta
+a(x)|\eta|^{q-2}\eta.
\]
The corresponding function is
\[
G(x,t)=|t|^p+a(x)|t|^q,
\]
where
\[
1<p<q<n,\qquad
a\in C^{0,\alpha}(\overline\Omega),
\qquad
0<a_0\le a(x)\le\bar a,
\]
with $\bar a\ge1$.

Conditions \eqref{D22} and \eqref{D22diff} hold with
\[
g^-=p,\qquad g^+=q.
\]
Moreover, condition \eqref{GG2} follows from
\[
G(x,t)\le\frac{\bar a}{a_0}G(y,t),
\]
so that one can take
\[
\mu=\left(\frac{a_0}{\bar a}\right)^{\frac{1}{p}}.
\]

To verify \eqref{7.41}, we use the Hölder continuity of $a$:
\[
\begin{aligned}
|a(x_1)-a(x_2)||\eta|^{q-1}
&\le [a]_\alpha |x_1-x_2|^\alpha
|\eta|^{q-1}\le \frac{[a]_\alpha}{qa_0}
|x_1-x_2|^\alpha g(x_2,|\eta|).
\end{aligned}
\]
Thus, \eqref{7.41} holds with $\beta_1=\alpha$.

Finally, \eqref{GG3} is satisfied with
\[
\ell\equiv1,\qquad
p_1=q_1=p,\qquad
a_1=a_1'=b_1=b_1'=\frac{1}{2},
\]
and
\[
p_2=q_2=q,\qquad
a_2=a_2'=0,\qquad
b_2=b_2'=a(x).
\]
In this case,
\[
\beta_0=\min\left\{\frac{1}{2},a_0\right\},
\]
and the sum of the two pairs gives precisely $G$.
For the Neumann problem, one additionally needs
\eqref{GG4}--\eqref{GG5}; the latter, which involves $\nabla_xG_*$,
requires $a$ to be Lipschitz continuous.

\item 
Let
\[
A(x,u,\eta)
=|\eta|^{p-2}\eta
+a(x)|\eta|^{q-2}\ln(e+|\eta|)\eta.
\]
The corresponding function is
\[
G(x,t)=|t|^p+a(x)|t|^q\ln(e+|t|),
\]
where $a$ satisfies the assumptions of Example 7 and
\[
q+1<n.
\]
In this case,
\[
g^-=p,\qquad g^+=q+1.
\]

Condition \eqref{GG3} holds with
\[
\ell(t)=\ln(e+t).
\]
More precisely, for the first pair we take
\[
\begin{gathered}
a_1=1,\qquad a_1'=\frac{1}{2},\qquad
b_1=0,\qquad b_1'=\frac{a_0}{2},\\
p_1=q_1=p_1'=q_1'=p,
\end{gathered}
\]
whereas for the second pair we take
\[
\begin{gathered}
a_2=a_2'=0,\qquad
b_2=a(x),\qquad b_2'=\frac{a(x)}{2},\\
p_2=q_2=p_2'=q_2'=q.
\end{gathered}
\]
The corresponding constants are
\[
\Theta=2,\qquad \beta_0=\frac{a_0}{2}.
\]

For the lower bound, we use the inequality
\[
a_0t^p\ell(t)\le a(x)t^q\ell(t),
\quad t\ge1,
\]
which follows from $q>p$ and $a(x)\ge a_0$.
Consequently,
\[
\frac{1}{2}t^p+\frac{a_0}{2}t^p\ell(t)
+\frac{a(x)}{2}t^q\ell(t)
\le G(x,t),
\qquad t\ge1.
\]
Thus, \eqref{D22}--\eqref{GG3} hold. Moreover,
\eqref{7.41} follows from the Hölder continuity of $a$,
with $\beta_1=\alpha$.

\end{enumerate}
\end{example}

\begin{remark}
\label{rem:degenerate}

The assumptions of this paper exclude degenerate coefficients.
To see why, consider the double-phase operator
\[
A(x,\eta)=|\eta|^{p-2}\eta
+a(x)|\eta|^{q-2}\eta,
\qquad 1<p<q,
\]
where $a\in C^{0,\alpha}(\overline\Omega)$ is allowed to vanish.
Choose $x_1\in\overline\Omega$ with $a(x_1)=0$ (for instance
$a(x)=|x-x_1|^\alpha$ near $x_1$) and points $x_2\in\overline\Omega$ at
distance $\varrho:=|x_1-x_2|$ arbitrarily small, so that
\[
a(x_1)=0,
\qquad
a(x_2)=|x_1-x_2|^\alpha.
\]
Then
\[
|A(x_1,\eta)-A(x_2,\eta)|
=|x_1-x_2|^\alpha|\eta|^{q-1}.
\]
On the other hand, the right-hand side of \eqref{7.41} is, up to a
multiplicative constant,
\[
|x_1-x_2|^{\beta_1}
\bigl(1+|\eta|^{p-1}
+|x_1-x_2|^\alpha|\eta|^{q-1}\bigr)
\bigl(1+\ln(e+|\eta|)\bigr).
\]
Note that, for fixed $x_1\neq x_2$, letting $|\eta|\to\infty$ does \textit{not}
contradict \eqref{7.41}, because the factor $1+\ln(e+|\eta|)$ eventually
dominates $|x_1-x_2|^{\beta_1}$; the failure is non-uniform as $\varrho\to0$.
Indeed, take $t=|\eta|=\varrho^{-\frac{\alpha}{q-p}}$, so that $\varrho^\alpha t^{q-1}=t^{p-1}\ge1$ for $\varrho$ small.
Then $\max\{g(x_1,t),g(x_2,t)\}\le (p+q)\,\varrho^\alpha t^{q-1}$ and the
right-hand side of \eqref{7.41} is at most
\[
C\Lambda\, \varrho^{\beta_1}\bigl(1+\ln(e+t)\bigr)\,\varrho^\alpha t^{q-1},
\qquad\text{with}\qquad
\varrho^{\beta_1}\bigl(1+\ln(e+t)\bigr)=\varrho^{\beta_1}\Bigl(1+\ln\bigl(e+\varrho^{-\frac{\alpha}{q-p}}\bigr)\Bigr)\xrightarrow[\varrho\to0]{}0 ,
\]
whereas the left-hand side equals $\varrho^\alpha t^{q-1}$. Hence no constant
$\Lambda$ can satisfy \eqref{7.41} uniformly in $x_1,x_2$, whatever $\beta_1>0$.
Hence, establishing $C^{1,\alpha}$-regularity for
degenerate double-phase equations requires additional structural
conditions, notably the Colombo--Mingione balance condition
\[
\frac{q}{p}\le 1+\frac{\alpha}{n},
\]
together with a different regularity approach. Such problems are
beyond the scope of the present work.
\end{remark}

\subsection{Paper Structure}

Section 2 introduces the necessary background on Musielak-Orlicz and Musielak-Orlicz-Sobolev spaces, including key embeddings and preliminary results. Section 3 contains the proofs of the main results. Subsection~\ref{subsec:tools} collects two lemmas used in all three proofs (a Morrey-type bootstrap and the exponent bookkeeping of the Campanato step); Subsection~\ref{subsec:proof1.1} is devoted to interior $C^{1,\alpha}$ regularity (Theorem \ref{Thm1.1}), while Subsection~\ref{subsec:proof-bdry} establishes boundary regularity for the Dirichlet (Theorem \ref{Thm1.2}) and Neumann (Theorem \ref{Thm1.3}) problems, which are treated together after the boundary has been flattened. The proofs combine modern techniques in regularity theory with careful adaptations to handle the generality of the Musielak-Orlicz framework.

\section{Notations and Preliminaries}\label{sec:prelim}
Here, we introduce the fundamental notations  and definitions used throughout the paper. A central concept is the Musielak-Orlicz-Sobolev space, whose properties are essential for the results that follow.
\subsection{Notations}\label{subsec:notation}
The following notations will be used throughout this paper:
\begin{itemize}
    \item $\mathbb{R}^n$ denotes the $n$-dimensional Euclidean space.
    \item $\Omega \subset \mathbb{R}^n$ is a bounded domain; as in \cite{HBO}, whenever a boundary condition is involved $\Omega$ is a bounded Lipschitz domain (of class $C^{1,\beta_3}$ in Theorems~\ref{Thm1.2} and \ref{Thm1.3}).
    \item The notation \(\Omega_0 \Subset \Omega\) denotes an open subset whose
closure is compact and contained in \(\Omega\).

    \item For a measurable set $E \subset \mathbb{R}^n$, $\mathrm{mes}(E)$ or $|E|$ denotes its $n$-dimensional Lebesgue measure.
\item \(\fint_{E} u\, \mathrm{d}x = \frac{1}{|E|} \int_E u\, \mathrm{d}x\).
\item  $\mathbb{R}^n = \mathbb{R}^{n-1} \times \mathbb{R}$.  
\item $\mathbb{R}^n_+ = \left\{(x', x'') \in \mathbb{R}^{n-1} \times \mathbb{R}: x'' > 0\right\}$ is the upper half-space and  $\mathbb{R}_0^n = \mathbb{R}^{n-1} \times \{0\}$ is the boundary hyperplane.
    \item If $u$ is a measurable function defined on $\Omega$, and $E \subset \Omega$ is measurable, we define
    \begin{equation}\label{os1}
    \max_E u := \mathop{\mathrm{ess\,sup}}_{x \in E} u(x), \quad
    \min_E u := \mathop{\mathrm{ess\,inf}}_{x \in E} u(x), \quad
    \mathrm{osc}_E u := \max_E u - \min_E u.
    \end{equation}
    \item For any $x_0 \in \mathbb{R}^n$ and $\rho > 0$, we define, as in \cite{HBO}, the open ball, its trace on $\Omega$ and on $\partial\Omega$,
    \[
    B_\rho(x_0) := \{x \in \mathbb{R}^n : |x - x_0| < \rho\},\qquad \Omega_\rho(x_0) := B_\rho(x_0)\cap \Omega, \qquad (\partial \Omega)_\rho(x_0) := \partial \Omega \cap B_\rho(x_0),
    \]
and the cube
   \[
Q(x_0, \rho):=\{x = (x^1, x^2, \ldots, x^n) \in \mathbb{R}^n: |x^i - x_0^i| < \rho,\ \forall i\}.
   \] 
    We may denote them simply by $B_\rho$, $\Omega_\rho$, $(\partial\Omega)_\rho$ and $Q_\rho$ when the center is clear from context; balls $B_{\rho_1}$ and $B_{\rho_2}$ with the same symbol for the center are assumed to be concentric. We denote by $\omega_n := |B_1|$ the measure of the unit ball in $\mathbb{R}^n$.
    \item For $0 < \alpha \leq 1$, we denote:
    \[
    C^{0,\alpha}(\overline{\Omega}) := \left\{ u : u \text{ is Hölder continuous on } \overline{\Omega} \text{ with exponent } \alpha \right\},
    \]
    with norm
    \[\|u\|_{C^{0,\alpha}(\overline{\Omega})} := \sup_{x \in \overline{\Omega}} |u(x)| 
\;+\; \sup_{\substack{x,y \in \overline{\Omega} \\ x \neq y}} \frac{|u(x) - u(y)|}{|x-y|^\alpha}\]
    \[
    C_{\text{loc}}^{0,\alpha}(\Omega) := \left\{ u : u \in C^{0,\alpha}(\overline{\Omega'}) \text{ for all } \Omega' \Subset \Omega \right\}
    \]
    (this space is denoted $C^{0,\alpha}(\Omega)$ in \cite{HBO}),
    \[
C^{1,\alpha}(\overline{\Omega}) := \left\{ u : u \in C^1(\overline{\Omega}), \text{ and } D u \text{ is H\"older continuous on } \overline{\Omega} \text{ with exponent } \alpha \right\},
\]
with norm

\[\|u\|_{C^{1,\alpha}(\overline{\Omega})} := 
\sup_{\overline{\Omega}} |u| 
+ \sup_{\overline{\Omega}} |\nabla u| 
+ \sup_{x \neq y} \frac{|\nabla u(x) - \nabla u(y)|}{|x-y|^\alpha}.\]
\[
C_{\text{loc}}^{1,\alpha}(\Omega) := \left\{ u : \Omega \to \mathbb{R} \ \middle| \ \forall x \in \Omega, \ \exists \text{ neighborhood } U_x \text{ of } x \text{ such that } u \in C^{1,\alpha}(U_x) \right\},
\]
\[
C(\overline{\Omega}) := \left\{ u : \overline{\Omega} \to \mathbb{R} \ \middle| \ u \text{ is continuous on } \overline{\Omega} \right\},
\]
\[
C^\infty(\overline{\Omega}) := \left\{ u : \overline{\Omega} \to \mathbb{R} \ \middle| \ u \text{ has continuous partial derivatives of all orders on } \overline{\Omega} \right\},
\]
\[
C_0^\infty(\Omega) := \left\{ u \in C^\infty(\Omega) \ \middle| \ \mathrm{supp}(u) \text{ is compact and } \mathrm{supp}(u) \subset \Omega \right\},
\]
 \[
L^\infty(\Omega) := \left\{ u : \Omega \to \mathbb{R} \ \middle| \ \exists M > 0 \text{ such that } |u(x)| \leq M \text{ for almost every } x \in \Omega \right\}.
\]
with norm 
\[
\|u\|_{L^\infty(\Omega)} = \operatorname{ess\,sup}_{x \in \Omega} |u(x)|,
\]
and 
\[
W^{1,\infty}(\mathbb{R}^n) = \left\{ u \in L^\infty(\mathbb{R}^n) \;:\; D u \in L^\infty(\mathbb{R}^n; \mathbb{R}^n) \text{ in the weak sense} \right\}
\]
with norm
\[
\|u\|_{W^{1,\infty}(\mathbb{R}^n)} = \|u\|_{L^\infty(\mathbb{R}^n)} + \|D u\|_{L^\infty(\mathbb{R}^n)}.
\]
    \item For any function \( u : \Omega \rightarrow \mathbb{R} \), we define the positive part and the negative part of \( u \) as:
\[
u^+ = \max\{u, 0\}, \quad u^- = \max\{-u, 0\}.
\]

\end{itemize}
We also adopt the notation of the Preliminaries of \cite{HBO} for the quantities derived from \eqref{GG3}:
\[
L:=e^{L_{0}},\qquad d_{*}:=\frac{d}{4},\qquad d_{\ell}:=d_{*}\min\{1,\ell(1)\},\qquad \omega(R):=\frac{L_0}{|\ln R|}\ \ (R<R_0),
\]
\[
H_{*}:=\max_{i}\max\left\{[a_{i}^{\prime}]_{0,\bar\alpha},[b_{i}^{\prime}]_{0,\bar\alpha}\right\},\quad
[b]_{*}:=\max_{i}[b_{i}^{\prime}]_{0,\bar\alpha},\quad
\mu_{*}:=\max\left\{q^{+}+1,p^{+}+\bar\alpha\right\},
\]
\[
R_{*}:=\min\left\{R_{0},\frac{1}{4M},\frac{1}{2}\left(\frac{d_{*}}{H_{*}}\right)^{\frac{1}{\bar\alpha}}\right\},
\]
where entries with $H_{*}=0$ are read as $1$, and, since only the alternative \eqref{GG6b} is used here, $R_{*}$ is further replaced by $\min\left\{R_{*},\left(\frac{\beta_{0}}{4H_{*}}\right)^{\frac{1}{\bar\alpha}}\right\}$. Since $d\le a_{i}^{\prime}+b_{i}^{\prime}\le 2\bar{a}$, we have $d_{*}\le \frac{\bar{a}}{2}$. All these quantities are functions of $\mathrm{data}_G$ and $M$ (see the paragraph ``Dependence of constants'' in Subsection~\ref{subsec:main}), so any constant depending on them depends only on $\mathrm{data}_G$ and $M$. We shall also use the elementary inequality
\begin{equation}\label{elem:log}
\ell(\lambda t)\leq \lambda\ell(t),\qquad \lambda\geq 1,\ t\geq 0,
\end{equation}
which follows from $\log(\delta+\lambda s)\leq \lambda\log(\delta+s)$ with $s=\eta t$. The radius $R_*$ is the one of \cite{HBO}; it must not be confused with the radii $r_*$, $R_1$, $R_2,\dots$ fixed in the proofs below.
\subsection{Dependence of constants}
We use the following groups of parameters; they are the quantities on which the constants of \cite{HBO} depend, once the structure conditions on $A$ are translated as in Remark~\ref{rem:notation-HBO}(b) below.
\begin{itemize}
\item $\mathrm{data}_G$ collects all the quantities of \eqref{D22}, \eqref{GG1}, \eqref{GG2}, \eqref{GG3} and \eqref{GG6b}:
\[
\begin{aligned}
\mathrm{data}_G:={}&\{n,\,g^\pm,\,F,\,\mu,\,\bar a,\,\Theta,\,\beta_0,\,\ell(1),\,d,\,p^\pm,\,q^\pm,\\
&\ \bar\alpha,\,L_0,\,R_0,\,[a_i']_{0,\bar\alpha},\,[b_i']_{0,\bar\alpha}\}.
\end{aligned}
\]
Note that $L=e^{L_0}$, and that $d_*,d_\ell,H_*,[b]_*,\mu_*,R_*$ of Subsection~\ref{subsec:notation} are functions of $\mathrm{data}_G$ and $M$.
\item $\mathrm{data}_A:=\{\lambda_0=\lambda(M),\,\Lambda_0=\Lambda(M),\,\beta_1,\,\beta_2\}$, from \eqref{6.31}--\eqref{7.41}. By Proposition~\ref{prop1.112}, $\lambda_1=\lambda_0$ and $\Lambda_1=\frac{\Lambda_0}{g^--1}$; in the language of \cite{HBO} (Remark~\ref{rem:notation-HBO}(b)) $a_4=\lambda_0$, $a_1=\frac{\Lambda_0}{g^--1}$ and $a_2=a_3=a_5=a_6=1$, so that no further constant is needed.
\item $\mathrm{data}_B:=\{b_1,b_2,b_3,\,h_1^\pm,\,F_1\}$, from \eqref{H111}, \eqref{H211} and \eqref{6.51}.
\item $\mathrm{data}_C:=\{h_2^\pm,\,F_2,\,c_1,\,c_2,\,\widetilde M,\,\delta_0,\,C_0,\,t_0\}$, only for the Neumann problem: $h_2^\pm,F_2$ come from \eqref{NH111}--\eqref{NH211}, $c_1,c_2$ from \eqref{6.61}, $\widetilde M\ge\sup|C(\partial\Omega\times[-M,M])|$, and $\delta_0,C_0,t_0$ from \eqref{GG5}.
\item $\mathrm{data}_\Omega:=\{\beta_3,\,m_0,\,\theta_0,\,r_\Omega,\,K_\Omega\}$, only for the boundary problems: $\beta_3$ is the regularity exponent of $\partial\Omega$, $m_0,\theta_0$ are the constants of \eqref{AO}, and $r_\Omega,K_\Omega$ describe the $C^{1,\beta_3}$ character of $\partial\Omega$ (for every $z\in\partial\Omega$ there are an open set $U_z\supset B_{r_\Omega}(z)$ and a $C^{1,\beta_3}$-diffeomorphism $\Psi_z$ of $U_z$ flattening $\partial\Omega$ near $z$, with $\Psi_z(z)=0$, $\Psi_z(\Omega\cap U_z)\cap Q(0,2r_\Omega)=Q^+(0,2r_\Omega)$ and $\|\Psi_z\|_{C^{1,\beta_3}}+\|\Psi_z^{-1}\|_{C^{1,\beta_3}}\le K_\Omega$; see Subsection~\ref{subsec:proof-bdry}).
\end{itemize}
We put
\[
\begin{aligned}
\mathrm{data}&:=\{\mathrm{data}_G,\mathrm{data}_A,\mathrm{data}_B,M\},\\
\mathrm{data}_D&:=\{\mathrm{data},\mathrm{data}_\Omega,M_\phi\},\\
\mathrm{data}_N&:=\{\mathrm{data},\mathrm{data}_C,\mathrm{data}_\Omega\}.
\end{aligned}
\]
The exponent $\alpha_1$ of Proposition~\ref{prop1.1} depends only on $\mathrm{data}$ in the interior case, and only on $\mathrm{data}_D$, respectively $\mathrm{data}_N$, in the boundary cases (more precisely, on the quantities listed in Proposition~\ref{prop1.1}; $\beta_1,\beta_2$ are not used in \cite{HBO}), while $\|u\|_{C^{0,\alpha_1}}$ depends in addition on $\mathrm{dist}(\Omega_0,\partial\Omega)$, respectively on $\Omega$. Every quantity built from $\alpha_1$, in particular $\beta=\min\{\beta_1,\alpha_1\beta_2\}$ and the exponent $\alpha$ below, depends on the same parameters as $\alpha_1$.
\subsection{Comments on the assumptions} Remark~\ref{rem:notation-HBO} compares our notation and assumptions with those of \cite{HBO}; Remark~\ref{rem:sol} (above) shows that bounded generalized solutions can be tested with bounded functions only; Example~\ref{ex:typical_examples} lists operators and growth functions satisfying all the assumptions; Remark~\ref{rem:degenerate} explains why degenerate coefficients are excluded.

\begin{remark}
Critical growth is allowed in \eqref{H211} and \eqref{NH211}. For instance, if $G(x,t)=\frac{t^p}{p}$ (so that $g^-=g^+=p<n$), one may take $H_1(x,t)=t^{p^*}$ with $p^*=\frac{np}{n-p}=g_*^-$, which is critical for the embedding $W^{1,p}\hookrightarrow L^{p^*}$, and $H_2(x,t)=t^{\frac{(n-1)p}{n-p}}$, which is critical for the trace embedding $W^{1,p}(\Omega)\hookrightarrow L^{\frac{(n-1)p}{n-p}}(\partial\Omega)$; indeed $p<\frac{(n-1)p}{n-p}=\frac{n-1}{n}\,p^*\le g_*^-$. For a variable $G$, the requirement $h_i^+\le g_*^-:=\frac{ng^-}{n-g^-}$ keeps the growth of $H_i$ below the critical growth attached to the \textit{smallest} exponent $g^-$ of $G$.
\end{remark}

\begin{remark}[Notation and comparison with \cite{HBO}]\label{rem:notation-HBO}
We use the notation of \cite{HBO} (see also Subsection~\ref{subsec:notation}), to which most of the results quoted below refer.
\begin{enumerate}[(a)]
\item \textit{Conditions on $G$, $H_1$, $H_2$.} Conditions (G0)--(G5) are those of \cite[Subsection~1.2]{HBO}, with the alternative (G3b) in (G3)(vi); the conditions $(\mathcal{HDN}_1)$, $(\mathcal{HDN}_2)$, $(\mathcal{HD}_1)$, $(\mathcal{HD}_2)$ on $H_1$ and $(\mathcal{HN}_1)$, $(\mathcal{HN}_2)$, $(\mathcal{HN}_3)$ on $H_2$ are those of \cite{HBO} with the same names, and $(\mathcal{B})$, $(\mathcal{C})$ are the growth conditions \eqref{6.51}, \eqref{6.61} of \cite{HBO}. The only additional requirement on $G$ is the differential form (G0$'$) of (G0).
\item \textit{Conditions on $A$.} The conditions $(\mathcal{A}_1)$--$(\mathcal{A}_4)$ above are of differential (Lieberman) type; they are not the algebraic conditions of the same names in \cite{HBO}, which we denote here $(\mathcal{A}_1^{\mathrm{HBO}})$, $(\mathcal{A}_2^{\mathrm{HBO}})$:
\begin{align}
A(x,u,\eta)\cdot\eta&\ge a_4\,G(x,|\eta|)-a_5\,H_1(x,|u|)-a_6, \tag{$\mathcal{A}_1^{\mathrm{HBO}}$}\label{A1HBO}\\
|A(x,u,\eta)|&\le a_1\,g(x,|\eta|)+a_2\,P(x,|u|)+a_3, \tag{$\mathcal{A}_2^{\mathrm{HBO}}$}\label{A2HBO}
\end{align}
where $a_1,\dots,a_6>0$ and $P(x,t):=\widetilde{G}^{-1}(x,H_1(x,t))$. If $|u|\le M$, Proposition~\ref{prop1.112} gives \eqref{A1HBO}--\eqref{A2HBO} with $a_4=\lambda_1$, $a_1=\Lambda_1$ and $a_2=a_3=a_5=a_6=1$ (the terms carrying $a_2,a_3,a_5,a_6$ have the favourable sign and are only added because \cite{HBO} asks for positive constants). Moreover, since $|u|\le M$ a.e., replacing $u$ by $T_Mu:=\max\{-M,\min\{u,M\}\}$ in the arguments of $A$, $B$, $C$ does not change the equations satisfied by $u$ and makes all the structure conditions valid for every $u\in\mathbb{R}$, with $\lambda(|u|)$, $\Lambda(|u|)$ replaced by $\lambda_0$, $\Lambda_0$; this is how the results of \cite{HBO} are applied. Conditions \eqref{7.411}--\eqref{7.41} are not used in \cite{HBO}.
\item \textit{Balls, spaces, exponents.} $B_\rho(x_0)$, $\Omega_\rho(x_0)=B_\rho(x_0)\cap\Omega$ and $(\partial\Omega)_\rho(x_0)=\partial\Omega\cap B_\rho(x_0)$ are as in \cite{HBO}; the space denoted $C^{0,\alpha}(\Omega)$ there is denoted $C^{0,\alpha}_{\mathrm{loc}}(\Omega)$ here, and $\mathrm{d}\sigma$ is the surface measure. The Hölder exponent of $u$ given by \cite{HBO} (the exponent $\alpha$ of \cite[Theorems~2.1, 2.2, 4.3, 4.4]{HBO}) is denoted $\alpha_1$ here, while the symbol $\alpha_1$ of \cite[Theorem~2.2]{HBO} (the exponent of the trace $u|_{\partial\Omega}$) equals $1$ in our setting.
\item \textit{Solutions.} \eqref{FDV} and \eqref{FNV} are the weak formulations of \cite[Definitions~4.1 and 4.1w]{HBO} (see Definition~\ref{def:sol} and Remark~\ref{rem:sol}).
\item \textit{Constants.} To avoid clashes with the symbols of \cite{HBO}: $\delta_0$, $C_0$, $t_0$ always denote the constants of \eqref{GG5}; the exponents of Gehring's higher integrability (Proposition~\ref{prop:Gehring}) are denoted $\delta_G$, $\delta_*$ and $\delta_{\mathrm{hi}}$; the exponent of Lemma~\ref{lem:B} is $\theta_B$; $m_0$, $\theta_0$ always refer to condition \eqref{AO}; and the radius $R_*$ of Subsection~\ref{subsec:notation} is different from the radii $r_*,R_\omega,R_1,R_2,\dots$ chosen in the proofs below.
\end{enumerate}
\end{remark}

\subsection{Musielak-Orlicz-Sobolev space \texorpdfstring{$W^{1,G(x,t)}$}{W1G(x,t)}}\label{sec2}
In this subsection, we recall some definitions and fundamental properties of Musielak-Orlicz and Musielak-Orlicz-Sobolev spaces. 
For a comprehensive bibliography on Musielak-Orlicz-Sobolev spaces, we refer the reader to \cite{Ala2, TD, Chlebicka2021, Diening2011, Harjulehto2019, Fan2012a, Fan2012b, Musielak1983, TS1,  TS4, TS6, TS7, TS10}.

\begin{definition}
Let $\Omega$ be an open subset of $\mathbb{R}^n$. A function $G:\Omega \times \mathbb{R} \to \mathbb{R}$ is called a generalized \textnormal{N}-function if it satisfies the following conditions:
\begin{enumerate}
    \item[(1)] For a.e. $x\in \Omega$, the function $G(x,t)$ is even, continuous, strictly increasing, and convex in $t$, and for each $t\in \mathbb{R}$, $G(x,t)$ is measurable in $x$;
    \item[(2)] $\displaystyle\lim_{t \to 0} \frac{G(x,t)}{t} = 0$, for a.e. $x\in \Omega$;
    \item[(3)] $\displaystyle\lim_{t \to \infty} \frac{G(x,t)}{t} = \infty$, for a.e. $x\in \Omega$;
    \item[(4)] $G(x,t) > 0$ for all $t > 0$ and all $x\in \Omega$, and $G(x,0) = 0$ for all $x\in \Omega$.
\end{enumerate}
\end{definition}

\begin{remark}[Remark~2.1 of \cite{Ala1}]\label{rem:intrep}
We give an equivalent definition of a generalized \textnormal{N}-function that admits an integral representation. For $x \in \Omega$ and $t \geq 0$, let $g(x,t)$ be the right-hand derivative of $G(x,\cdot)$ at $t$, and define $g(x,t) = -g(x,-t)$ for $t < 0$. Then for each $x \in \Omega$, the function $g(x,\cdot)$ is odd, real-valued, satisfies $g(x,0) = 0$ and $g(x,t) > 0$ for $t > 0$, is right-continuous and nondecreasing on $[0, +\infty)$, and
\[
G(x,t) = \int_0^{|t|} g(x,s)\, \mathrm{d}s, \quad \text{for all } x\in \Omega, \ t\in \mathbb{R}.
\]
\end{remark}

\begin{definition}\label{definv}
Let $G:\Omega\times[0,\infty)\to[0,\infty)$ be a generalized $\mathrm{N}$-function.
Since for a.e.\ $x\in\Omega$ the mapping $t\mapsto G(x,t)$ is continuous and
strictly increasing on $[0,\infty)$, we can define $G^{-1}(x,\cdot)$ as the
continuous inverse of $G(x,\cdot)$. In particular, it holds that
\[
G^{-1}(x,G(x,t)) = G(x,G^{-1}(x,t)) = t,
\quad \text{for a.e.\ } x\in\Omega \text{ and all } t\ge 0.
\]
\end{definition}
\begin{definition}
A generalized \textnormal{N}-function $G$ satisfies the \textit{$\Delta_2$-condition} if there exist $C_0 > 0$ and a nonnegative function $\varphi \in L^1(\Omega)$ such that
\[
G(x,2t) \leq C_0 G(x,t) + \varphi(x), \quad \text{for a.e. } x\in \Omega \text{ and all } t \geq 0.
\]
\end{definition}

\begin{definition}\label{def:A012}
A generalized \textnormal{N}-function $G(x,t)$ is said to satisfy:
\begin{enumerate}
    \item[\textnormal{(A0)}] if there exists $\mu \in (0,1]$ such that
    \[
    \mu \leq G^{-1}(x,1) \leq \frac{1}{\mu}, \quad \text{for a.e. } x \in \mathbb{R}^n;
    \]
    \item[\textnormal{(A1)}] if there exists $\mu \in (0,1]$ such that
    \[
    \mu G^{-1}(x,t) \leq G^{-1}(y,t),
    \]
    for every $t \in \left[1, \frac{1}{|B|}\right]$, for a.e. $x,y \in B$, and every ball $B \subset \mathbb{R}^n$ with $|B| \leq 1$;
    \item[\textnormal{(A2)}] if for every $s > 0$ there exist $\mu \in (0,1]$ and $\varphi \in L^1(\mathbb{R}^n) \cap L^\infty(\mathbb{R}^n)$ such that
    \[
    \mu G^{-1}(x,t) \leq G^{-1}(y,t),
    \]
    for a.e. $x,y \in \mathbb{R}^n$ and for all $t \in [\varphi(x)+\varphi(y), s]$.
\end{enumerate}
\end{definition}

\begin{definition}\label{prec}
Let $G_1(x,t)$ and $G_2(x,t)$ be two generalized \textnormal{N}-functions.
\begin{itemize}
    \item[(1)] We say that $G_1(x,t)$ increases essentially slower than $G_2(x,t)$ near infinity, and we write $G_1 \prec\prec G_2$, if for any $k > 0$
    \[
    \lim_{t \to \infty} \frac{G_1(x, kt)}{G_2(x,t)} = 0, \quad \text{uniformly in } x \in \Omega.
    \]
    \item[(2)] We say that $G_1(x,t)$ is $\displaystyle{\mathop{\textit{weaker}}_{\Omega}}$ than $G_2(x,t)$, denoted by $\displaystyle{G_1 \mathop{\prec}_{\Omega} G_2}$, if there exist constants $C_1, C_2 > 0$ and a nonnegative function $\varphi \in L^1(\Omega)$ such that
    \[
    G_1(x,t) \leq C_1 G_2(x, C_2 t) + \varphi(x), \quad \text{for a.e. } x \in \Omega \text{ and all } t \geq 0.
    \]
\end{itemize}
\end{definition}

\begin{definition}[Complementary function]\label{CF}
Let $G:\Omega\times[0,\infty)\to[0,\infty)$ be a generalized $\mathrm{N}$-function.
The function $\widetilde{G}:\Omega\times[0,\infty)\to[0,\infty)$ defined by
\begin{equation}\label{Cf}
\widetilde{G}(x,t)
:= \sup_{s \ge 0} \bigl\{ ts - G(x,s) \bigr\},
\quad \text{for a.e.\ } x\in\Omega \text{ and all } t\ge 0,
\end{equation}
is called the \textit{complementary function} (or \textit{conjugate function}) of $G$.
Moreover, $\widetilde{G}$ is also a generalized $\mathrm{N}$-function.
\end{definition}

\begin{remark}
From the definition of the complementary function $\widetilde{G}(x,t)$, we derive the following Young-type inequality:
\begin{equation}\label{Yi}
st \leq G(x,s) + \widetilde{G}(x,t), \quad \text{for all } x \in \Omega, \ s,t \geq 0.
\end{equation}
\end{remark}

The next lemma is taken from Bahrouni--Bahrouni--Missaoui \cite[Lemma~2.3]{Ala1}.

\begin{lemma}\label{lm1}
Let $G(x,t)$ be a generalized \textnormal{N}-function that satisfies \eqref{D22}. Suppose that $t \mapsto g(x,t)$ is continuous and increasing on $\mathbb{R}$ for a.e. $x \in \Omega$. Then
\begin{align}
\widetilde{G}(x, g(x,s)) &\leq (g^+ - 1) G(x,s), \quad \text{for all } s \geq 0, \ x \in \Omega, \label{L1} \end{align}
and
\begin{align}
\frac{g^+}{g^+ - 1} =: \widetilde{g^-} &\leq \frac{\widetilde{g}(x,s)s}{\widetilde{G}(x,s)} \leq \widetilde{g^+} := \frac{g^-}{g^- - 1}, \quad \text{for all } x \in \Omega, \ s > 0, \label{D3}
\end{align}
where $\displaystyle \widetilde{G}(x,s) = \int_0^s \widetilde{g}(x,t)\, \mathrm{d}t$.
\end{lemma}

\begin{remark}\label{lm1-hyp}
Under \eqref{D22diff}, which is assumed throughout, $g(x,\cdot)\in C^1((0,\infty))$ and $t\,\partial_tg(x,t)\ge(g^--1)\,g(x,t)>0$, so that $g(x,\cdot)$ is continuous and increasing and the hypothesis of Lemma~\ref{lm1} is satisfied.
\end{remark}

\begin{remark}\label{compl}
The condition \eqref{D22} implies that $G(x,t)$ and its complementary function $\widetilde{G}(x,t)$ satisfy the $\Delta_2$-condition.
\end{remark}
 Now, we  define the Musielak-Orlicz space as follows:
\[
L^{G(x,t)}(\Omega):=\left\lbrace u:\Omega\longrightarrow \mathbb{R}\ \text{measurable :}\ \rho_{G}(\lambda u)<+\infty,\ \ \text{for some}\ \lambda>0\right\rbrace,
\]
where
\begin{equation}\label{Mo}
  \rho_{G}(u):= \int_{\Omega}G(x, u)\, \mathrm{d}x.
\end{equation}
The space $L^{G(x,t)}(\Omega)$ is endowed with the Luxemburg norm
\begin{equation}\label{No}
  \Vert u\Vert_{L^{G(x,t)}(\Omega)}:=\inf\left\lbrace \lambda>0:\ \rho_{G}\left(x,\frac{u}{\lambda}\right)\leq 1\right\rbrace.
\end{equation}
\begin{proposition}\label{HM}
    Let $G$ be a generalized \textnormal{N}-function satisfies the $\Delta_2$-condition, then
\[
L^{G(x,t)}(\Omega) =\left\lbrace  u: \Omega \longrightarrow \mathbb{R}\ \text{measurable :}\ \rho_{G}( u)<+\infty\right\rbrace.
\]
\end{proposition}
\begin{proposition} \label{zoo}
     Let $G$ be a generalized \textnormal{N}-function satisfies \eqref{D22}, then the following assertions hold:
     \begin{itemize}
    \item [(1)]
    \[
    \min \left\{s^{g^-}, s^{g^+}\right\}G(x,t)\leq  G(x,s t)\leq \max \left\{s^{g^-}, s^{g^+}\right\}G(x,t),\quad\text{for a.e. } x \in \Omega\text{ and all }  s, \ t \geq 0.
    \]
    \item [(2)] if in addition \eqref{GG1} holds:
    \[
    \min \left\{s^{g^-}, s^{g^+}\right\}F^{-1}\leq  G(x,s )\leq \max \left\{s^{g^-}, s^{g^+}\right\}F,\quad\text{for a.e. } x \in \Omega\text{ and all }  s \geq 0.
    \]
     \item [(3)]
     \[
     \min \left\{s^{\widetilde{g^-}}, s^{\widetilde{g^+}}\right\}\widetilde{G}(x,t)\leq  \widetilde{G}(x,s t)\leq \max \left\{s^{\widetilde{g^-}}, s^{\widetilde{g^+}}\right\}\widetilde{G}(x,t),\quad\text{for a.e. } x \in \Omega\text{ and all }  s, \ t \geq 0.
     \]
 \item [(4)]  for all $u\in L^{G(x,t)}(\Omega)$,
 \[
 \min \left\{\|u\|_{L^{G(x,t)}(\Omega)}^{g^-},\|u\|_{L^{G(x,t)}(\Omega)}^{g^+}\right\} \leq \rho_{G}(u) \leq\max \left\{\|u\|_{L^{G(x,t)}(\Omega)}^{g^-},\|u\|_{L^{G(x,t)}(\Omega)}^{g^+}\right\};
 \]
\end{itemize}
\end{proposition}

\begin{lemma}\label{lem:elem}
Assume \eqref{D22} and \eqref{GG1}. Then there exists a constant
$C=C(g^\pm,F)>0$ such that, for every $x$ and every $t\ge0$,
the following estimates hold:
\begin{enumerate}[(i)]
\item For every $\kappa\ge0$, there exists $C_\kappa>0$ such that
\[
(1+t)^\kappa
\le C_\kappa(1+G(x,t))^{\frac{\kappa}{g^-}}.
\]

\item For every $\varepsilon>0$, there exists
$C_\varepsilon>0$ such that
\[
t\le\varepsilon G(x,t)+C_\varepsilon.
\]

\item We have
\[
tg(x,t)\le g^+G(x,t),
\qquad
g(x,t)\le C(1+G(x,t)).
\]

\item The complementary function satisfies
\[
\widetilde G(x,1+g(x,t))
\le C(1+G(x,t)).
\]
\end{enumerate}
\end{lemma}

\begin{proof}
We prove the four assertions separately.

\medskip
\noindent
\textbf{(i)} First, suppose that $t\ge1$. By \eqref{D22},
\[
t^{g^-}\le F G(x,t),
\]
and hence
\[
(1+t)^\kappa
\le 2^\kappa t^\kappa
\le 2^\kappa F^{\frac{\kappa}{g^-}}G(x,t)^{\frac{\kappa}{g^-}}.
\]
If $0\le t<1$, we simply use
\[
(1+t)^\kappa\le 2^\kappa.
\]
Combining these estimates yields
\[
(1+t)^\kappa
\le C_\kappa(1+G(x,t))^{\frac{\kappa}{g^-}},
\]
where $C_\kappa$ depends only on $\kappa$, $g^-$ and $F$.

\medskip
\noindent
\textbf{(ii)} For $t\ge1$, the same consequence of \eqref{D22}
gives
\[
t\le F^{\frac{1}{g^-}}G(x,t)^{\frac{1}{g^-}}.
\]
Since $g^->1$, Young's inequality implies that, for every
$\varepsilon>0$,
\[
F^{\frac{1}{g^-}}G(x,t)^{\frac{1}{g^-}}
\le \varepsilon G(x,t)+C_\varepsilon.
\]
For $0\le t<1$, the desired estimate follows by increasing
$C_\varepsilon$ if necessary. This proves (ii).

\medskip
\noindent
\textbf{(iii)} The first estimate follows directly from
\eqref{D22}:
\[
tg(x,t)\le g^+G(x,t).
\]
In particular, for $t\ge1$,
\[
g(x,t)\le \frac{g^+G(x,t)}{t}
\le g^+G(x,t).
\]
For $0<t<1$, again by \eqref{D22},
\[
g(x,t)\le g^+F t^{g^--1}\le g^+F.
\]
Combining the two cases, we obtain
\[
g(x,t)\le C(1+G(x,t)).
\]

\medskip
\noindent
\textbf{(iv)} By the $\Delta_2$-condition for $\widetilde G$
and the elementary inequality
\[
1+g(x,t)\le 2\max\{1,g(x,t)\},
\]
we have
\[
\begin{aligned}
\widetilde G(x,1+g(x,t))
&\le \widetilde G\bigl(x,2\max\{1,g(x,t)\}\bigr)\\
&\le 2^{\widetilde g^+}\widetilde G\bigl(x,\max\{1,g(x,t)\}\bigr)\\
&\le 2^{\widetilde g^+}\bigl(\widetilde G(x,1)+\widetilde G(x,g(x,t))\bigr).
\end{aligned}
\]
Moreover, by Proposition~\ref{zoo}(2) and $g^->1$,
\[
\widetilde G(x,1)=\sup_{s\ge0}\bigl(s-G(x,s)\bigr)\le\sup_{s\ge0}\bigl(s-F^{-1}\min\{s^{g^-},s^{g^+}\}\bigr)\le C(g^-,F).
\]
Using this bound, \eqref{L1} and the estimate established in (iii), we
conclude that
\[
\widetilde G(x,1+g(x,t))
\le C(1+G(x,t)).
\]
This completes the proof.
\end{proof}

As a consequence of \eqref{Yi}, we have the following result:
\begin{lemma}[H\"older's type
inequality]\label{H1}
  Let $\Omega$ be an open subset of $\mathbb{R}^n$ and $G$ be a generalized \textnormal{N}-function satisfies \eqref{D22}, then
  \begin{equation}\label{Ho}
     \left\vert \int_{\Omega} uv \, \mathrm{d}x \right\vert \leq 2 \Vert u\Vert_{L^{G(x,t)}(\Omega)}\Vert v\Vert_{L^{\widetilde{G}(x,t)}(\Omega)},\ \text{for all}\ u\in L^{G(x,t)}(\Omega)\ \text{and all}\ v\in L^{\widetilde{G}(x,t)}(\Omega).
  \end{equation}
\end{lemma}
The subsequent proposition deals with some topological properties of
the Musielak-Orlicz space, see \cite[Theorem 7.7 and Theorem
8.5]{Musielak1983}.
\begin{proposition}\label{AB}
\begin{enumerate}
\item [(1)] Let $G(x,t)$ be a generalized \textnormal{N}-function and $\Omega$ an open subset of
 $\mathbb{R}^n$. Then,
 \begin{enumerate}
     \item[(a)] the space
    $\left(L^{G(x,t)}(\Omega),\|\cdot\|_{G}\right)$ is a Banach space;
     \item[(b)] if $G(x,t)$ satisfies \eqref{D22}, then
$L^{G(x,t)}(\Omega)$ is a separable and reflexive space.
 \end{enumerate}
\item [(2)] Let $\Omega$ be an open bounded subset of $\mathbb{R}^n$  and $G_1(x,t)$ and $G_2(x,t)$  be two generalized \textnormal{N}-functions such that $\displaystyle{G_1(x,t) \mathop{\prec}_{\Omega} G_2(x,t)}$. Then,
\[
L^{G_2(x,t)}(\Omega) \hookrightarrow L^{G_1(x,t)}(\Omega).
\]
\end{enumerate}
\end{proposition}

Now, we are ready to define the  Musielak-Orlicz Sobolev space. Let
$G(x,t)$ be a generalized \textnormal{N}-function and $\Omega$ be an open subset of
$\mathbb{R}^n$. The  Musielak-Sobolev space is defined as follows
\[
W^{1,G(x,t)}(\Omega):=\left\lbrace u\in L^{G(x,t)}(\Omega):\ |D u| \in L^{G(x,t)}(\Omega)\right\rbrace.
\]
The space $W^{1,G(x,t)}(\Omega)$ is endowed with the norm
\begin{equation}\label{NM}
  \Vert u\Vert:=\Vert u\Vert_{L^{G(x,t)}(\Omega)}+\Vert D u\Vert_{L^{G(x,t)}(\Omega)},\ \ \text{for all}\ u\in W^{1,G(x,t)}(\Omega),
\end{equation}
where $\|D u\|_{L^{G(x,t)}(\Omega)} := \| |D u| \|_{L^{G(x,t)}(\Omega)}$.\\ We denote by
$W^{1,G(x,t)}_0(\Omega)$
 the completion of $C^\infty _0(\Omega)$ in $W^{1,G(x,t)}(\Omega)$.


\begin{remark}\label{Ref}
 If $G$ satisfies \eqref{D22} and \eqref{GG1} (so that $L^{G(x,t)}(\Omega)\subset L^1(\Omega)$ for bounded $\Omega$ and the distributional gradient makes sense), then the space
$W^{1,G(x,t)}(\Omega)$ is a reflexive and separable Banach space with
respect to the norm $\Vert \cdot\Vert$.
\end{remark}

Now we introduce the Sobolev conjugate from \cite{Ala3}, which refines the one proposed in \cite{Cianchi2024} and yields several important results presented below.

\begin{definition}\label{DefCL}
The Sobolev conjugate of $G(x,t)$ is defined as the generalized \textnormal{N}-function $G^\ast(x,t)$ given by
			\begin{equation*}
  G ^ \ast (x,t)=G (x, N^{-1} (x,t)) , \text{ for a.e. } x \in \mathbb{R}^n \text{ and all } t \geq 0,
\end{equation*}
where
\begin{equation*}
  N(x,t)= \left( \int_{0}^{t} \left( \frac{\tau}{G (x, \tau)} \right)^{\frac{1}{n-1}} \, \mathrm{d}\tau \right)^{\frac{n-1}{n}}, \text{ for } x \in \mathbb{R}^n \text{ and } t\geq 0.
\end{equation*}
 \end{definition}
 \begin{proposition}[\cite{Ala3}] \label{zoo*}
     Let $G(x,t)$ be a generalized \textnormal{N}-function that satisfies \eqref{D22}--\eqref{GG2} and $G^*(x,t)$ its Sobolev conjugate function, then the following assertions hold:
     \begin{itemize}
    \item [(1)]
    \[
    \min \left\{s^{g^-_*}, s^{g^+_*}\right\}G^*(x,t)\leq  G^*(x,s t)\leq \max \left\{ s^{g^-_*}, s^{g^+_*}\right\}G^*(x,t),\quad\text{for a.e. } x \in \Omega\text{ and all }  s, \ t \geq 0;
    \]
 \item [(2)]
 \[
 1<g^-_*\leq \frac{g^*(x,t)t}{G^*(x,t)}\leq g^+_*,\qquad \text{for all}\ x\in \Omega\ \text{and all}\ t> 0;
 \]
    \item[(3)]
    \[
    \min \left\{\|u\|_{L^{G^*(x,t)}(\Omega)}^{g^-_*},\|u\|_{L^{G^*(x,t)}(\Omega)}^{g^+_*}\right\} \leq \rho_{G^*}(u) \leq\max \left\{\|u\|_{L^{G^*(x,t)}(\Omega)}^{g^-_*},\|u\|_{L^{G^*(x,t)}(\Omega)}^{g^+_*}\right\},
    \]
    for all $u\in L^{G^*(x,t)}(\Omega)$;
\end{itemize}
where
\[
g_*^- := \frac{ng^-}{n - g^-},\qquad g_*^+ := \frac{ng^+}{n - g^+},\qquad G^*(x,t)=\int_0^t g^*(x,s)\, \mathrm{d}s.
\]
\end{proposition}
\begin{proposition}[\cite{Cianchi2024}]\label{embb}
Let $\Omega$ be a bounded domain in $\mathbb{R}^n$ and $G(x,t)$ be a generalized \textnormal{N}-function that satisfies \eqref{GG1}--\eqref{GG2}. Let $\vartheta(x,t)$ be a generalized \textnormal{N}-function such that $\displaystyle{\vartheta(x,t)\mathop{\prec}_{\Omega}G^*(x,t)}$
and
\begin{equation}
 \int_{\Omega} \vartheta(x,t) \, \mathrm{d}x < \infty,\ \text{for } t>0.
\end{equation}
\begin{enumerate}
\item[(1)] The embedding
\begin{equation}
W^{1,G(x,t)}_0(\Omega) \hookrightarrow L^{\vartheta(x,t)}(\Omega) 
\end{equation}
is continuous.
\item[(2)] Assume, in addition, that $\Omega$ is a bounded domain with a Lipschitz boundary $\partial \Omega$. Then, the embedding
\begin{equation}
W^{1,G(x,t)}(\Omega) \hookrightarrow L^{\vartheta(x,t)}(\Omega) 
\end{equation}
is continuous.
\end{enumerate}
\end{proposition}
\begin{remark}
If the function $G(x,t)$ satisfies assumptions \eqref{D22}--\eqref{GG2}, then it also fulfills conditions (A0)--(A1) of Definition~\ref{def:A012} (see L. Diening and A. Cianchi~\cite{Cianchi2024}), and it satisfies condition (A2) due to the boundedness of $\Omega \subset \mathbb{R}^n$. Consequently, the embedding result given in Proposition~\ref{embb} remains valid under assumptions \eqref{D22}--\eqref{GG2}.
\end{remark}

Since we need the Trace embedding theorem, we give the conjugate Sobolev function defined by Fan \cite{Fan2012a}.
\begin{definition}[\cite{Fan2012a}]\label{DefF}
Let $\Omega$ be an open bounded subset of $\mathbb{R}^n$ and $G(x,t)$ be a generalized \textnormal{N}-function that satisfies \eqref{D22}--\eqref{GG1} and \eqref{GG4}. Define a function $G_*^{-1} : \overline{\Omega} \times [0,+\infty) \to [0,+\infty)$ by
\[
G_*^{-1}(x,s) = \int_0^s \frac{G^{-1}(x,t)}{t^{\frac{n+1}{n}}}\, \mathrm{d}t,\quad \text{for } x \in \overline{\Omega} \text{ and } s \in [0,+\infty).
\]
Then under the assumption \eqref{GG4}, $G_*^{-1}$ is well defined, and for each $x \in \overline{\Omega}$, $G_*^{-1}(x,\cdot)$ is strictly increasing, $G_*^{-1}(x,\cdot) \in C^1((0,+\infty))$ and the function $G_*^{-1}(x,\cdot)$ is concave.\\
Set
\begin{equation}\label{Tx}
    T(x)=\lim_{s\to +\infty}G_*^{-1}(x,s),\quad\forall x\in \overline{\Omega}.
\end{equation}
Then $0<T(x)\leq +\infty$. We define  Fan's Sobolev conjugate function $G_* : \overline{\Omega} \times [0,+\infty) \to [0,+\infty)$ by
\[
G_*(x,t) =
\begin{cases}
s, & \text{for } x \in \overline{\Omega},\ t \in [0,T(x)) \text{ and } G_*^{-1}(x,s) = t, \\
+\infty, & \text{for } x \in \overline{\Omega} \text{ and } t\geq T(x).
\end{cases}
\]
The function $G_*(x,t)\in C^1((0,T(x)))$.  In particular, if $T(x)=+\infty,\quad\forall x\in \overline{\Omega}$,  $G_*(x,t)$ is a generalized \textnormal{N}-function.
\end{definition}
\begin{proposition}[\cite{Fan2012a}]\label{fembb}
Let $\Omega \subset \mathbb{R}^n$ be a bounded Lipschitz domain, and let $G(x,t)$ be a generalized \textnormal{N}-function satisfying \eqref{D22}--\eqref{GG1} and \eqref{GG4}--\eqref{GG5}. If $\vartheta(x,t)$ is a generalized \textnormal{N}-function such that
\[
\vartheta(x,t)\mathop{\prec}_{\Omega} G_*(x,t),
\]
then the embedding
\[
W^{1,G(x,t)}(\Omega) \hookrightarrow L^{\vartheta(x,t)}(\Omega)
\]
is continuous.
\end{proposition}
\begin{proposition}[\cite{LWZ}, Theorem 4.2]\label{ffembd}
Let $\Omega \subset \mathbb{R}^n$ be a bounded Lipschitz domain, and let $G(x,t)$ be a generalized \textnormal{N}-function satisfying \eqref{D22}--\eqref{GG1} and \eqref{GG4}--\eqref{GG5}. Suppose that
\(
\left[G_*(x,t)\right]^{\frac{n-1}{n}}
\)
satisfies the $\Delta_2$-condition. Then, for any generalized \textnormal{N}-function $\vartheta(x,t)$ such that
\[
\vartheta(x,t)\mathop{\prec}_{\partial\Omega}
\left[G_*(x,t)\right]^{\frac{n-1}{n}},
\]
the boundary trace embedding
\[
W^{1,G(x,t)}(\Omega)
\hookrightarrow
L^{\vartheta(x,t)}(\partial\Omega)
\]
is continuous.
\end{proposition}
 Proceeding as in \cite[Lemma 2.2]{Fuk2004}, we obtain the following result:
 \begin{proposition} \label{fzoo*}
     Let $G(x,t)$ be a generalized \textnormal{N}-function that satisfies \eqref{D22}--\eqref{GG2} and \eqref{GG4}; and $G_*(x,t)$ be its Fan's Sobolev conjugate function; then the following assertions hold:
     \begin{itemize}
    \item [(1)]
    \[
    \min \left\{s^{g^-_*}, s^{g^+_*}\right\}G_*(x,t)\leq  G_*(x,s t)\leq \max \left\{ s^{g^-_*}, s^{g^+_*}\right\}G_*(x,t),\quad\text{for a.e. } x \in \Omega\text{ and all }  s, \ t \geq 0;
    \]
 \item [(2)]
 \[
 1<g^-_*\leq \frac{g_*(x,t)t}{G_*(x,t)}\leq g^+_*,\qquad \text{for all}\ x\in \Omega\ \text{and all}\ t> 0;
 \]
    \item[(3)]
    \[
    \min \left\{\|u\|_{L^{G_*(x,t)}(\Omega)}^{g^-_*},\|u\|_{L^{G_*(x,t)}(\Omega)}^{g^+_*}\right\} \leq \rho_{G_*}(u) \leq\max \left\{\|u\|_{L^{G_*(x,t)}(\Omega)}^{g^-_*},\|u\|_{L^{G_*(x,t)}(\Omega)}^{g^+_*}\right\},
    \]
    for all $u\in L^{G_*(x,t)}(\Omega)$;
\end{itemize}
where
\[
g_*^- = \frac{ng^-}{n - g^-},\qquad g_*^+ = \frac{ng^+}{n - g^+},\qquad G_*(x,t)=\int_0^t g_*(x,s)\, \mathrm{d}s.
\]
\end{proposition}

\subsection{Technical Tools}
In this subsection, we present several established results from the literature that will serve as essential tools in our proof of the main regularity theorems.
\begin{definition}[weak solution, see \cite{HBO}]\label{def:weak}
Let $u \in W^{1, G(x,t)} (\Omega)$. Then $u$ is called a \textit{weak solution}
\begin{enumerate}
    \item[$(1)$] of problem \eqref{P} if \eqref{FDV} holds for all $v \in W_0^{1,G(x,t)}(\Omega)$;

\item[$(2)$] of the boundary value problem \eqref{PD} with $\phi \in W^{1, G(x,t)}(\Omega)$ if $u-\phi \in W_0^{1, G(x,t)} (\Omega)$ and \eqref{FDV} holds;

\item[$(3)$] of the boundary value problem \eqref{PN} if \eqref{FNV} holds for all $v \in W^{1,G(x,t)}(\Omega)$.
\end{enumerate}
\end{definition}

Obviously, every bounded generalized solution is also a weak solution (Definitions~\ref{def:sol} and \ref{def:weak}). Conversely, in \cite[Theorems~4.1w and 6.1w]{HBO} we have provided a sufficient condition for the boundedness of weak solutions.

\begin{remark}
In this paper, we study the $C^{1,\alpha}$ regularity of bounded generalized solutions. For a bounded generalized solution $u$, we already know that $u \in L^\infty(\Omega)$; consequently, there exists a constant $M\ge1$ such that
\[
\max_{\Omega} |u(x)| := \text{esssup}_{\Omega} |u(x)| \leq M. 
\]
Hence the functions $\lambda(|u|)$ and $\Lambda(|u|)$ in Assumptions \eqref{6.31}--\eqref{6.41} and \eqref{6.51} can be taken as the constants
$\lambda(M)$ and $\Lambda(M)$, respectively. Below we denote $\lambda(M)=\lambda_0$ and $\Lambda(M)=\Lambda_0$.
\end{remark}

According to \cite[Lemma 2.4]{ACAN}, we have the following result.

\begin{proposition}\label{prop1.112}
Let $A$ satisfy \eqref{6.31}, \eqref{6.41} and \eqref{7.411}, and let $\lambda_0:=\lambda(M)$, $\Lambda_0:=\Lambda(M)$. Then, for all $x\in\overline\Omega$, $|u|\le M$ and $\eta\in\mathbb{R}^n$,
\begin{equation}\label{eqc1}
   |A(x,u,\eta)| \leq \Lambda_{1}\,g(x,|\eta|),
\end{equation}
and  
\begin{equation}\label{eqc2}
  A(x,u,\eta) \cdot \eta \geq \lambda_{1}\,G(x,|\eta|),
\end{equation}
where $\lambda_{1}:=\lambda_0$ and $\Lambda_{1}:=\dfrac{\Lambda_0}{g^--1}$. In particular, $\lambda_{1}$ and $\Lambda_{1}$ depend only on $\lambda_{0}$, $\Lambda_{0}$ and $g^{-}$.
\end{proposition}

\begin{proof}
Fix $x\in\overline\Omega$, $|u|\le M$ and $\eta\neq0$ (for $\eta=0$ there is nothing to prove). By \eqref{7.411}, $A(x,u,s\eta)\to0$ as $s\to0^+$, and $A(x,u,\cdot)\in C^1(\mathbb{R}^n\setminus\{0\})$, so
\[
A(x,u,\eta)=\int_0^1\frac{\mathrm{d}}{\mathrm{d}s}A(x,u,s\eta)\, \mathrm{d}s=\int_0^1\partial_\xi A(x,u,s\eta)\,\eta\, \mathrm{d}s .
\]
By \eqref{6.41},
\[
|\partial_\xi A(x,u,s\eta)\,\eta|\le|\eta|\sum_{i,j}\Bigl|\frac{\partial A^i}{\partial\xi_j}(x,u,s\eta)\Bigr|\le\Lambda_0\,\frac{g(x,s|\eta|)}{s},
\]
and by \eqref{D22diff} the function $t\mapsto g(x,t)\,t^{-(g^--1)}$ is nondecreasing, so that $g(x,s|\eta|)\le s^{g^--1}g(x,|\eta|)$ for $s\in(0,1]$. Hence
\[
|A(x,u,\eta)|\le\Lambda_0\,g(x,|\eta|)\int_0^1 s^{g^--2}\, \mathrm{d}s=\frac{\Lambda_0}{g^--1}\,g(x,|\eta|).
\]
Similarly, by \eqref{6.31},
\[
\begin{aligned}
A(x,u,\eta)\cdot\eta
&=\int_0^1\sum_{i,j}\frac{\partial A^i}{\partial\xi_j}(x,u,s\eta)\,\eta_i\eta_j\, \mathrm{d}s\\
&\ge\lambda_0\int_0^1\frac{g(x,s|\eta|)}{s|\eta|}\,|\eta|^2\, \mathrm{d}s\\
&\ge\lambda_0\int_0^1 g(x,s|\eta|)\,|\eta|\, \mathrm{d}s=\lambda_0\,G(x,|\eta|),
\end{aligned}
\]
because $s^{-1}\ge1$ on $(0,1]$.
\end{proof}

By Proposition~\ref{prop1.112}, if $A$ satisfies \eqref{6.31}--\eqref{7.411} and $|u|\le M$, then $A$ satisfies the growth conditions $(\mathcal{A}_1^{\mathrm{HBO}})$--$(\mathcal{A}_2^{\mathrm{HBO}})$ of \cite{HBO} with $a_4=\lambda_1$, $a_1=\Lambda_1$ and $a_2=a_3=a_5=a_6=1$, see Remark~\ref{rem:notation-HBO}(b); moreover \eqref{6.51} and \eqref{6.61} are the conditions $(\mathcal{B})$ and $(\mathcal{C})$ of \cite{HBO}. Consequently, the results
of \cite{HBO} are applicable in our setting. In particular, we can invoke
\cite{HBO} to obtain the following H\"older continuity results for bounded
generalized solutions.

\begin{proposition}\label{prop1.1}
Under the assumptions \eqref{6.31}--\eqref{6.41}, \eqref{6.51}, \eqref{D22}--\eqref{GG6b} and \eqref{H111}--\eqref{YoungTriple4} (together with \eqref{H12} in $(2)$, and with \eqref{6.61}--\eqref{C1}, \eqref{NH111}--\eqref{NYoungTriple4}, \eqref{GG4}--\eqref{GG5} in $(3)$), the following statements hold:
\begin{enumerate} 
\item[$(1)$] If $u \in W^{1,G(x,t)} (\Omega)$ is a bounded generalized solution of problem \eqref{P} and satisfies \eqref{M}, then $u \in C_{\text{loc}}^{0,\alpha_1} (\Omega)$, where $\alpha_1$ depends only on $\mathrm{data}_G$, $\lambda_0$, $\Lambda_0$, $\mathrm{data}_B$ and $M$ (the constants $\lambda_1,\Lambda_1$ of Proposition~\ref{prop1.112} are functions of $\lambda_0,\Lambda_0,g^-$). Moreover, for any $\Omega_0 \Subset \Omega$, the norm $\Vert u\Vert_{C^{0,\alpha_1} (\overline{\Omega_0})}$ depends only on the same quantities and on $\mathrm{dist}(\Omega_0, \partial \Omega)$.

\item[$(2)$] If $\partial \Omega$ satisfies condition \eqref{AO} and $u \in W^{1,G(x,t)} (\Omega)$ is a bounded generalized solution of the boundary value problem \eqref{PD} satisfying \eqref{M}, then $u \in C^{0,\alpha_1} (\overline{\Omega})$, where $\alpha_1$ depends only on $\mathrm{data}_G$, $\lambda_0$, $\Lambda_0$, $\mathrm{data}_B$, $M$, on $\theta_0$ of \eqref{AO}, and on the Hölder exponent and seminorm of the datum $\phi$ (exponent $1$ and seminorm $\le M_\phi$, since $\phi\in C^{1,\beta_3}$); the norm $\Vert u\Vert_{C^{0,\alpha_1} (\overline{\Omega})}$ depends only on the same quantities and on $\Omega$.

\item[$(3)$] If $\partial \Omega$ is Lipschitz and $u \in W^{1,G(x,t)} (\Omega)$ is a bounded generalized solution of the boundary value problem \eqref{PN} satisfying \eqref{M}, then $u \in C^{0,\alpha_1} (\overline{\Omega})$, where $\alpha_1$ depends only on $\mathrm{data}_G$, $\lambda_0$, $\Lambda_0$, $\mathrm{data}_B$, $\mathrm{data}_C$ and $M$, and the norm $\Vert u\Vert_{C^{0,\alpha_1} (\overline{\Omega})}$ depends only on the same quantities and on $\Omega$. (In \cite{HBO}, the dependence of $\alpha_1$ on $b_1,b_2,b_3,h_1^+,F_1,F_2,c_1,c_2$ goes through $\gamma_1$ and $\delta=\min\{\frac{\lambda_1}{4b_3M},2\}$ of the De Giorgi class, Theorems~4.2, 4.21 and 4.3 there; the boundary case uses Theorem~2.2 there.)
\end{enumerate}
\end{proposition}

\begin{proof}
Throughout, as in \cite{HBO}, we may assume $M\ge1$ (replace $M$ by $\max\{M,1\}$).
Statement $(1)$ follows from Theorems~4.2, 4.3 and 2.1 of \cite{HBO}: the solution belongs to the De Giorgi class, whose Hölder exponent is given by Theorem~2.1 there. Statement $(2)$ follows from Theorems~4.2 and 2.2 of \cite{HBO}; the boundary-oscillation lemma and the proof of Theorem~2.2 there show that the exponent depends on the data, on $\theta_0$, $\delta$, $M$, and on the Hölder exponent and seminorm of the boundary datum. Statement $(3)$ follows from Theorems~4.21, 4.3 and 4.4 of \cite{HBO}. In all cases the structural constants of \cite{HBO} enter through $\gamma=\gamma(a_2,a_3,a_4,g^\pm)$, $\gamma_1=\gamma_1(a_1,\dots,a_6,b_1,b_2,b_3,g^+,h_1^+,F_1,M)$ (together with $c_1,c_2,F_2$ in the Neumann case) and $\delta=\min\{\frac{a_4}{4b_3M},2\}$, with $a_4=\lambda_1=\lambda_0$, $a_1=\Lambda_1=\frac{\Lambda_0}{g^--1}$ and $a_2=a_3=a_5=a_6=1$ (Remark~\ref{rem:notation-HBO}(b)); in particular $\gamma$ depends only on $\lambda_0$ and $g^\pm$, and $\gamma_1$, $\delta$ only on $\lambda_0,\Lambda_0,g^\pm,b_1,b_2,b_3,h_1^+,F_1,M$ (and $c_1,c_2,F_2$ in the Neumann case).
For the Dirichlet problem with a non-zero datum, the class $\mathcal B_{G}(\overline\Omega,M,\gamma,\gamma_1,\delta)$ of \cite{HBO} is obtained exactly as for $\phi=0$: for balls $B_\rho$ meeting $\partial\Omega$ one tests with $\zeta^{g^+}(u-k)^+$, where $k\ge\sup_{(\partial\Omega)_\rho}u$ as required in the definition of the boundary class in \cite{HBO}, so that the test function vanishes on $\partial\Omega$ and $\phi$ plays no role. Since $\phi\in C^{1,\beta_3}\subset C^{0,1}(\partial\Omega)$, Theorem~2.2 of \cite{HBO} applies with boundary exponent $1$ and seminorm $\le M_\phi$, and condition \eqref{AO} holds because $\partial\Omega\in C^{1,\beta_3}$.
\end{proof}

\begin{remark}
It is clear that all assumptions of Proposition \ref{prop1.1} are satisfied if the assumptions of Theorems \ref{Thm1.1}--\ref{Thm1.3} hold. Hence, if $u$ is a bounded generalized solution as in Theorems \ref{Thm1.1}--\ref{Thm1.3}, then $u \in C^{0,\alpha_1}_{\text{loc}}(\Omega)$ in the case of Theorem \ref{Thm1.1}, and $u \in C^{0,\alpha_1}(\overline{\Omega})$ in the cases of Theorems \ref{Thm1.2} and \ref{Thm1.3}.
\end{remark}

\begin{proposition}\label{Prop2.2}
Let $A$ satisfy Assumptions \eqref{6.31}--\eqref{6.41}, $A(x,u,0)=0$, and $|u|\le M$. Then for all $\eta,\eta'\in\mathbb{R}^n$ with $S:=|\eta|+|\eta'|>0$,
\begin{align}\label{SIM}
\big(A(x, u, \eta) - A(x, u, \eta')\big)\cdot(\eta - \eta') \geq 
c_A\frac{g(x,S)}{S}\,|\eta - \eta'|^2,
\end{align}
where $c_A = \frac{\lambda_0}{4}4^{-\max\{g^+-2,0\}}$, $\lambda_0=\lambda(M)$.
\end{proposition}
\begin{proof}
Set
\[
\varphi(t):=\frac{g(x,t)}{t}.
\]
By \eqref{D22diff}, the functions
\[
t\longmapsto g(x,t)t^{-(g^--1)}
\quad\text{and}\quad
t\longmapsto g(x,t)t^{-(g^+-1)}
\]
are respectively nondecreasing and nonincreasing. Consequently,
\[
t\longmapsto\varphi(t)t^{-(g^--2)}
\quad\text{is nondecreasing},
\]
whereas
\[
t\longmapsto\varphi(t)t^{-(g^+-2)}
\quad\text{is nonincreasing}.
\]
Let
\[
\eta_s:=\eta'+s(\eta-\eta'),\quad s\in[0,1].
\]
Since $|\eta-\eta'|\le S$, the map $s\mapsto A(\eta_s)$
is absolutely continuous. Indeed, it is continuously differentiable
away from at most one value of $s$, and its derivative is integrable
across that point by Proposition~\ref{prop1.112} and the monotonicity
of $g$. Therefore, the fundamental theorem of calculus yields
\[
\begin{aligned}
&(A(\eta)-A(\eta'))\cdot(\eta-\eta')\ge
\lambda_0|\eta-\eta'|^2
\int_0^1\varphi(|\eta_s|)\, \mathrm{d}s.
\end{aligned}
\]

By the triangle inequality, at least one of $|\eta|$ and
$|\eta'|$ is not smaller than $\frac{S}{2}$. Without loss of generality,
assume that $|\eta|\ge \frac{S}{2}$; the other case is treated symmetrically,
by integrating over $s\in\left[0,\frac{1}{4}\right]$.

For every $s\in\left[\frac{3}{4},1\right]$, we have
\[
\begin{aligned}
|\eta_s|
&\ge |\eta|-(1-s)|\eta-\eta'|\ge \frac{S}{2}-\frac{S}{4}
=\frac{S}{4},
\end{aligned}
\]
while $|\eta_s|\le S$. By the monotonicity properties of $\varphi$,
it follows that
\[
\varphi(|\eta_s|)
\ge 4^{-\max\{g^+-2,0\}}\varphi(S).
\]
Consequently,
\[
\int_0^1\varphi(|\eta_s|)\, \mathrm{d}s
\ge
\frac{1}{4}\,4^{-\max\{g^+-2,0\}}\varphi(S).
\]
Combining the preceding estimates completes the proof.
\end{proof}
\begin{lemma}\label{lem:orlicz-ineq}
Let $\mathsf G$ be an \textnormal{N}-function with $\mathsf g:=\mathsf G'\in C^1((0,\infty))$ satisfying \eqref{D22diff} (with $x$ suppressed; in particular \eqref{D22} holds) and set
\[
\theta_*:=\min\left\{1,\frac{g^-}{2}\right\}.
\]
Then, for every $\eta,\eta'\in\mathbb R^n$, writing
\[
S:=|\eta|+|\eta'|,
\qquad s:=|\eta-\eta'|\le S,
\]
we have
\begin{equation}\label{eq:orlicz-int}
\mathsf G(s)
\le C\,\mathsf G(S)^{1-\theta_*}
\left[\frac{\mathsf g(S)}{S}s^2\right]^{\theta_*},
\end{equation}
where $C=C(g^\pm)$.
\end{lemma}

\begin{proof}
Set
\[
\varphi(t):=\frac{\mathsf g(t)}{t}.
\]
By \eqref{D22}, the functions
\[
t\longmapsto \frac{\mathsf G(t)}{t^{g^-}}
\quad\text{and}\quad
t\longmapsto \frac{\mathsf G(t)}{t^{g^+}}
\]
are respectively nondecreasing and nonincreasing. In particular,
for $0<s\le S$,
\[
\mathsf G(s)\le
\mathsf G(S)\left(\frac{s}{S}\right)^{g^-}.
\]

We distinguish two cases.

\medskip
\noindent
\textbf{Case 1: $g^-\ge2$.}
In this case, $\theta_*=1$. The \eqref{D22diff} condition also implies
that
\[
t\longmapsto \frac{\varphi(t)}{t^{g^--2}}
\]
is nondecreasing. Hence, for $s\le S$,
\[
\varphi(s)\le
\varphi(S)\left(\frac{s}{S}\right)^{g^--2}
\le\varphi(S).
\]
Consequently,
\[
\mathsf G(s)
\le \frac{s\mathsf g(s)}{g^-}
=\frac{s^2\varphi(s)}{g^-}
\le \frac{\mathsf g(S)}{g^-S}s^2,
\]
which gives \eqref{eq:orlicz-int}.

\medskip
\noindent
\textbf{Case 2: $1<g^-<2$.}
Here, $\theta_*=\frac{g^-}{2}$. By the preceding growth estimate,
\[
\mathsf G(s)
\le \mathsf G(S)\left(\frac{s}{S}\right)^{g^-}
=\mathsf G(S)\left(\frac{s^2}{S^2}\right)^{\theta_*}.
\]
Moreover, the lower growth bound gives
\[
\frac{\mathsf g(S)}{S}s^2
\ge g^-\mathsf G(S)\frac{s^2}{S^2}.
\]
Therefore,
\[
\mathsf G(s)
\le (g^-)^{-\theta_*}
\mathsf G(S)^{1-\theta_*}
\left[\frac{\mathsf g(S)}{S}s^2\right]^{\theta_*}.
\]
This proves \eqref{eq:orlicz-int} with a constant depending
only on $g^-$.
\end{proof}

\begin{lemma}[growth of $B$]\label{lem:B}
Assume \eqref{H111}--\eqref{H211}, \eqref{6.51} and \eqref{M}. Let $\theta_B:=1-\frac{1}{h_1^+}\in(0,1)$. There are constants $C_B>0$ and, for every $\varepsilon>0$, $C_{B,\varepsilon}>0$, depending only on $\mathrm{data}_B$ and $M$ (and on $\varepsilon$ for $C_{B,\varepsilon}$), such that
\begin{equation}\label{eq:B-bound}
|B(x,u,\eta)|\le C_B\big(1+G(x,|\eta|)^{\theta_B}\big)\le \varepsilon G(x,|\eta|)+C_{B,\varepsilon}\quad(|u|\le M).
\end{equation}
\end{lemma}

\begin{proof}
By \eqref{H111}--\eqref{H211}, for $0\le s\le1$,
\[
H_1(x,s)\le F_1s^{h_1^-}.
\]
Consequently, the complementary function satisfies
\[
\widetilde H_1(x,1)\ge c_H>0,
\]
with $c_H=c_H(F_1,h_1^-)>0$ independent of $x$ (indeed $\widetilde H_1(x,1)\ge s-H_1(x,s)\ge\frac{s}{2}$ for $s:=(2F_1)^{-\frac{1}{h_1^--1}}\in(0,1)$). Moreover, Proposition~\ref{zoo}(3)
yields
\[
\widetilde H_1(x,t)
\ge c_Ht^{\frac{h_1^+}{h_1^+-1}},
\qquad t\ge1.
\]
It follows that
\[
\widetilde H_1^{-1}(x,s)
\le C\left(1+s^{\frac{1}{(h_1^+)'}}\right),
\qquad s\ge0.
\]
Thus, setting
\[
\theta_B:=\frac{1}{(h_1^+)'}<1,
\]
we obtain
\[
\widetilde H_1^{-1}(x,s)\le C(1+s^{\theta_B}).
\]

On the other hand, since $|u|\le M$, we have
\[
h_1(x,|u|)\le h_1^+F_1M^{h_1^+-1},
\]
uniformly in $x$ and $u$ (here $M\ge1$, and we used $h_1(x,s)\le \frac{h_1^+H_1(x,s)}{s}$ together with \eqref{H111}--\eqref{H211}). Combining these estimates with
Young's inequality gives the desired second bound. By construction, $\theta_B$, $C_B$ and $C_{B,\varepsilon}$ depend only on $\mathrm{data}_B$ and $M$ (and on $\varepsilon$).
\end{proof}

\begin{remark}\label{Rem2}
For given $\delta > 0$ and $|u_1|,|u_2|\le M$, from $\ln (t)\leq ct^\delta$ for all $t\geq 1$, it follows by \eqref{7.41} that
\begin{equation}\label{eq2.10}
\begin{split}
|A(x_1, u_1, \eta) - A(x_2, u_2, \eta)| \leq{}& \Lambda_{\delta}  \left(|x_1 - x_2|^{\beta_1} + |u_1 - u_2|^{\beta_2}\right)\\
&\times\max\left\lbrace g(x_1,1+|\eta|),g(x_2,1+|\eta|)\right\rbrace\left(1+(1+|\eta|)^{\delta}\right),
\end{split}
\end{equation}
where $\Lambda_{\delta}$ is a positive constant depending only on $\Lambda_0$, $F$, $g^-$ and $\delta$ (we use $\Lambda(\max\{|u_1|,|u_2|\})\le\Lambda_0$, the monotonicity of $g(x,\cdot)$ and $1\le\frac{F}{g^-}\,g(x,1)$, which follows from \eqref{D22} and \eqref{GG1}).
\end{remark}

\begin{lemma}\label{lem:comp-factor}
Assume \eqref{D22}, \eqref{GG1}, \eqref{GG3} and \eqref{GG6b}.
There exists a constant
\[
K:=\max\Bigl\{F^2,\ \frac{1}{\beta_0}\Bigl(\frac{\bar a}{\ell(1)}+\Theta\bar a\Bigr)\Bigr\}\ge F^2,
\]
depending only on $F,\bar a,\Theta,\beta_0$ and $\ell(1)$ (the parameters $d,p^\pm,q^\pm$ and $L_0$ do not enter $K$; $L_0$ enters only through the exponent $\omega(R)=\frac{L_0}{|\ln R|}$),
such that, for every $x_0\in\overline\Omega$ and $R<R_0$,
all $x,y\in B_R(x_0)\cap\overline\Omega$, and all $t\ge0$,
\begin{equation}\label{eq:cf-raw}
\begin{aligned}
G(x,t)&\le F+Kt^{\omega(R)}G(y,t),\\
G(y,t)&\le F+Kt^{\omega(R)}G(x,t).
\end{aligned}
\end{equation}
Moreover, for every $\kappa\ge0$, there exists a constant
$C_\kappa>0$, depending only on
$\kappa,\allowbreak\,g^-,\allowbreak\,F,\allowbreak\,\bar a,\allowbreak\,\Theta,\allowbreak\,\beta_0,\allowbreak\,\ell(1),\allowbreak\,L_0,$
such that
\begin{equation}\label{eq:cf-main}
(1+t)^\kappa G(x,t)
\le C_\kappa
\bigl(1+G(y,t)\bigr)^{1+\frac{\kappa+\omega(R)}{g^-}},
\end{equation}
and the same estimate holds with $x$ and $y$ interchanged.
In particular, if $\delta>0$ and $\kappa\ge0$ satisfy
\[
\frac{\kappa+\omega(R)}{g^-}\le\frac{\delta}{2},
\]
then
\begin{equation}\label{eq:cf-small}
(1+t)^\kappa G(x,t)
\le C_\kappa\bigl(1+G(y,t)\bigr)^{1+\frac{\delta}{2}},
\end{equation}
and likewise with $x$ and $y$ interchanged.
\end{lemma}

\begin{proof}
We first establish \eqref{eq:cf-raw}. For $0\le t\le1$,
the desired estimates follow directly from the monotonicity
of $G$ and \eqref{GG1}: $G(x,t)\le G(x,1)\le F$.

Now let $t\ge1$. Since $p_i\le q_i$ and
$\ell(t)\ge\ell(1)>0$, assumptions \eqref{GG3}(i),(ii),
together with
\[
b_i\le\Theta b_i'\le\Theta\bar a,
\]
give
\[
\begin{aligned}
G(x,t)
&\le
\sum_i\left(\frac{\bar a}{\ell(1)}+\Theta\bar a\right)
t^{q_i(x)}\ell(t)=C_{\mathrm{up}}\sum_i t^{q_i(x)}\ell(t),
\end{aligned}
\]
where $C_{\mathrm{up}}$ depends only on $\bar a,\Theta$ and $\ell(1)$.

On the other hand, by \eqref{GG3}(i),(vi),
\[
G(y,t)\ge
\beta_0\sum_i t^{q_i'(y)}\ell(t).
\]
Furthermore, \eqref{GG3}(iv) yields
\[
q_i(x)\le q_i'(y)+\omega(R).
\]
Since $t\ge1$, we therefore have
\[
t^{q_i(x)}
\le t^{\omega(R)}t^{q_i'(y)}.
\]
Combining these inequalities, we obtain
\[
G(x,t)
\le\frac{C_{\mathrm{up}}}{\beta_0}
t^{\omega(R)}G(y,t).
\]
Interchanging $x$ and $y$ gives the other inequality.
Thus, with $K:=\max\{F^2,\frac{C_{\mathrm{up}}}{\beta_0}\}$,
we obtain \eqref{eq:cf-raw}.

We next prove \eqref{eq:cf-main}. By \eqref{D22},
\[
t^{g^-}\le FG(y,t),\qquad t\ge1.
\]
Consequently,
\[
\begin{aligned}
(1+t)^\kappa
&\le 2^\kappa t^\kappa
\le 2^\kappa F^{\frac{\kappa}{g^-}}
G(y,t)^{\frac{\kappa}{g^-}},\\
t^{\omega(R)}
&\le F^{\frac{\omega(R)}{g^-}}
G(y,t)^{\frac{\omega(R)}{g^-}}.
\end{aligned}
\]
Using the first inequality in \eqref{eq:cf-raw} and the bound $\omega(R)\le \frac{L_0}{\ln2}$ (recall $R<R_0<\frac{1}{2}$), which makes $F^{\frac{\omega(R)}{g^-}}\le F^{\frac{L_0}{g^-\ln2}}$ uniform in $R$, we get
\[
\begin{aligned}
(1+t)^\kappa G(x,t)
&\le F(1+t)^\kappa
+K(1+t)^\kappa t^{\omega(R)}G(y,t)\\
&\le C_\kappa
\left[
G(y,t)^{\frac{\kappa}{g^-}}
+G(y,t)^{1+\frac{\kappa+\omega(R)}{g^-}}
\right]\\
&\le C_\kappa
\bigl(1+G(y,t)\bigr)^{
1+\frac{\kappa+\omega(R)}{g^-}}.
\end{aligned}
\]
For $0\le t\le1$, the same conclusion follows directly
from \eqref{eq:cf-raw}, since $(1+t)^\kappa\le2^\kappa$
and $t^{\omega(R)}\le1$.

This proves \eqref{eq:cf-main}. The estimate with $x$ and
$y$ interchanged follows in exactly the same way from the
second inequality in \eqref{eq:cf-raw}.

Finally, if
\[
\frac{\kappa+\omega(R)}{g^-}\le\frac{\delta}{2},
\]
then
\[
1+\frac{\kappa+\omega(R)}{g^-}
\le1+\frac{\delta}{2}.
\]
Since $1+G(y,t)\ge1$, estimate \eqref{eq:cf-small}
follows immediately from \eqref{eq:cf-main}.
\end{proof}

\begin{corollary}\label{cor:g-comp}
Under the assumptions of Lemma~\ref{lem:comp-factor}, there is $C_g=C_g(g^\pm,F,K)\ge1$ such that, for every $x_0\in\overline\Omega$ and $R<R_0$, all $x,y\in B_R(x_0)\cap\overline\Omega$ and all $t\ge0$,
\[
g(x,t)\le C_g\,(1+t)^{\omega(R)}\bigl(1+g(y,t)\bigr).
\]
\end{corollary}

\begin{proof}
For $0<t<1$, \eqref{D22} and \eqref{GG1} give $g(x,t)\le g^+F\,t^{g^--1}\le g^+F$, as in the proof of Lemma~\ref{lem:elem}(iii). For $t\ge1$, \eqref{D22} gives $g(x,t)\le \frac{g^+G(x,t)}{t}$ and $G(y,t)\le \frac{t\,g(y,t)}{g^-}$; hence, by \eqref{eq:cf-raw},
\[
g(x,t)\le\frac{g^+}{t}\bigl(F+Kt^{\omega(R)}G(y,t)\bigr)\le g^+F+\frac{g^+K}{g^-}\,t^{\omega(R)}g(y,t).\qedhere
\]
\end{proof}

\begin{remark}\label{rem:omega}
Recall that
\[
\omega(R)=\frac{L_0}{|\ln R|}.
\]
We shall repeatedly use the following two elementary properties:
\begin{enumerate}
\item As $R\to0$, we have $\omega(R)\to0$.

\item For every $m>0$, every $0<\rho<1$, and every
$1\le t\le\rho^{-m}$,
\[
t^{\omega(\rho)}
\le \exp\bigl(m|\ln\rho|\,\omega(\rho)\bigr)
=e^{mL_0}=:L^m,
\]
where $L:=e^{L_0}$. Consequently, by \eqref{eq:cf-raw},
on every set $B_\rho(x_0)\cap\overline\Omega$ ($x_0\in\overline\Omega$) and for
$1\le t\le\rho^{-m}$,
\[
G(x,t)\le F+KL^mG(y,t),
\]
and similarly with $x$ and $y$ interchanged. In particular,
\[
\frac{G(x,t)}{G(y,t)}
\le F+KL^m,
\]
provided $G(y,t)\ge1$.
\end{enumerate}
\end{remark}

\subsection{Analytic tools}
We collect the standard tools used in the proofs below.

\begin{proposition}[ {\cite[Theorem.~6.2.8]{Harjulehto2019}}]
\label{prop:SP}
Assume \eqref{D22}, \eqref{GG1}, and \eqref{GG2}. Then there exist
$r_0>0$, $\varepsilon\in(0,1)$, and $c\ge1$, depending only on $n$, $g^\pm$, $F$ and $\mu$, such that, for every ball $B_r\subset\mathbb{R}^n$ with $r\le r_0$ and every $w\in W^{1,G}(B_r)$,
\[
\fint_{B_r}
G\left(x,\frac{|w-(w)_r|}{r}\right)\, \mathrm{d}x
\le
c\left(
\fint_{B_r}G(x,|Dw|)^{\frac{1}{1+\varepsilon}}\, \mathrm{d}x
\right)^{1+\varepsilon}
+c.
\]
Moreover, if $w$ vanishes on a measurable subset of $B_r$ of measure
at least $\sigma|B_r|$, then the same inequality holds with $(w)_r$
replaced by $0$, with a constant $c$ that depends in addition on $\sigma$.
The additive constant arises from the restriction
$t\in[1,|B_r|^{-1}]$ in \eqref{GG2}. In the case of an
$x$-independent $N$-function, this additive constant is not needed.
\end{proposition}

\begin{proposition}[
{\cite[Ch.~V, Prop.~1.1]{11}}]
\label{prop:Gehring}
Let $Q$ be a ball, a cube, or the intersection of a ball or a cube
with a Lipschitz domain, with constants independent of the size of
the sets, and let $\kappa\ge2$. Let $s>1$, $f\in L^s(Q)$, and
$\varphi\in L^{s(1+\delta_1)}(Q)$ for some $\delta_1>0$.
Suppose that
\[
\fint_{Q_r} f^s\, \mathrm{d}x
\le
b\left(\fint_{Q_{\kappa r}}f\, \mathrm{d}x\right)^s
+\fint_{Q_{\kappa r}}\varphi^s\, \mathrm{d}x
\]
for every pair of concentric sets $Q_r\subset Q_{\kappa r}\subset Q$,
where the balls or cubes are allowed to be centered at any point
of $Q$. Then there exists $\delta_G=\delta_G(n,s,b,\kappa)>0$ (depending also on the Lipschitz character of $Q$) such that
\[
f\in L^{s(1+\delta)}_{\mathrm{loc}}(Q)
\qquad\text{for every }0<\delta\le\min\{\delta_G,\delta_1\},
\]
and the corresponding reverse H\"older inequality
\[
\left(\fint_{Q_r}f^{s(1+\delta)}\, \mathrm{d}x\right)^{\frac{1}{1+\delta}}
\le c\left[\fint_{Q_{\kappa r}}f^s\, \mathrm{d}x+\left(\fint_{Q_{\kappa r}}\varphi^{s(1+\delta)}\, \mathrm{d}x\right)^{\frac{1}{1+\delta}}\right]
\]
holds with $c=c(n,s,b,\kappa,\delta)$, for every pair $Q_r\subset Q_{\kappa r}\subset Q$.
\end{proposition}
For $\kappa=2$ this is the statement of \cite[Ch.~V, Prop.~1.1]{11}; for $\kappa>2$ the proof is the same, the parent cube in the Calder\'on--Zygmund argument being replaced by a dyadic ancestor containing the $\kappa$-fold dilation of the cube (the constants then depend also on $\kappa$). We use $\kappa=2$ in Lemmas~\ref{Lem1.1} and~\ref{Lem2.1} and $\kappa=4$ in Lemmas~\ref{Lem1.4a} and~\ref{Lem2.21}; below, all radii are restricted to $r\le r_1$ for some fixed $r_1$, which is harmless because the statement is invariant under dilations and may be applied in balls of radius $r_1$, followed by a finite covering.

\begin{lemma}[{\cite[Ch.~III, Lemma~2.1]{11}}]
\label{lem:iter}
Let $\Phi:(0,\bar R]\to[0,\infty)$ be nondecreasing and suppose that
\[
\Phi(\rho)
\le
A\left[\left(\frac{\rho}{R}\right)^a+\epsilon\right]\Phi(R)
+BR^b
\]
for all $0<\rho\le R\le \bar R$, where $0<b<a$.
Then there exists $\epsilon_0=\epsilon_0(A,a,b)>0$ such that,
if $\epsilon\le\epsilon_0$, then
\[
\Phi(\rho)
\le
C\left[
\left(\frac{\rho}{R}\right)^b\Phi(R)+B\rho^b
\right]
\]
for all $0<\rho\le R\le \bar R$, where $C$ depends only on
$A,a,b$.
\end{lemma}

\begin{proposition}
\label{prop:Camp}
Let $Q$ be a ball or a Lipschitz domain, and let $f\in L^s(Q)$
for some $s\ge1$. Suppose that
\[
\int_{B_\rho(y)\cap Q}
|f-\lbrace f\rbrace_{y,\rho}|^s\, \mathrm{d}x
\le C\rho^{n+s\alpha}
\]
for every $y\in \overline Q$ and every $0<\rho\le\rho_0$, where
\[
\lbrace f\rbrace_{y,\rho} := \frac{1}{|B_{\rho}(y)\cap Q|} \int_{B_{\rho}(y)\cap Q} f(x)\, \mathrm{d}x
\]
(recall that $|B_\rho(y)\cap Q|\ge c\rho^n$ for a Lipschitz domain $Q$; below, $\lbrace f\rbrace_\rho$ denotes this average).
Then $f\in C^{0,\alpha}(\overline Q)$, with the corresponding
norm bound.
\end{proposition}

\section{Proofs of main results}\label{sec:proofs}
This section is devoted to the proof of Theorems~\ref{Thm1.1}--\ref{Thm1.3}. For the reader's convenience, we follow the approach developed in \cite{Fan2007}, adapting it to the Musielak-Orlicz framework considered in this paper. The Morrey--Campanato argument that concludes each of the three proofs is isolated in Subsection~\ref{subsec:tools}.

\subsection{Two lemmas for the Morrey--Campanato argument}\label{subsec:tools}
The three proofs below end with the same Morrey--Campanato argument: a Morrey-type bootstrap, based on a comparison with a frozen problem, followed by a Campanato estimate. We state once the two ingredients that are common to the interior, Dirichlet and Neumann cases.

\begin{lemma}\label{Lem:boot}
Let $G$ be the growth function of the problem under consideration, in the original coordinates (interior case) or in the flattened coordinates of Subsection~\ref{subsec:proof-bdry} (boundary cases), with the constants $K$ and $L_0$ of Lemma~\ref{lem:comp-factor} and $\omega(r)=\frac{L_0}{|\ln r|}$, and let $u\in W^{1,G}$ be the solution under consideration. Let $\mathcal Y\subset\mathbb{R}^n$ and, for $y\in\mathcal Y$ and $r>0$, let $E_r(y)\subset B_r(y)$ be measurable sets, nondecreasing in $r$, and put
\[
\psi_y(r):=\int_{E_r(y)}\bigl(1+G(x,|Du|)\bigr)\, \mathrm{d}x,\qquad
\varphi_y(r):=\int_{E_r(y)}\bigl(1+G(y,|Du|)\bigr)\, \mathrm{d}x .
\]
Let $\tau_0\in(0,n]$, $\delta,\delta'\in(0,1]$, $\chi_0>0$, $R_\flat\in(0,1)$ and $A_0\ge1$ be such that $\tau_0\delta\le\frac{\chi_0}{2}$ and $\omega(r)\le g^-\delta'$ for $0<r\le R_\flat$, and assume that, for every $y\in\mathcal Y$:
\begin{enumerate}[(a)]
\item for $0<r\le R_\flat$ and $x\in E_r(y)$, the conclusions of Lemma~\ref{lem:comp-factor}, namely \eqref{eq:cf-raw} and \eqref{eq:cf-main} (with $\kappa=0$), also with the roles of $x$ and $y$ interchanged, hold with $R=r$;
\item $\psi_y(r)\le A_0\,r^{\,n-\tau_0}$ for $0<r\le R_\flat$;
\item for $0<r\le\frac{R_\flat}{2}$ and $s\in\{\delta,\delta'\}$,
\[
\int_{E_r(y)}\bigl(1+G(x,|Du|)\bigr)^{1+s}\, \mathrm{d}x\le A_0\,r^{-ns}\,\psi_y(2r)^{1+s};
\]
\item for $0<R\le\frac{R_\flat}{2}$ and $0<\rho\le\frac{R}{4}$,
\[
\int_{E_\rho(y)}G(y,|Du|)\, \mathrm{d}x\le A_0\Bigl(\frac{\rho}{R}\Bigr)^n\int_{E_R(y)}\bigl(1+G(y,|Du|)\bigr)\, \mathrm{d}x+A_0\rho^n+A_0R^{\chi_0-\frac{n\delta}{2}}\psi_y(2R)^{1+\frac{\delta}{2}}.
\]
\end{enumerate}
Then for every $\tau\in(0,n)$ there are $R_\tau\in(0,R_\flat]$ and $C_\tau>0$ such that
\begin{equation}\label{eq:boot}
\psi_y(\rho)+\varphi_y(\rho)\le C_\tau\,\rho^{\,n-\tau}\qquad\forall y\in\mathcal Y,\ \forall\rho\in(0,R_\tau].
\end{equation}
Here $R_\tau$ and $C_\tau$ depend only on $n$, $g^\pm$, $F$, $K$, $L_0$, $A_0$, $\tau_0$, $\delta$, $\delta'$, $\chi_0$, $R_\flat$, $\tau$ and on the constant $C_0$ of \eqref{eq:cf-main} with $\kappa=0$.
\end{lemma}

\begin{proof}
In this proof $C$ denotes a constant depending only on the quantities listed in the statement; it may change from line to line. Since $\rho^{\,n-\tau'}\le\rho^{\,n-\tau}$ for $\tau'\le\tau$ and $\rho\le1$, we may assume $\tau\le\tau_0$. Put
\[
\varsigma:=\min\Big\{\frac{\chi_0}{4},\ \delta n\Big\}>0,\qquad
\tau_k:=\max\Big\{\tau_0-k\varsigma,\ \frac{\tau}{2}\Big\},\qquad
\varrho_k:=2^{-k}R_\flat\qquad(k=0,1,2,\dots),
\]
so that $\tau_{k+1}=\max\{\tau_k-\varsigma,\frac{\tau}{2}\}\le\tau_k\le\tau_0$. Let $k_\tau$ be the least integer with $\tau_{k_\tau}=\frac{\tau}{2}$ (thus $1\le k_\tau\le1+\frac{\tau_0}{\varsigma}$) and put $R_\tau:=\varrho_{k_\tau}$. We prove by induction on $k=0,\dots,k_\tau$ that there are constants $C_k$, depending only on the quantities listed in the statement, such that
\begin{equation}\label{eq:Pk}
\psi_y(r)\le C_k\,r^{\,n-\tau_k}\qquad\forall y\in\mathcal Y,\ \forall r\in(0,\varrho_k]\tag{$\mathrm{P}_k$}
\end{equation}
and, for $k\ge1$, that $\varphi_y(r)\le C_kr^{\,n-\tau_k}$ in the same range. For $k=0$ this is hypothesis (b). For $k=k_\tau$ we obtain $\psi_y(\rho)+\varphi_y(\rho)\le C\rho^{\,n-\frac{\tau}{2}}\le C_\tau\rho^{\,n-\tau}$ on $(0,R_\tau]$, which is \eqref{eq:boot}.

Let now $0\le k<k_\tau$ and assume $(\mathrm{P}_k)$.

\textbf{Step 1: decay for $G(y,\cdot)$.}
Fix $y\in\mathcal Y$, $R\in(0,\varrho_{k+1}]$ and $0<\rho\le\frac{R}{4}$. Since $2R\le\varrho_k\le R_\flat$, hypothesis (d) applies. By $(\mathrm{P}_k)$ at scale $2R\le\varrho_k$ we have $\psi_y(2R)\le C_k(2R)^{n-\tau_k}$, hence, since $\tau_k\le\tau_0$, $\tau_0\delta\le\frac{\chi_0}{2}$ and $R\le1$,
\[
R^{\chi_0-\frac{n\delta}{2}}\psi_y(2R)^{1+\frac{\delta}{2}}
\le CR^{\,n-\tau_k+\chi_0-\tau_k\frac{\delta}{2}}
\le CR^{\,n-\tau_k+\frac{3\chi_0}{4}}
\le CR^{\,n-\tau_k+\frac{\chi_0}{2}}.
\]
Adding $|E_\rho(y)|\le|B_\rho|$ to both sides of (d) (recall $\varphi_y(\rho)=|E_\rho(y)|+\int_{E_\rho(y)}G(y,|Du|)\, \mathrm{d}x$), we obtain
\[
\varphi_y(\rho)
\le C\left(\frac{\rho}{R}\right)^n\varphi_y(R)
+C\rho^n
+CR^{\,n-\tau_k+\frac{\chi_0}{2}},
\qquad 0<\rho\le \frac{R}{4},
\]
and, $\varphi_y$ being nondecreasing, the same inequality (with a larger constant) holds for $\frac{R}{4}<\rho\le R$. Put
\[
b:=\min\Big\{n-\tau_k+\frac{\chi_0}{2},\ n-\frac{\tau}{2}\Big\}.
\]
Then $0<b<n$, $\rho^n\le R^b$ and $R^{\,n-\tau_k+\frac{\chi_0}{2}}\le R^b$ (recall $R\le1$), so that $\varphi_y(\rho)\le C\bigl(\frac{\rho}{R}\bigr)^n\varphi_y(R)+CR^b$ for $0<\rho\le R\le\varrho_{k+1}$. Lemma~\ref{lem:iter}, with $a=n$, $\epsilon=0$ and $\bar R=\varrho_{k+1}$, gives
\[
\varphi_y(\rho)
\le C\left[
\left(\frac{\rho}{R}\right)^b\varphi_y(R)
+\rho^b
\right],\qquad 0<\rho\le R\le\varrho_{k+1}.
\]
We take $R=\varrho_{k+1}$. The quantity $\varphi_y(\varrho_{k+1})$ is bounded uniformly in $y$: by (a) with $\kappa=0$ and $\omega(\varrho_{k+1})\le g^-\delta'$ we have $G(y,t)\le C\bigl(1+G(x,t)\bigr)^{1+\delta'}$ for $x\in E_{\varrho_{k+1}}(y)$, so that, by (c) with $s=\delta'$ and by $(\mathrm{P}_k)$ at scale $2\varrho_{k+1}=\varrho_k$,
\[
\varphi_y(\varrho_{k+1})\le C\int_{E_{\varrho_{k+1}}(y)}\bigl(1+G(x,|Du|)\bigr)^{1+\delta'}\, \mathrm{d}x
\le C\varrho_{k+1}^{-n\delta'}\psi_y(\varrho_k)^{1+\delta'}\le C .
\]
Hence $\varphi_y(\rho)\le C\rho^b$ for $\rho\in(0,\varrho_{k+1}]$. Since $\varsigma\le\frac{\chi_0}{4}$, we have $b\ge\min\{n-\tau_k+\varsigma,\,n-\frac{\tau}{2}\}=n-\tau_{k+1}$, and therefore
\begin{equation}\label{eq:phi-step}
\varphi_y(\rho)\le C\rho^{\,n-\tau_{k+1}}\qquad\forall y\in\mathcal Y,\ \forall\rho\in(0,\varrho_{k+1}].
\end{equation}

\textbf{Step 2: passage from $G(y,\cdot)$ to $G$.}
Fix $y\in\mathcal Y$ and $0<\rho\le\varrho_{k+1}$ (so that $2\rho\le\varrho_k$). Put
\[
m:=\frac{2n}{g^-},\qquad T:=\rho^{-m}\ (\ge1),
\]
and split $E_\rho(y)$ into $\{|Du|\le T\}$ and $\{|Du|>T\}$. By (a) and \eqref{eq:cf-raw}, $G(x,t)\le F+Kt^{\omega(\rho)}G(y,t)$ for $x\in E_\rho(y)$ and $t\ge0$. On $\{|Du|\le T\}$ we have $t^{\omega(\rho)}\le e^{mL_0}$ for $1\le t\le T$ by Remark~\ref{rem:omega}(2), while $G(x,t)\le F$ for $t\le1$; hence $G(x,|Du|)\le C\bigl(1+G(y,|Du|)\bigr)$ there, and
\[
\int_{E_\rho(y)\cap\{|Du|\le T\}}\bigl(1+G(x,|Du|)\bigr)\, \mathrm{d}x\le C\varphi_y(\rho).
\]
On $\{|Du|>T\}$, Proposition~\ref{zoo}(2) gives $G(x,|Du|)\ge F^{-1}T^{g^-}=F^{-1}\rho^{-2n}$, and consequently
\[
G(x,|Du|)\le\bigl(1+G(x,|Du|)\bigr)^{1+\delta}\,G(x,|Du|)^{-\delta}
\le F^{\delta}\rho^{2n\delta}\bigl(1+G(x,|Du|)\bigr)^{1+\delta}.
\]
By (c) with $s=\delta$ at scale $\rho$, and by $(\mathrm{P}_k)$ at scale $2\rho\le\varrho_k$,
\[
\begin{aligned}
\int_{E_\rho(y)\cap\{|Du|>T\}}G(x,|Du|)\, \mathrm{d}x
&\le C\rho^{2n\delta}\int_{E_\rho(y)}\bigl(1+G(x,|Du|)\bigr)^{1+\delta}\, \mathrm{d}x
\le C\rho^{2n\delta}\rho^{-n\delta}\psi_y(2\rho)^{1+\delta}\\
&\le C\rho^{\,n\delta+(n-\tau_k)(1+\delta)}
=C\rho^{\,n-\tau_k+\delta(2n-\tau_k)}
\le C\rho^{\,n-\tau_k+\varsigma},
\end{aligned}
\]
where we used $\rho\le1$ and $\varsigma\le\delta n\le\delta(2n-\tau_k)$ (recall $\tau_k\le\tau_0\le n$; the factor $\rho^{n\delta}$ coming from $\rho^{2n\delta}\rho^{-n\delta}$ is what makes the case $\tau_0=n$ admissible). Since $\psi_y(\rho)=|E_\rho(y)|+\int_{E_\rho(y)}G(x,|Du|)\, \mathrm{d}x$ and $|E_\rho(y)|\le\varphi_y(\rho)$, we conclude, with \eqref{eq:phi-step}, that
\[
\psi_y(\rho)\le C\varphi_y(\rho)+C\rho^{\,n-\tau_k+\varsigma}\le C\rho^{\,n-\tau_{k+1}},
\]
uniformly in $y\in\mathcal Y$, because $n-\tau_k+\varsigma\ge n-\tau_{k+1}$. Together with \eqref{eq:phi-step}, this proves $(\mathrm{P}_{k+1})$ and completes the proof.
\end{proof}

\begin{lemma}[Exponents of the Campanato step]\label{Lem:expo}
Let $n\ge2$, $\chi_0>0$, $\lambda\in(0,\frac{1}{2}]$ and $\kappa_0\ge1$, and put
\[
\tau:=\frac{\chi_0\lambda}{4(n+\lambda)},\qquad
\theta:=\frac{\chi_0}{2(n+\lambda)},\qquad
\varepsilon:=\frac{\tau}{1+\theta}.
\]
Then $\theta\lambda=2\tau$, $n\theta+2\tau=\frac{\chi_0}{2}$ and $\varepsilon<\tau\le\frac{\chi_0}{20}$. Moreover, if $0<R\le1$ and $\rho:=\kappa_0^{-1}R^{1+\theta}$, then $\frac{\rho}{R}\le R^\theta$, $R^\tau=(\kappa_0\rho)^{\varepsilon}$ and
\begin{equation}\label{eq:expo}
\rho^n\Big(\frac{\rho}{R}\Big)^{\lambda}R^{-\tau}\le\rho^nR^{\tau},\qquad
R^{\,n-\tau+\frac{\chi_0}{2}}\le\kappa_0^{\,n}\rho^nR^{\tau}.
\end{equation}
Finally, let $G_c=G(x_c,\cdot)$ with $G$ satisfying \eqref{D22} and \eqref{GG1}, let $0<\rho\le1$, let $E\subset B_\rho$ be measurable and let $h\ge0$ be a measurable function on $E$ such that $\int_EG_c(h)\, \mathrm{d}x\le M\rho^{\,n+\varepsilon}$. Then
\begin{equation}\label{eq:conv}
\int_Eh^{g^-}\, \mathrm{d}x\le C\rho^{\,n+\frac{\varepsilon g^-}{g^+}},\qquad C=C(n,g^\pm,F,M).
\end{equation}
\end{lemma}

\begin{proof}
We have $\theta\lambda=\frac{\chi_0\lambda}{2(n+\lambda)}=2\tau$ and $n\theta+2\tau=\frac{\chi_0(n+\lambda)}{2(n+\lambda)}=\frac{\chi_0}{2}$. Since $\theta>0$, $\varepsilon<\tau$; since $\lambda\mapsto\frac{\lambda}{n+\lambda}$ is increasing, $\frac{\lambda}{n+\lambda}\le\frac{\frac{1}{2}}{n+\frac{1}{2}}\le\frac{1}{5}$, hence $\tau\le\frac{\chi_0}{20}$. If $\rho=\kappa_0^{-1}R^{1+\theta}$, then $\frac{\rho}{R}=\kappa_0^{-1}R^{\theta}\le R^\theta$, $\kappa_0\rho=R^{1+\theta}$ gives $R^\tau=(\kappa_0\rho)^{\frac{\tau}{1+\theta}}$, and $R^n=\kappa_0^n\rho^nR^{-n\theta}$. Therefore
\[
\rho^n\Big(\frac{\rho}{R}\Big)^{\lambda}R^{-\tau}\le\rho^nR^{\theta\lambda-\tau}=\rho^nR^{\tau},\qquad
R^{\,n-\tau+\frac{\chi_0}{2}}=\kappa_0^n\rho^nR^{-n\theta-\tau+\frac{\chi_0}{2}}=\kappa_0^n\rho^nR^{\tau},
\]
because $\theta\lambda=2\tau$ and $n\theta+2\tau=\frac{\chi_0}{2}$. For \eqref{eq:conv}, Proposition~\ref{zoo}(2) gives $h^{g^-}\le FG_c(h)$ on $\{h>1\}$ and $h^{g^+}\le FG_c(h)$ on $\{h\le1\}$. Thus $\int_{E\cap\{h>1\}}h^{g^-}\, \mathrm{d}x\le FM\rho^{n+\varepsilon}$ and, by H\"older's inequality,
\[
\int_{E\cap\{h\le1\}}h^{g^-}\, \mathrm{d}x
\le\Big(\int_{E\cap\{h\le1\}}h^{g^+}\, \mathrm{d}x\Big)^{\frac{g^-}{g^+}}|B_\rho|^{1-\frac{g^-}{g^+}}
\le\bigl(FM\rho^{n+\varepsilon}\bigr)^{\frac{g^-}{g^+}}\bigl(|B_1|\rho^n\bigr)^{1-\frac{g^-}{g^+}}\le C\rho^{\,n+\frac{\varepsilon g^-}{g^+}} .
\]
Since $\rho\le1$ and $\frac{\varepsilon g^-}{g^+}\le\varepsilon$, \eqref{eq:conv} follows.
\end{proof}

\subsection{Proof of Theorem \ref{Thm1.1}}\label{subsec:proof1.1}
We begin by establishing a higher integrability result for bounded generalized solutions of \eqref{P}, which is stated in Lemma~\ref{Lem1.1} below. Its proof relies only on a Caccioppoli inequality, on the Sobolev--Poincar\'e inequality (Proposition~\ref{prop:SP}) and on Gehring's lemma (Proposition~\ref{prop:Gehring}); the continuity of $u$ given by Proposition~\ref{prop1.1} is used only afterwards, through the constant $L_2$ of \eqref{eq3.88}. Throughout this subsection, unless otherwise stated, ``depending only on $\mathrm{data}$'' refers to the parameters of Subsection~\ref{subsec:main}.

\begin{lemma}\label{Lem1.1}
Let assumptions \eqref{6.31}--\eqref{6.41}, \eqref{H111}--\eqref{H211}, \eqref{6.51}, \eqref{D22}--\eqref{GG6b} be satisfied and let $u$ be a bounded generalized solution of \eqref{P} satisfying \eqref{M}. There are positive constants $r_*\le\frac{R_0}{2}$, $c$, $c_0$ and $\delta_{\mathrm{hi}}$, depending only on $\mathrm{data}$, such that for every ball $B_{2R} \subset \Omega$ with $R \in (0, r_*]$ and every $\delta \in (0, \delta_{\mathrm{hi}}]$:
\begin{gather}
\int_{B_R}G(x,|Du|)\, \mathrm{d}x\le c\int_{B_{2R}}G\Big(x,\frac{|u-\lbrace u\rbrace_{2R}|}{R}\Big)\, \mathrm{d}x+c|B_{2R}|,\label{4.36}\\
\Big( \fint_{B_R} G(x,|Du|)^{(1+\delta)}\, \mathrm{d}x \Big)^{\frac{1}{1+\delta}} \leq c_0 \Big( 1 + \fint_{B_{2R}} G(x,|Du|)\, \mathrm{d}x \Big). \label{3.1}
\end{gather}
Equivalently,
\[
\int_{B_R}\big(1+G(x,|Du|)\big)^{1+\delta}\, \mathrm{d}x\le cR^{-n\delta}\Theta_R^{1+\delta},\qquad\text{with }\Theta_R:=\int_{B_{2R}}(1+G(x,|Du|))\, \mathrm{d}x.
\]
\end{lemma}

\begin{proof}
Let $B_{2R}\subset\Omega$ with $R\le r_*$, where $r_*:=\frac{1}{2}\min\{r_0,R_0\}$, $r_0$ being the radius of Proposition~\ref{prop:SP} and $R_0$ that of \eqref{GG3}. Let $\xi \in C_0^\infty(B_{2R})$ with $0 \leq \xi \leq 1$, $\xi = 1$ on $B_R$, $|D\xi| \leq \frac{4}{R}$, and $\displaystyle{\omega=\lbrace u\rbrace_{2R} = \fint_{B_{2R}} u\, \mathrm{d}x}$. Taking  $v \equiv \xi^{g^+}( u - \omega)\in W_0^{1,G}\cap L^\infty$ (since $u\in L^\infty$ and $\xi$ has compact support) as a test function in \eqref{FDV}, we obtain
\begin{align}\label{eq55}
      \int_{B_{2R}} \xi^{g^+} A(x, u, Du) \cdot Du\, \mathrm{d}x &+ \int_{B_{2R}} g^+\xi^{g^+-1} ( u - \omega) A(x, u, Du) \cdot D\xi\, \mathrm{d}x \\
      &= \int_{B_{2R}} B(x, u, Du) \xi^{g^+} ( u - \omega)\, \mathrm{d}x.    
\end{align}
Using \eqref{eqc1}--\eqref{eqc2} and the growth of $B$ given by Lemma~\ref{lem:B} (applied with $\varepsilon_B$ chosen below), we get
\begin{align}\label{4.31}
\begin{split}
\lambda_1 \int_{B_{2R}} G(x,|Du|) \xi^{g^+}\, \mathrm{d}x &\leq g^+\Lambda_1 \int_{B_{2R}} g(x,|Du|) \xi^{g^+-1} |D\xi|\,| u - \omega|\, \mathrm{d}x \\
&\quad + \int_{B_{2R}} \xi^{g^+}\big(\varepsilon_B G(x,|Du|)+C_{B,\varepsilon_B}\big)| u - \omega|\, \mathrm{d}x.
\end{split}
\end{align}
Choose $\varepsilon_B>0$ so that $2M\varepsilon_B\le\frac{\lambda_1}{4}$ and use $|u-\omega|\le2M$ for the second term, giving 
\begin{equation}\label{Eqa10}
    \frac{\lambda_1}{4}\int_{B_{2R}}\xi^{g^+}G(x,|Du|)\, \mathrm{d}x+2MC_{B,\varepsilon_B}|B_{2R}|.
 \end{equation}
    For the first term, Young's inequality \eqref{Yi} with $\varepsilon_1 \in (0,1)$ chosen so that $g^+(g^+-1)\Lambda_1\, \varepsilon_1^{\frac{g^+}{g^+-1}} \leq \frac{\lambda_1 }{4}$, together with Proposition~\ref{zoo}(3) (note that $\varepsilon_1\xi^{g^+-1}\le1$ and $(g^+-1)\,\widetilde{g^-}=g^+$), \eqref{L1} and $|D\xi|\le\frac{4}{R}$, gives
\begin{align}\label{Eqa11}
\begin{split}
&g^+\Lambda_1 \int_{B_{2R}} g(x,|Du|) \xi^{g^+-1} |D\xi|\,|u - \omega|\, \mathrm{d}x \\
&\quad \leq g^+\Lambda_1  \int_{B_{2R}} \widetilde{G}\left(x, g(x,|Du|) \xi^{g^+-1}\varepsilon_1\right)\, \mathrm{d}x + g^+\Lambda_1  \int_{B_{2R}} G\left(x, \frac{|D\xi||u - \omega|}{\varepsilon_1}\right)\, \mathrm{d}x \\
&\quad \leq g^+(g^+-1)\Lambda_1\varepsilon_1^{\frac{g^+}{g^+-1}} \int_{B_{2R}} G(x,|Du|) \xi^{g^+}\, \mathrm{d}x + d_1 \int_{B_{2R}} G\left(x, \frac{|u - \omega|}{R} \right)\, \mathrm{d}x\\
&\quad \leq \frac{\lambda_1 }{4} \int_{B_{2R}} G(x,|Du|) \xi^{g^+}\, \mathrm{d}x + d_1 \int_{B_{2R}} G\left(x, \frac{|u - \omega|}{R} \right)\, \mathrm{d}x,
\end{split}
\end{align}
with \(d_1 := g^+\Lambda_1\,\bigl(\frac{4}{\varepsilon_1}\bigr)^{g^+}\), which depends only on $\Lambda_1,\lambda_1,g^+$ through the choice of $\varepsilon_1$ (we used Proposition~\ref{zoo}(1) with $s=\frac{4}{\varepsilon_1}\ge1$ in the last integral of the second line).\\
Combining the two estimates \eqref{Eqa10}-\eqref{Eqa11} and absorbing the $\frac{\lambda_1}{4}$ into the left side of \eqref{4.31}, we obtain
\begin{equation}\label{3.2m1}
\int_{B_{R}} G(x,|Du|)\, \mathrm{d}x \leq c_2 \int_{B_{2R}} G\left(x, \frac{\vert u - \omega\vert}{R} \right)\, \mathrm{d}x + c_2 |B_{2R}|,
\end{equation}
where $c_2:=c_2(g^-,g^+,\Lambda_1,\lambda_1,M,\mathrm{data}_B)>0$ (the dependence on $M$ and $\mathrm{data}_B$ comes from the choice of $\varepsilon_B$ and from $C_{B,\varepsilon_B}$ in Lemma~\ref{lem:B}).\\
By the Sobolev-Poincar\'e inequality Proposition \ref{prop:SP} (applied in $B_{2R}$, which has radius $2R\le 2r_*\le r_0$, so that it yields $G\bigl(x,\frac{|u-\omega|}{2R}\bigr)$ on the left-hand side) and by $G\bigl(x,\frac{s}{R}\bigr)\le 2^{g^+}G\bigl(x,\frac{s}{2R}\bigr)$ (Proposition~\ref{zoo}(1)), there are $\varepsilon \in (0, 1)$ and $c_3\ge1$, depending only on $n,g^\pm,F,\mu$, such that, for $R \in (0, r_*]$,
\begin{equation}\label{4.37}
 \fint_{B_{2R}} G\left(x, \frac{\vert u - \omega\vert}{R} \right)\, \mathrm{d}x \leq c_3+c_3\left( \fint_{B_{2R}} G\left(x, \vert D u\vert \right) ^{\frac{1}{1+\varepsilon}}\, \mathrm{d}x\right)^{1+\varepsilon}.
\end{equation}
Combining \eqref{3.2m1} and \eqref{4.37}, we get
\begin{equation}\label{4.38}
\fint_{B_{R}} G(x,|Du|)\, \mathrm{d}x \leq c+c\left( \fint_{B_{2R}} G\left(x, \vert D u\vert \right) ^{\frac{1}{1+\varepsilon}}\, \mathrm{d}x\right)^{1+\varepsilon},
\end{equation}
which is a reverse H\"older inequality; here $c=c(\mathrm{data})$, and Gehring's lemma (Proposition~\ref{prop:Gehring} with $s=1+\varepsilon$, $f=G(\cdot,|Du|)^{\frac{1}{s}}$, $\varphi\equiv c^{\frac{1}{s}}$ and $b=c$) gives \eqref{3.1} for $\delta\le\delta_{\mathrm{hi}}:=\delta_G(n,1+\varepsilon,c)$, and the equivalent form follows from \eqref{3.1}, from $(1+G)^{1+\delta}\le 2^{\delta}(1+G^{1+\delta})$ and from $\Theta_R=|B_{2R}|\big(1+\fint_{B_{2R}}G(x,|Du|)\, \mathrm{d}x\big)$:
\[
|B_R|\fint_{B_R}\big(1+G(x,|Du|)\big)^{1+\delta}\, \mathrm{d}x\le c|B_R|\left(1+\fint_{B_{2R}}G(x,|Du|)\, \mathrm{d}x\right)^{1+\delta}\le cR^{-n\delta}\Theta_R^{1+\delta}.
\]
\end{proof}

\medskip\noindent\textbf{Choice of $\Omega_0$, $L_2$, $x_0$ and $R_1$.}
Fix $\Omega_0\Subset\Omega$, put $d_0:=\mathrm{dist}(\Omega_0,\partial\Omega)>0$ and $\Omega_0':=\{x\in\Omega:\ \mathrm{dist}(x,\partial\Omega)>\frac{d_0}{2}\}$, so that $\Omega_0\subset\Omega_0'\Subset\Omega$ and $\mathrm{dist}(\Omega_0',\partial\Omega)=\frac{d_0}{2}$. By Proposition~\ref{prop1.1}(1) applied to $\Omega_0'$, $u\in C^{0,\alpha_1}(\overline{\Omega_0'})$; decreasing $\alpha_1$ if necessary we assume $\alpha_1<1$. Put $L_2:=\|u\|_{C^{0,\alpha_1}(\overline{\Omega_0'})}$, which depends only on $\mathrm{data}$ and $d_0$; thus
\begin{equation}\label{eq3.88}
|u(x_1) - u(x_2)| \leq L_2 |x_1 - x_2|^{\alpha_1} \quad \text{for all } x_1, x_2 \in \Omega_0'.
\end{equation}
Let $x_0\in\overline{\Omega_0}$ and $R_1:=\frac{1}{2}\min\{r_*,\frac{d_0}{4}\}$, where $r_*$ is the radius of Lemma~\ref{Lem1.1}. Then $B_{2R_1}(x_0)\subset\Omega_0'$, so \eqref{eq3.88} holds in $B_{2R_1}(x_0)$, and $R_1$ depends only on $\mathrm{data}$ and $d_0$.

For a ball $B_R(x_c)\subset B_{2R}(x_c)\subset B_{2R_1}(x_0)$ we freeze at the fixed centre $x_c$ and not at a $t$-dependent point 
\[
\overline A(\eta):=A(x_c,u(x_c),\eta),\qquad G_c(t):=G(x_c,t),\qquad \varphi_c(t):=\frac{g_c(t)}{t},\quad g_c(t):=\partial_tG_c(t) .
\]
From \eqref{D22diff} $G_c$ is an N-function with
\[
g^--1\le \frac{tg_c'(t)}{g_c(t)}\le g^+-1;
\]
$\overline A\in C^1(\mathbb{R}^n\setminus\{0\})$ satisfies \eqref{6.31}--\eqref{6.41} with $g_c$ and constants $\lambda_0,\Lambda_0$, and $\overline A(0)=0$. Consider
\begin{equation}\label{FP}
\begin{cases}
- \operatorname{div} \overline{A}(Dv) = 0 & \text{in } B_R, \\
v = u & \text{on } \partial B_R.
\end{cases} \tag{$\mathcal{FP}_1$}
\end{equation}

\begin{remark}
The freezing point depends on the ball (hence on $R$) but not on $t$; this is essential for the comparison Lemma~\ref{Lem1.3} below. 
\end{remark}

The following lemma collects classical results for the frozen problem, i.e. for a quasilinear equation with $x$-independent Orlicz structure; see Lieberman \cite[Lemmas 5.1 and 5.2]{Li1991} and Antonini \cite[Theorem 4.1]{ACAN}. It is applied with $G_c$ in place of $G$; all the constants depend only on $n,g^\pm,\lambda_0,\Lambda_0$.

\begin{lemma}\label{Lem1.2}
Let $B_R=B_R(x_c)$ and let $u\in W^{1,G_c}(B_R)\cap C(\overline{B_R})$. Then \eqref{FP} has a unique solution $v\in u+W_0^{1,G_c}(B_R)$. Moreover $v\in C^{1,\lambda_2}_{\mathrm{loc}}(B_R)\cap C(\overline{B_R})$ and $v=u$ on $\partial B_R$, where $\lambda_2=\lambda_2(n,g^\pm,\lambda_0,\Lambda_0)\in(0,1)$, and
\begin{equation}\label{eq3.10}
\sup_{x \in B_{\frac{R}{2}}} G_c( |Dv(x)|) \leq c_1 R^{-n} \int_{B_R} G_c( |Dv|)\, \mathrm{d}x,
\end{equation}
\begin{equation}\label{eq3.11}
\fint_{B_{\rho}} G_c( |Dv(x) - \lbrace Dv\rbrace_\rho|)\, \mathrm{d}x \leq c_2 \left( \frac{\rho}{R} \right)^{\lambda_2} \fint_{B_R} G_c( |Dv(x)|)\, \mathrm{d}x, \quad \forall \rho \in \Big(0, \frac{R}{2}\Big],
\end{equation}
\begin{equation}\label{eq3.12}
\int_{B_R} G_c( |Dv(x)|)\, \mathrm{d}x \leq c_3 \int_{B_R} \big( 1 + G_c( |Du(x)|) \big)\, \mathrm{d}x,
\end{equation}
\begin{equation}\label{eq3.13}
\sup_{B_R} |u - v| \leq \operatorname{osc}_{B_R} u,
\end{equation}
where \( c_1, c_2, c_3 \) are positive constants depending only on \( n,g^-, g^+, \lambda_0, \Lambda_0 \), and
\[
\lbrace Dv\rbrace_\rho := \frac{1}{|B_{\rho}|} \int_{B_{\rho}} Dv(x)\, \mathrm{d}x.
\]
\end{lemma}
\begin{proof}
Existence and uniqueness of $v\in u+W_0^{1,G_c}(B_R)$ follow from the Browder--Minty theorem, since $\overline A$ is bounded, coercive and strictly monotone on $W^{1,G_c}(B_R)$ by \eqref{6.31}, \eqref{6.41} and Proposition~\ref{Prop2.2} (applied to $\overline A$ with $G_c$ in place of $G$). The continuity of $v$ up to $\partial B_R$, with $v=u$ there, and the maximum principle $\inf_{\partial B_R}u\le v\le\sup_{\partial B_R}u$, which gives \eqref{eq3.13}, are as in \cite[Lemma 5.1]{Li1991}. The energy bound \eqref{eq3.12} follows by testing the equation with $v-u$ and using Young's inequality \eqref{Yi} together with \eqref{L1}. Finally, \eqref{eq3.10} and the interior H\"older estimate for the gradient
\begin{equation}\label{eq:holder-grad}
\operatorname*{osc}_{B_\rho}Dv\le C_L\Big(\frac{\rho}{R}\Big)^{\lambda_L}\sup_{B_{\frac{R}{2}}}|Dv|,\qquad 0<\rho\le\frac{R}{2},
\end{equation}
with $C_L\ge1$ and $\lambda_L\in(0,1)$ depending only on $n,g^\pm,\lambda_0,\Lambda_0$, are the interior $C^{1,\lambda}$-estimates of \cite[Lemma 5.2]{Li1991} (see also \cite[Theorem.~4.1]{ACAN}); since $\overline A$ does not depend on $x$, the equation is invariant under $v\mapsto R^{-1}v(x_c+R\,\cdot)$, which is why the constants do not depend on $R$. Estimate \eqref{eq3.11} is a consequence of \eqref{eq:holder-grad} and \eqref{eq3.10}: since $|Dv-\lbrace Dv\rbrace_\rho|\le\operatorname*{osc}_{B_\rho}Dv$ in $B_\rho$, Proposition~\ref{zoo}(1) gives, with $S:=\sup_{B_{\frac{R}{2}}}|Dv|$,
\[
\begin{aligned}
\fint_{B_\rho}G_c(|Dv-\lbrace Dv\rbrace_\rho|)\, \mathrm{d}x
&\le G_c\Big(C_L\Big(\frac{\rho}{R}\Big)^{\lambda_L}S\Big)
\le C_L^{g^+}\Big(\frac{\rho}{R}\Big)^{\lambda_Lg^-}G_c(S)\\
&\le c_1|B_1|\,C_L^{g^+}\Big(\frac{\rho}{R}\Big)^{\lambda_Lg^-}\fint_{B_R}G_c(|Dv|)\, \mathrm{d}x ,
\end{aligned}
\]
where \eqref{eq3.10} was used in the last step. Thus \eqref{eq3.11} holds with $\lambda_2:=\min\{\lambda_Lg^-,\tfrac{1}{2}\}$ and $c_2:=c_1|B_1|C_L^{g^+}$.
\end{proof}

\begin{lemma}\label{Lem1.4a}
Let $r_*$ and $\delta_{\mathrm{hi}}$ be as in Lemma~\ref{Lem1.1}, let $\delta_*=\delta_*(n,g^\pm,F,\lambda_0,\Lambda_0)>0$ be the Gehring exponent of the frozen problem defined in the proof below, and set
\[
\delta':=\min\Big\{\frac{\delta_{\mathrm{hi}}}{4},\,\delta_*\Big\},\qquad
R_\omega:=\min\Big\{r_*,\ \exp\Big(-\frac{L_0}{g^-\delta'}\Big)\Big\},
\]
so that $\omega(R)\le g^-\delta'$ for $R\le R_\omega$; $R_\omega$ and $\delta'$ depend only on $\mathrm{data}$. Let $B_{2R}(x_c)\subset B_{2R_1}(x_0)$ with $R\le R_\omega$. Then $u\in W^{1,G_c}(B_R)\cap C(\overline{B_R})$, so that the solution $v$ of \eqref{FP} given by Lemma~\ref{Lem1.2} exists; moreover $v-u\in W_0^{1,G}(B_R)\cap L^\infty(B_R)$ and
\begin{equation}\label{eq:v-hi}
\int_{B_R}G_c(|Dv|)^{1+\delta'}\, \mathrm{d}x\le C\int_{B_R}(1+G_c(|Du|))^{1+\delta'}\, \mathrm{d}x ,
\end{equation}
where $C=C(n,g^\pm,F,\lambda_0,\Lambda_0)$.
\end{lemma}

\begin{proof}
\textbf{Step 0: the frozen problem is well posed.} By Lemma~\ref{lem:comp-factor} (with $\kappa=0$) and $\omega(R)\le g^-\delta'$,
\[
G_c(|Du|)\le C\bigl(1+G(x,|Du|)\bigr)^{1+\frac{\omega(R)}{g^-}}\le C\bigl(1+G(x,|Du|)\bigr)^{1+\delta'}\qquad\text{in }B_R,
\]
and the right-hand side is integrable on $B_R$ by Lemma~\ref{Lem1.1}, since $\delta'\le\delta_{\mathrm{hi}}$ and $B_{2R}(x_c)\subset\Omega$. As $u$ is H\"older continuous on $\overline{B_R}$ by \eqref{eq3.88}, we have $u\in W^{1,G_c}(B_R)\cap C(\overline{B_R})$, and Lemma~\ref{Lem1.2} yields the solution $v$ of \eqref{FP} with
\[
v-u\in W_0^{1,G_c}(B_R)\cap C(\overline{B_R}),\qquad v=u\ \text{on }\partial B_R,\qquad
\|v-u\|_{L^\infty(B_R)}\le\operatorname*{osc}_{B_R}u\le 2M,\qquad |v|\le M .
\]

\textbf{Step 1: a reverse H\"older inequality.} Let $s:=1+\varepsilon$, where $\varepsilon\in(0,1)$ is the exponent given by Proposition~\ref{prop:SP} for the $x$-independent $\textnormal{N}$-function $G_c$ (for which the additive constant of Proposition~\ref{prop:SP} is not needed), and let $r_0$ be the corresponding radius. Put $Q:=B_{2R}(x_c)$, extend $v-u$ and $Dv$ by zero outside $B_R$ (recall that $v-u\in W_0^{1,G_c}(B_R)$, so that the extension of $v-u$ belongs to $W^{1,G_c}(\mathbb{R}^n)$) and set, on $Q$,
\[
f:=G_c(|Dv|)^{\frac{1}{s}},\qquad
\varphi:=
\begin{cases}
\bigl(1+G_c(|Du|)\bigr)^{\frac{1}{s}}&\text{in }B_R,\\
1&\text{in }Q\setminus B_R.
\end{cases}
\]
Then $f\in L^s(Q)$ and, by Step 0, $\varphi\in L^{s(1+\delta')}(Q)$. We claim that there is $b_*=b_*(n,g^\pm,F,\lambda_0,\Lambda_0)\ge1$ such that
\begin{equation}\label{eq:RH-frozen}
\fint_{B_r(z)}f^s\, \mathrm{d}x
\le b_*\Big(\fint_{B_{4r}(z)}f\, \mathrm{d}x\Big)^s
+b_*\fint_{B_{4r}(z)}\varphi^s\, \mathrm{d}x
\qquad\text{whenever }B_{4r}(z)\subset Q .
\end{equation}
Note that $B_{4r}(z)\subset Q$ forces $r\le\frac{R}{2}$, so that $4r\le2R\le2r_*\le r_0$. Fix such $z$ and $r$, and a cut-off function $\zeta\in C_c^\infty(B_{2r}(z))$ with $0\le\zeta\le1$, $\zeta=1$ on $B_r(z)$ and $|D\zeta|\le\frac{2}{r}$. Three cases are possible.

\textit{(i) $B_{2r}(z)\cap B_R=\emptyset$.} Then $f=0$ on $B_r(z)$ and \eqref{eq:RH-frozen} is trivial.

\textit{(ii) $B_{2r}(z)\subset B_R$.} We test the equation of $v$ with $\vartheta:=\zeta^{g^+}\bigl(v-\lbrace v\rbrace_{2r,z}\bigr)\in W_0^{1,G_c}(B_R)$. By \eqref{6.31}--\eqref{6.41} for $\overline A$ (with $g_c$), Young's inequality \eqref{Yi}, \eqref{L1} and Proposition~\ref{zoo}(3), exactly as in the proof of Lemma~\ref{Lem1.1}, we obtain
\[
\int_{B_r(z)}G_c(|Dv|)\, \mathrm{d}x\le C\int_{B_{2r}(z)}G_c\Big(\frac{|v-\lbrace v\rbrace_{2r,z}|}{r}\Big)\, \mathrm{d}x ,
\]
and Proposition~\ref{prop:SP} in $B_{2r}(z)$ (together with $G_c\bigl(\frac{t}{r}\bigr)\le2^{g^+}G_c\bigl(\frac{t}{2r}\bigr)$) gives
\[
\fint_{B_r(z)}f^s\, \mathrm{d}x\le C\Bigl(\fint_{B_{2r}(z)}f\, \mathrm{d}x\Bigr)^s\le C\Bigl(\fint_{B_{4r}(z)}f\, \mathrm{d}x\Bigr)^s .
\]

\textit{(iii) $B_{2r}(z)$ meets both $B_R$ and $\mathbb{R}^n\setminus B_R$.} Then $B_{2r}(z)$ contains a point $z'\in\partial B_R$. We test the equation of $v$ with $\vartheta:=\zeta^{g^+}(v-u)\in W_0^{1,G_c}(B_R)$. The same computation, now with the additional term $\Lambda_1\int_{B_{2r}(z)\cap B_R}\zeta^{g^+}g_c(|Dv|)|Du|\, \mathrm{d}x$, which is absorbed by Young's inequality, gives
\[
\int_{B_r(z)}G_c(|Dv|)\, \mathrm{d}x\le C\int_{B_{2r}(z)}G_c\Big(\frac{|v-u|}{r}\Big)\, \mathrm{d}x+C\int_{B_{2r}(z)\cap B_R}G_c(|Du|)\, \mathrm{d}x .
\]
Let $w$ be the extension of $v-u$ by zero. The ball of radius $r$ centred at $z'+r\frac{z'-x_c}{R}$ lies outside $B_R$ and inside $B_{2r}(z')\subset B_{4r}(z)$; hence $w$ vanishes on a subset of $B_{4r}(z)$ of measure at least $|B_r|=4^{-n}|B_{4r}|$. Proposition~\ref{prop:SP} in $B_{4r}(z)$, in its version for functions vanishing on a set of measure $\ge\sigma|B_{4r}|$ with $\sigma=4^{-n}$, together with $G_c\bigl(\frac{t}{r}\bigr)\le4^{g^+}G_c\bigl(\frac{t}{4r}\bigr)$, yields
\[
\fint_{B_{4r}(z)}G_c\Big(\frac{|w|}{r}\Big)\, \mathrm{d}x\le C\Big(\fint_{B_{4r}(z)}G_c(|Dw|)^{\frac{1}{s}}\, \mathrm{d}x\Big)^s .
\]
Since $G_c(|Dw|)^{\frac{1}{s}}\le C(f+\varphi)$ and $G_c(|Du|)\mathbf 1_{B_R}\le\varphi^s$ on $B_{4r}(z)$, Jensen's inequality gives \eqref{eq:RH-frozen} also in this case. This proves the claim.

\textbf{Step 2: higher integrability.} Let $\delta_*:=\delta_G(n,s,b_*,4)$ be the exponent of Proposition~\ref{prop:Gehring} with $\kappa=4$; it depends only on $n,g^\pm,F,\lambda_0,\Lambda_0$. Since $\delta'\le\delta_*$, Proposition~\ref{prop:Gehring} applied in $Q$ with $\kappa=4$, $\delta_1=\delta'$ and the pairs $B_{\frac{R}{4}}(z)\subset B_R(z)\subset Q$, $z\in B_R(x_c)$, gives
\[
\Big(\fint_{B_{\frac{R}{4}}(z)}f^{s(1+\delta')}\, \mathrm{d}x\Big)^{\frac{1}{1+\delta'}}
\le c\Bigg[\fint_{B_R(z)}f^s\, \mathrm{d}x+\Big(\fint_{B_R(z)}\varphi^{s(1+\delta')}\, \mathrm{d}x\Big)^{\frac{1}{1+\delta'}}\Bigg].
\]
We cover $B_R(x_c)$ by $N=N(n)$ balls $B_{\frac{R}{4}}(z_i)$ with $z_i\in B_R(x_c)$. By \eqref{eq3.12},
\[
\fint_{B_R(z_i)}f^s\le2^n\fint_Qf^s\le C\fint_Q\varphi^s\le C\bigl(\fint_Q\varphi^{s(1+\delta')}\bigr)^{\frac{1}{1+\delta'}}
\]
and $\fint_{B_R(z_i)}\varphi^{s(1+\delta')}\le2^n\fint_Q\varphi^{s(1+\delta')}$; summing over $i$ we find
\[
\int_{B_R}G_c(|Dv|)^{1+\delta'}\, \mathrm{d}x\le C\int_Q\varphi^{s(1+\delta')}\, \mathrm{d}x\le C\int_{B_R}(1+G_c(|Du|))^{1+\delta'}\, \mathrm{d}x ,
\]
that is \eqref{eq:v-hi} (we used $|Q\setminus B_R|\le C|B_R|$).

\textbf{Step 3: $v-u\in W_0^{1,G}(B_R)$.} By Lemma~\ref{lem:comp-factor} (with $\kappa=0$) and $\omega(R)\le g^-\delta'$,
\[
G(x,|Dv|)\le C\bigl(1+G_c(|Dv|)\bigr)^{1+\frac{\omega(R)}{g^-}}\le C\bigl(1+G_c(|Dv|)\bigr)^{1+\delta'}\qquad\text{in }B_R,
\]
which is integrable by \eqref{eq:v-hi} and Step 0. As $v$ is bounded, $v\in W^{1,G}(B_R)$; hence $w:=v-u\in W^{1,G}(B_R)\cap C(\overline{B_R})$, with $w=0$ on $\partial B_R$ and $|w|\le2M$. For $k\in\mathbb{N}$ let $w_k:=\bigl(w-\frac{1}{k}\bigr)^+-\bigl(w+\frac{1}{k}\bigr)^-$. The support of $w_k$ is contained in $\{|w|\ge\frac{1}{k}\}$, a compact subset of $B_R$; hence $w_k\in W^{1,G}(B_R)$ has compact support in $B_R$, and therefore $w_k\in W_0^{1,G}(B_R)$ (mollification converges in $W^{1,G}$ for compactly supported functions under \eqref{D22}, \eqref{GG1} and \eqref{GG2}, see \cite{Harjulehto2019}). Moreover $|w_k-w|\le\frac{1}{k}$ and $Dw_k-Dw=-Dw\,\mathbf 1_{\{|w|\le\frac{1}{k}\}}\to-Dw\,\mathbf 1_{\{w=0\}}=0$ a.e., so that $w_k\to w$ in $W^{1,G}(B_R)$ by dominated convergence (recall $G(x,|Dw|)\in L^1(B_R)$ and \eqref{D22}). Consequently $v-u=w\in W_0^{1,G}(B_R)\cap L^\infty(B_R)$.
\end{proof}

\begin{lemma}\label{Lem1.3}
Let $v$ be as in Lemma \ref{Lem1.2}. Let $\beta=\min\{\beta_1,\alpha_1\beta_2\}$ and $\theta_*=\min\left\{1,\frac{g^-}{2}\right\}$. For every $\delta\in\left(0,\frac{\delta_{\mathrm{hi}}}{2}\right]$ there is $R_2=R_2(\delta)\in(0,\min\{R_1,R_\omega\}]$ such that, for every ball $B_{2R}(x_c)\subset B_{2R_1}(x_0)$ with $R\le R_2$,
\begin{equation}\label{eq3.14}
\int_{B_R} G_c(|Du(x) - Dv(x)|)\, \mathrm{d}x \leq C\,R^{\beta\theta_*-n\frac{\delta}{2}}\,\Theta_R^{1+\frac{\delta}{2}},\qquad \Theta_R:=\int_{B_{2R}}\big(1+G(x,|Du|)\big)\, \mathrm{d}x,
\end{equation}
where $C$ and $R_2$ depend only on $\mathrm{data}$, $\delta$ and $L_2$ (hence on $\alpha_1$, which is a function of $\mathrm{data}$, and on $d_0=\mathrm{dist}(\Omega_0,\partial\Omega)$).
\end{lemma}

\begin{proof}
Set
\[
X:=R^{-n\frac{\delta}{2}}\Theta_R^{1+\frac{\delta}{2}}.
\]
Since $\Theta_R\ge cR^n$, we have $X\ge c\Theta_R$.\\

 For every $\kappa\in\left[0,\frac{\delta}{2}g^-\right]$, Lemma~\ref{lem:elem}(i) yields
\[
(1+|Du|)^\kappa
\le C_\kappa(1+G(x,|Du|))^{\frac{\delta}{2}}.
\]
Consequently, by H\"older's inequality and Lemma~\ref{Lem1.1},
\begin{equation}\label{eq:step0-int}
\begin{aligned}
\int_{B_R}(1+G(x,|Du|))^{1+\frac{\delta}{2}}\, \mathrm{d}x
&\le
C R^{-n\frac{\delta}{2}}\Theta_R^{1+\frac{\delta}{2}}=CX.
\end{aligned}
\end{equation}
In particular, whenever $\kappa\in\left[0,\frac{\delta}{2}g^-\right]$,
\[
\int_{B_R}(1+|Du|)^\kappa
(1+G(x,|Du|))\, \mathrm{d}x\le CX.
\]

\textit{Choice of $\sigma$ and of $R_2$.} Put
\[
\sigma:=\min\Big\{\frac{g^-\delta}{4\,\widetilde{g^+}},\ \frac{g^--1}{2}\Big\},\qquad
\varepsilon_\delta:=\min\Big\{\frac{g^-\delta}{4(\widetilde{g^+}+1)},\ \frac{g^--1}{2}\Big\},\qquad
R_2:=\min\Big\{R_1,\ R_\omega,\ e^{-\frac{L_0}{\varepsilon_\delta}}\Big\},
\]
so that $\omega(R)\le\varepsilon_\delta$ for $R\le R_2$; $\sigma$ and $R_2$ depend only on $\mathrm{data}$, $\delta$ and $L_2$. For $R\le R_2$ and $\kappa_0:=\sigma+\omega(R)$ we have, recalling $\widetilde{g^+}=\frac{g^-}{g^--1}>1$,
\begin{equation}\label{eq:choice-sigma}
\kappa_0\,\widetilde{g^+}+\omega(R)\le\frac{g^-\delta}{2},\qquad\kappa_0\le g^--1 ;
\end{equation}
in particular $\kappa_0+\omega(R)\le\frac{g^-\delta}{2}$ and $\kappa_0+1\le g^-\bigl(1+\frac{\delta}{2}\bigr)$. These are the only restrictions on the exponents used below.

 Define
\[
I:=\int_{B_R}
\big(\overline A(Du)-\overline A(Dv)\big)
\cdot(Du-Dv)\, \mathrm{d}x.
\]
By Lemma~\ref{Lem1.4a}, the function
$u-v\in W_0^{1,G}(B_R)\cap L^\infty(B_R)$ is an admissible test
function both in \eqref{FDV} and in the weak formulation of
\eqref{FP}. Therefore,
\[
I=I_1+I_2,
\]
where
\[
\begin{aligned}
I_1&:=
\int_{B_R}
\big(\overline A(Du)-A(x,u,Du)\big)
\cdot(Du-Dv)\, \mathrm{d}x,\\
I_2&:=
\int_{B_R}B(x,u,Du)\,(u-v)\, \mathrm{d}x,
\end{aligned}
\]
\textbf{Estimate of $I_2$:}
By \eqref{eq3.13}, \eqref{eq3.88}, and Lemma~\ref{lem:B}, we get
\[
|I_2|\le CR^{\alpha_1}\Theta_R
\le CR^\beta X.
\]
\textbf{Estimate of $I_1$:}
For $x\in B_R$, we have
\[
|x-x_c|\le R,
\qquad
|u(x)-u(x_c)|\le L_2R^{\alpha_1}.
\]
Thus, by \eqref{7.41} and the elementary estimate
\[
1+\ln(e+t)\le C_\sigma(1+t)^\sigma,
\]
we obtain
\[
|\overline A(Du)-A(x,u,Du)|
\le
CR^\beta(1+|Du|)^\sigma
\big(1+g_c(|Du|)+g(x,|Du|)\big).
\]
Moreover, applying Corollary~\ref{cor:g-comp} to the pair
$(x_c,x)$ gives
\[
1+g_c(|Du|)+g(x,|Du|)
\le
C(1+|Du|)^{\omega(R)}
\big(1+g_c(|Du|)\big).
\]
Writing $\kappa_0:=\sigma+\omega(R)$ and using
$|Du-Dv|\le |Du|+|Dv|$, we infer that
\[
|I_1|\le CR^\beta(J_1+J_2),
\]
where
\[
\begin{aligned}
J_1&:=
\int_{B_R}(1+|Du|)^{\kappa_0}
\big(1+g_c(|Du|)\big)|Du|\, \mathrm{d}x,\\
J_2&:=
\int_{B_R}(1+|Du|)^{\kappa_0}
\big(1+g_c(|Du|)\big)|Dv|\, \mathrm{d}x.
\end{aligned}
\]
\textbf{Estimate of $J_1$:}
Since
\[
|Du|\big(1+g_c(|Du|)\big)
\le |Du|+g^+G_c(|Du|),
\]
we have
\[
J_1\le\int_{B_R}(1+|Du|)^{\kappa_0+1}\, \mathrm{d}x+g^+\int_{B_R}(1+|Du|)^{\kappa_0}G_c(|Du|)\, \mathrm{d}x.
\]
By Lemma~\ref{lem:elem}(i) and \eqref{eq:choice-sigma},
\[
(1+|Du|)^{\kappa_0+1}\le C\bigl(1+G(x,|Du|)\bigr)^{\frac{\kappa_0+1}{g^-}}\le C\bigl(1+G(x,|Du|)\bigr)^{1+\frac{\delta}{2}},
\]
and by Lemma~\ref{lem:comp-factor} (estimate \eqref{eq:cf-small} for the pair $(x_c,x)$, with $\kappa=\kappa_0$, which is allowed since $\frac{\kappa_0+\omega(R)}{g^-}\le\frac{\delta}{2}$ by \eqref{eq:choice-sigma})
\[
(1+|Du|)^{\kappa_0}G_c(|Du|)\le C\bigl(1+G(x,|Du|)\bigr)^{1+\frac{\delta}{2}}.
\]
Hence, by \eqref{eq:step0-int},
\[
J_1\le CX.
\]
\textbf{Estimate of $J_2$:}
Applying Young's inequality for $G_c$ with
\[
a:=(1+|Du|)^{\kappa_0}\big(1+g_c(|Du|)\big),
\]
and using Proposition~\ref{zoo}(3) and
Lemma~\ref{lem:elem}(iv), we get
\[
\begin{aligned}
J_2
&\le
C\int_{B_R}
(1+|Du|)^{\kappa_0\widetilde g^+}
\big(1+G_c(|Du|)\big)\, \mathrm{d}x+C\int_{B_R}G_c(|Dv|)\, \mathrm{d}x.
\end{aligned}
\]
The first integral is bounded by $CX$ by
Lemma~\ref{lem:elem}(i), Lemma~\ref{lem:comp-factor} (estimate \eqref{eq:cf-small} with $\kappa=\kappa_0\widetilde{g^+}$) and \eqref{eq:step0-int}, because $\frac{\kappa_0\widetilde{g^+}+\omega(R)}{g^-}\le\frac{\delta}{2}$ by \eqref{eq:choice-sigma}. The second integral satisfies
\[
\int_{B_R}G_c(|Dv|)\, \mathrm{d}x
\le C\int_{B_R}G_c(|Du|)\, \mathrm{d}x
\le CX,
\]
by \eqref{eq3.12} and \eqref{eq:step0-int}. Hence
\[
J_2\le CX.
\]

Since the exponent restrictions \eqref{eq:choice-sigma} hold for $R\le R_2$, it follows that
\[
|I_1|\le CR^\beta (J_1+J_2)\le CR^\beta X.
\]
Combining this with the estimate for $I_2$, we conclude that
\[
I\le CR^\beta X.
\]

By Lemma~\ref{lem:orlicz-ineq} and Proposition~\ref{Prop2.2},
applied to $\overline A$ with $G_c$ in place of $G$, we have
\[
\begin{aligned}
G_c(|Du-Dv|)
\le {}&
C\,G_c(|Du|+|Dv|)^{1-\theta_*}\times
\big[
(\overline A(Du)-\overline A(Dv))
\cdot(Du-Dv)
\big]^{\theta_*}.
\end{aligned}
\]
Integrating and applying H\"older's inequality with exponents
$\frac{1}{1-\theta_*}$ and $\frac{1}{\theta_*}$, we obtain
\[
\begin{aligned}
\int_{B_R}G_c(|Du-Dv|)\, \mathrm{d}x
&\le
C\left(\int_{B_R}G_c(|Du|+|Dv|)\, \mathrm{d}x\right)^{1-\theta_*}
I^{\theta_*}\\
&\le
CX^{1-\theta_*}(R^\beta X)^{\theta_*}\\
&=
CR^{\beta\theta_*}X.
\end{aligned}
\]
Here we used the energy bound
\[
\int_{B_R}G_c(|Du|+|Dv|)\, \mathrm{d}x\le CX.
\]
This proves \eqref{eq3.14}.
\end{proof}

\begin{lemma}\label{Lem1.2a}
Let $\tau_0:=\min\{(1-\alpha_1)g^+,\,n\}\in(0,n]$. For every $y\in B_{\frac{R_1}{2}}(x_0)$ and every $\rho\in\bigl(0,\frac{R_1}{2}\bigr]$,
\[
\int_{B_\rho(y)}\big(1+G(x,|Du|)\big)\, \mathrm{d}x\le C\rho^{\,n-\tau_0}.
\]
where $C$ depends only on $\mathrm{data}$ and $L_2$.
\end{lemma}

\begin{proof}
By \eqref{eq3.88}, we have
\[
|u-\lbrace u\rbrace_{2\rho}|
\le L_2(4\rho)^{\alpha_1}
\qquad\text{in }B_{2\rho}.
\]
Consequently,
\[
\frac{|u-\lbrace u\rbrace_{2\rho}|}{\rho}
\le L_2 4^{\alpha_1}\rho^{\alpha_1-1}.
\]
Since $\alpha_1\le1$ and $\rho\le1$, Proposition~\ref{zoo}(2)
yields
\[
\begin{aligned}
G\left(x,\frac{|u-\lbrace u\rbrace_{2\rho}|}{\rho}\right)
&\le F\max\left\{
1,\left(\frac{|u-\lbrace u\rbrace_{2\rho}|}{\rho}\right)^{g^+}
\right\}\\
&\le C\max\left\{1,\rho^{-(1-\alpha_1)g^+}\right\}\\
&\le C\rho^{-\tau_0^\sharp},
\end{aligned}
\]
where
\[
\tau_0^\sharp:=(1-\alpha_1)g^+ .
\]
Inserting this estimate into \eqref{4.36}, we obtain, for every $y\in B_{\frac{R_1}{2}}(x_0)$ and $\rho\le\frac{R_1}{2}$ (so that $B_{2\rho}(y)\subset B_{2R_1}(x_0)$),
\[
\int_{B_\rho(y)}G(x,|Du|)\, \mathrm{d}x
\le C\rho^n(1+\rho^{-\tau_0^\sharp})
\le C\rho^{n-\tau_0^\sharp}.
\]
If $\tau_0^\sharp<n$, then $\tau_0=\tau_0^\sharp$ and the claim follows, because $|B_\rho|\le C\rho^{n-\tau_0}$. If $\tau_0^\sharp\ge n$, then $\tau_0=n$, and by monotonicity and the estimate at the scale $\frac{R_1}{2}$,
\[
\int_{B_\rho(y)}(1+G(x,|Du|))\, \mathrm{d}x\le\int_{B_{\frac{R_1}{2}}(y)}(1+G(x,|Du|))\, \mathrm{d}x\le C
\]
for $\rho\le\frac{R_1}{2}$ (the constant depends on $R_1$, hence only on $\mathrm{data}$ and $L_2$). This completes the proof.
\end{proof}

For $y\in B_{\frac{R_1}{2}}(x_0)$ and small $r$ put
\[
\psi_y(r):=\int_{B_r(y)}\big(1+G(x,|Du|)\big)\, \mathrm{d}x,\qquad\varphi_y(r):=\int_{B_r(y)}\big(1+G(y,|Du|)\big)\, \mathrm{d}x .
\]
Let $\delta:=\min\bigl\{1,\frac{\delta_{\mathrm{hi}}}{2},\frac{\beta\theta_*}{2\tau_0}\bigr\}$, a number depending only on $\mathrm{data}$, and let $R_2=R_2(\delta)\in(0,\min\{R_1,R_\omega\}]$ be the radius given by Lemma~\ref{Lem1.3} for this $\delta$; thus $\tau_0\delta\le\frac{\beta\theta_*}{2}$. These two numbers are used in Lemma~\ref{Lem1.4} and in the conclusion of the proof of Theorem~\ref{Thm1.1}.

\begin{lemma}\label{Lem1.4}
For every $\tau\in(0,n)$ there are $R_\tau\in\bigl(0,\min\{R_2,\frac{R_1}{2}\}\bigr]$ and $C_\tau>0$, depending only on $\mathrm{data}$, $\tau$ and $L_2$, such that
\begin{equation}\label{eq3.17}
\psi_y(\rho)+\varphi_y(\rho)\le C_\tau\rho^{\,n-\tau}\qquad\forall y\in B_{\frac{R_1}{2}}(x_0),\ \forall\rho\in(0,R_\tau].
\end{equation}
\end{lemma}
\begin{proof}
We apply Lemma~\ref{Lem:boot} with
\[
\mathcal Y:=B_{\frac{R_1}{2}}(x_0),\qquad E_r(y):=B_r(y),\qquad R_\flat:=\min\Big\{R_2,\frac{R_1}{2}\Big\},\qquad\chi_0:=\beta\theta_*,
\]
with the numbers $\tau_0$ and $\delta$ fixed above, with $\delta'$ as in Lemma~\ref{Lem1.4a}, and with $\psi_y,\varphi_y$ as defined above (they are the functions of Lemma~\ref{Lem:boot} for this choice of $E_r(y)$). We have $R_\flat\le R_2\le R_\omega$, hence $\omega(r)\le g^-\delta'$ for $0<r\le R_\flat$, and $\tau_0\delta\le\frac{\chi_0}{2}$ by the choice of $\delta$. Let $y\in\mathcal Y$. For $0<r\le R_\flat$ we have $2r\le R_1$, so that $B_{2r}(y)\subset B_{2R_1}(x_0)\subset\Omega$, and the hypotheses of Lemma~\ref{Lem:boot} are verified as follows.
\begin{enumerate}[(a)]
\item Lemma~\ref{lem:comp-factor} applies to the points of $B_r(y)$, because $r\le R_\omega\le r_*\le\frac{R_0}{2}$.
\item This is Lemma~\ref{Lem1.2a} (note that $y\in B_{\frac{R_1}{2}}(x_0)$ and $r\le R_\flat\le\frac{R_1}{2}$).
\item This is Lemma~\ref{Lem1.1} at the scale $r\le\frac{R_\flat}{2}\le r_*$, because $\delta\le\frac{\delta_{\mathrm{hi}}}{2}$ and $\delta'\le\frac{\delta_{\mathrm{hi}}}{4}$.
\item Let $0<R\le\frac{R_\flat}{2}$ and $0<\rho\le\frac{R}{4}$, and write $G_c:=G(y,\cdot)$. Since $2R\le R_\flat\le\min\{R_2,\frac{R_1}{2}\}$, we have $B_{2R}(y)\subset B_{R_1}(x_0)$ and $R\le R_2\le R_\omega$, so that Lemmas~\ref{Lem1.2}, \ref{Lem1.4a} and \ref{Lem1.3} apply in $B_R(y)$ with $x_c=y$; let $v$ be the solution of \eqref{FP}. By convexity, the $\Delta_2$-condition, \eqref{eq3.10}, \eqref{eq3.12} and \eqref{eq3.14} (recall $\Theta_R=\psi_y(2R)$),
\begin{align*}
\int_{B_\rho(y)}G_c(|Du|)\, \mathrm{d}x
&\le C\int_{B_\rho(y)}G_c(|Dv|)\, \mathrm{d}x
+C\int_{B_R(y)}G_c(|Du-Dv|)\, \mathrm{d}x\\
&\le C\rho^n\sup_{B_{\frac{R}{2}}(y)}G_c(|Dv|)
+CR^{\beta\theta_*-\frac{n\delta}{2}}\psi_y(2R)^{1+\frac{\delta}{2}}\\
&\le C\left(\frac{\rho}{R}\right)^n
\int_{B_R(y)}\bigl(1+G_c(|Du|)\bigr)\, \mathrm{d}x
+CR^{\beta\theta_*-\frac{n\delta}{2}}\psi_y(2R)^{1+\frac{\delta}{2}}.
\end{align*}
\end{enumerate}
Lemma~\ref{Lem:boot} now gives \eqref{eq3.17}, with some $R_\tau\le R_\flat=\min\{R_2,\frac{R_1}{2}\}$ and constants depending only on $\mathrm{data}$, $\tau$ and $L_2$.
\end{proof}

Now we are ready to prove Theorem \ref{Thm1.1}.

\begin{proof}[\textbf{Conclusion of the Proof of Theorem \ref{Thm1.1}}]
Let $B_{2R_1}(x_0)$, $\beta$ and $\lambda_2\ \bigl(\le\frac{1}{2}\bigr)$ be as above, and let $\tau,\theta,\varepsilon$ be the numbers given by Lemma~\ref{Lem:expo} for $\chi_0:=\beta\theta_*$, $\lambda:=\lambda_2$ and $\kappa_0:=2$. Let $\delta$ and $R_2=R_2(\delta)$ be the numbers fixed before Lemma~\ref{Lem1.4}, and let $R_\tau\le\min\{R_2,\frac{R_1}{2}\}$ be given by Lemma~\ref{Lem1.4} for this $\tau$. Let
$x_c\in B_{\frac{R_1}{4}}(x_0)$ and
$\rho<\left(\frac{R_\tau}{4}\right)^{1+\theta}$. Define
\[
R:=(2\rho)^{\frac{1}{1+\theta}}.
\]
Thus,
\[
2\rho<R<2R<R_\tau
\qquad\text{and}\qquad
\rho=\frac{1}{2} R^{1+\theta}.
\]
Then $B_{2R}(x_c)\subset B_{2R_1}(x_0)$ and $R\le R_2\le R_\omega$, so that Lemmas~\ref{Lem1.2}, \ref{Lem1.4a} and \ref{Lem1.3} apply in $B_R(x_c)$; let $v$ be the solution of \eqref{FP} in $B_R=B_R(x_c)$, and write $G_c=G(x_c,\cdot)$.

For any $c\in\mathbb{R}^n$, we have
\[
|Du-(Du)_\rho|
\le |Du-c|+\fint_{B_\rho}|Du-c|\, \mathrm{d}x.
\]
Consequently, choosing $c:=(Dv)_\rho$ and using Jensen's inequality, convexity and the $\Delta_2$-condition,
\[
\begin{aligned}
\int_{B_\rho}G_c\bigl(|Du-(Du)_\rho|\bigr)\, \mathrm{d}x
\le C\Bigg[
&\int_{B_\rho}G_c\bigl(|Dv-(Dv)_\rho|\bigr)\, \mathrm{d}x+\int_{B_R}G_c(|Du-Dv|)\, \mathrm{d}x
\Bigg].
\end{aligned}
\]

Since $2\rho<R$, estimates \eqref{eq3.11} and \eqref{eq3.12} give
\[
\int_{B_\rho}G_c\bigl(|Dv-(Dv)_\rho|\bigr)\, \mathrm{d}x
\le c_2|B_\rho|\left(\frac{\rho}{R}\right)^{\lambda_2}\fint_{B_R}G_c(|Dv|)\, \mathrm{d}x
\le C\rho^n\left(\frac{\rho}{R}\right)^{\lambda_2}R^{-n}\varphi_{x_c}(R),
\]
and $\varphi_{x_c}(R)\le C_\tau R^{n-\tau}$ by Lemma~\ref{Lem1.4}; hence
\[
\int_{B_\rho}G_c\bigl(|Dv-(Dv)_\rho|\bigr)\, \mathrm{d}x
\le C\rho^n\left(\frac{\rho}{R}\right)^{\lambda_2}R^{-\tau}.
\]
Moreover, \eqref{eq3.14} and Lemma~\ref{Lem1.4} (which gives $\Theta_R=\psi_{x_c}(2R)\le C_\tau(2R)^{n-\tau}$) yield, since $\tau\delta\le\tau<\frac{\beta\theta_*}{4}$,
\[
\int_{B_R}G_c(|Du-Dv|)\, \mathrm{d}x
\le CR^{\beta\theta_*-\frac{n\delta}{2}+(n-\tau)(1+\frac{\delta}{2})}
=CR^{\,n-\tau+\beta\theta_*-\frac{\tau\delta}{2}}
\le CR^{n-\tau+\frac{\beta\theta_*}{2}}.
\]
By \eqref{eq:expo} (with $\kappa_0=2$) both right-hand sides are bounded by $C\rho^nR^\tau=C\rho^n(2\rho)^{\varepsilon}$; hence
\[
\int_{B_\rho}G_c\bigl(|Du-(Du)_\rho|\bigr)\, \mathrm{d}x
\le C\rho^{n+\varepsilon}.
\]

To convert this estimate into a power-type estimate we apply \eqref{eq:conv} (with $E=B_\rho$, $h=|Du-(Du)_\rho|$ and $M=C$; recall $\rho\le1$):
\[
\int_{B_\rho}|Du-(Du)_\rho|^{g^-}\, \mathrm{d}x
\le C\rho^{n+\frac{\varepsilon g^-}{g^+}}.
\]

Finally, Campanato's characterization implies that
\[
u\in C^{1,\alpha}\left(B_{\frac{R_1}{8}}(x_0)\right),
\qquad
\alpha=\frac{\varepsilon}{g^+}.
\]
\textit{Dependence of $\alpha$ and of the norm.} The exponent $\alpha=\frac{\varepsilon}{g^+}$ is a function of $n$, $g^\pm$, $\beta=\min\{\beta_1,\alpha_1\beta_2\}$, $\theta_*=\min\{1,\frac{g^-}{2}\}$ and $\lambda_2=\lambda_2(n,g^\pm,\lambda_0,\Lambda_0)$ only; since $\alpha_1$ depends only on $\mathrm{data}$ (Proposition~\ref{prop1.1}(1)), so does $\alpha$. In particular $\alpha$ does not depend on $\Omega_0$, $x_0$, $R_1$ or $L_2$. The radii $R_1,R_\tau$ and the constants of the Campanato estimates depend in addition on $L_2$, hence on $d_0=\mathrm{dist}(\Omega_0,\partial\Omega)$. Covering $\overline{\Omega_0}$ by finitely many balls $B_{\frac{R_1}{16}}(x_0)$, $x_0\in\overline{\Omega_0}$, and using the trivial bound
\[
|Du(x)-Du(y)|\le 2\sup|Du|\,\bigl(\frac{16}{R_1}\bigr)^{\alpha}|x-y|^{\alpha}
\]
for $|x-y|\ge \frac{R_1}{16}$, we obtain $u\in C^{1,\alpha}(\overline{\Omega_0})$ with a norm depending only on $\mathrm{data}$ and $d_0$. This completes the proof of Theorem~\ref{Thm1.1}.
\end{proof}

\subsection{Proofs of Theorems \ref{Thm1.2} and \ref{Thm1.3}}\label{subsec:proof-bdry}
The Dirichlet problem \eqref{PD} (Theorem~\ref{Thm1.2}, \textit{case~(D)}) and the Neumann problem \eqref{PN} (Theorem~\ref{Thm1.3}, \textit{case~(N)}) are treated together: after the boundary has been flattened, the two proofs have the same structure and share most of their lemmas, and we indicate at each step what depends on the case. In both cases $u$ denotes a bounded generalized solution of the problem satisfying \eqref{M}, in case~(N) with $\sup|C(\partial\Omega\times[-M,M])|\le\widetilde M$, and we write $\mathrm{data}_\partial:=\mathrm{data}_D$ in case~(D) and $\mathrm{data}_\partial:=\mathrm{data}_N$ in case~(N). The interior regularity is given by Theorem~\ref{Thm1.1} (in case~(N) it applies because \eqref{FNV} contains \eqref{FDV}: the test functions of \eqref{FDV} belong to $W_0^{1,G}(\Omega)$), so it remains to prove the regularity up to $\partial\Omega$. We flatten the boundary and check that the structure of the problem is preserved; the higher integrability (Lemma~\ref{Lem2.1}), the comparison with a frozen problem posed on a cylinder (Lemmas~\ref{Lem2.21} and~\ref{Lem2.2}) and the Morrey estimates (Lemma~\ref{Lem2.3}, an application of Lemma~\ref{Lem:boot}) are then proved in the flattened coordinates, and a Campanato argument concludes. Throughout, $\mathbb{R}^n=\mathbb{R}^{n-1}\times\mathbb{R}$ and $B'_\rho(a')$ denotes a ball of $\mathbb{R}^{n-1}$.

\medskip
\noindent
\textbf{Flattening of the boundary.}
Since $\partial\Omega$ is compact and of class $C^{1,\beta_3}$, there are $r_\Omega\in(0,1]$ and $K_\Omega\ge1$ such that for every $z\in\partial\Omega$ there exist an open set $U_z\supset B_{r_\Omega}(z)$ and a $C^{1,\beta_3}$-diffeomorphism $\Psi=\Psi_z$ of $U_z$ onto an open set containing $\overline{Q(0,2r_\Omega)}$, with $\Psi(z)=0$,
\[
\begin{gathered}
\Psi(\Omega\cap U_z)\cap Q(0,2r_\Omega)=Q^+(0,2r_\Omega),
\qquad
\Psi(\partial\Omega\cap U_z)\cap Q(0,2r_\Omega)=Q^0(0,2r_\Omega),\\
\|\Psi\|_{C^{1,\beta_3}}+\|\Psi^{-1}\|_{C^{1,\beta_3}}\le K_\Omega ,
\end{gathered}
\]
where $Q^+(0,\rho):=Q(0,\rho)\cap\{y_n>0\}$ and $Q^0(0,\rho):=Q(0,\rho)\cap\{y_n=0\}$. The numbers $r_\Omega,K_\Omega$, together with $\beta_3$, form the $C^{1,\beta_3}$ character of $\partial\Omega$ contained in $\mathrm{data}_\Omega$. Decreasing $r_\Omega$ if necessary, we assume $4K_\Omega\sqrt n\,r_\Omega\le|B_1|^{-\frac{1}{n}}$, so that every ball of radius $\le4K_\Omega\sqrt n\,r_\Omega$ has measure at most $1$.

Fix $z\in\partial\Omega$ and write $\Psi=\Psi_z$, $J(y):=D\Psi(\Psi^{-1}(y))$ and $a(y):=|\det D\Psi^{-1}(y)|$. For $y\in Q^+:=Q^+(0,2r_\Omega)$ we put
\[
\begin{aligned}
&\hat u:=u\circ\Psi^{-1},\qquad \hat G(y,t):=G(\Psi^{-1}(y),t),\\
&\hat A(y,s,\zeta):=a(y)\,J(y)\,A\bigl(\Psi^{-1}(y),s,J(y)^T\zeta\bigr),\qquad
\hat B(y,s,\zeta):=a(y)\,B\bigl(\Psi^{-1}(y),s,J(y)^T\zeta\bigr),
\end{aligned}
\]
and, in case~(D), $\hat\phi:=\phi\circ\Psi^{-1}$, where $\phi$ is the extension of \eqref{H12} to $\mathbb{R}^n$ (so that $\hat\phi$ is defined in $Q(0,2r_\Omega)$). Since $Du=J^TD\hat u$, the change of variables $x=\Psi^{-1}(y)$ shows that $|\hat u|\le M$ and $\hat u\in W^{1,\hat G}(Q^+)$.

\textit{Case~(D).} Moreover $\eta(\hat u-\hat\phi)\in W_0^{1,\hat G}(Q^+)$ for every $\eta\in C_c^\infty(Q(0,2r_\Omega))$ (because $u-\phi\in W_0^{1,G}(\Omega)$), and
\begin{equation}\label{eq:flat-weak}
\int_{Q^+}\hat A(y,\hat u,D\hat u)\cdot D\psi\, \mathrm{d}y=\int_{Q^+}\hat B(y,\hat u,D\hat u)\,\psi\, \mathrm{d}y
\end{equation}
for every $\psi\in W_0^{1,\hat G}(Q^+)\cap L^\infty(Q^+)$ vanishing outside a compact subset of $Q(0,2r_\Omega)$ (indeed $\psi\circ\Psi$, extended by zero, is then an admissible test function in \eqref{FDV}).

\textit{Case~(N).} Let $\gamma(y'):=\Psi^{-1}(y',0)$ be the parametrization of $\partial\Omega$ induced by the chart and $\mathcal J_\gamma(y'):=\bigl(\det(D\gamma^TD\gamma)\bigr)^{\frac{1}{2}}$ the corresponding surface Jacobian, so that $\mathrm{d}\sigma=\mathcal J_\gamma(y')\, \mathrm{d}y'$. Since $K_\Omega^{-1}|\zeta|\le|D\gamma\,\zeta|\le K_\Omega|\zeta|$ and $D\gamma\in C^{0,\beta_3}$ with norm $\le K_\Omega$, we have $K_\Omega^{-(n-1)}\le\mathcal J_\gamma\le K_\Omega^{n-1}$ and $\mathcal J_\gamma\in C^{0,\beta_3}$ with a norm depending only on $n$ and $K_\Omega$. The change of variables turns \eqref{FNV} into
\begin{equation}\label{eq:flat-weak-N}
\int_{Q^+}\hat A(y,\hat u,D\hat u)\cdot D\psi\, \mathrm{d}y=\int_{Q^+}\hat B(y,\hat u,D\hat u)\,\psi\, \mathrm{d}y+\int_{\Sigma}\hat C(y',\hat u)\,\psi\, \mathrm{d}y'
\end{equation}
for every $\psi\in W^{1,\hat G}(Q^+)\cap L^\infty(Q^+)$ vanishing outside a compact subset of $Q(0,2r_\Omega)$ (the trace of $\psi$ on $\Sigma:=Q^0(0,2r_\Omega)$ is used in the last integral; $\psi\circ\Psi$, extended by zero, is an admissible test function in \eqref{FNV}), where
\[
\hat C(y',s):=\mathcal J_\gamma(y')\,C\bigl(\gamma(y'),s\bigr).
\]

\medskip
\noindent
\textbf{Invariance of the structure.}
The transformed objects satisfy the same hypotheses as the original ones, with constants depending only on the original constants, on $n$, $g^\pm$ and on $K_\Omega$:
\begin{enumerate}[(a)]
\item \textit{The function $\hat G$.} $\hat G$ satisfies \eqref{D22}, \eqref{D22diff} and \eqref{GG1} with the same $g^\pm$ and $F$, and \eqref{GG3}, \eqref{GG6b} with the exponents and coefficients composed with $\Psi^{-1}$ (the H\"older seminorms in \eqref{GG3}(v) are multiplied by at most $K_\Omega^{\bar\alpha}$). Condition \eqref{GG3}(iv) holds with $L_0$ replaced by $2L_0$ and $R_0$ by $\widetilde R_0:=\min\{\frac{R_0}{2K_\Omega},K_\Omega^{-2}\}$: if $y_0\in\overline{Q^+}$, $B_R(y_0)\subset Q(0,2r_\Omega)$ and $R<\widetilde R_0$, then $\Psi^{-1}(B_R(y_0)\cap\overline{Q^+})\subset B_{K_\Omega R}(\Psi^{-1}(y_0))$, and $\frac{L_0}{|\ln(K_\Omega R)|}\le\frac{2L_0}{|\ln R|}$ because $R\le K_\Omega^{-2}$. Finally, \eqref{GG2} holds for $\hat G$ on balls $B_r(c)\subset Q(0,2r_\Omega)$ (for almost every $x,y\in B_r(c)\cap Q^+$) with $\mu$ replaced by $\widetilde\mu:=\mu(2K_\Omega)^{-\frac{n}{g^-}}$. Indeed, $\Psi^{-1}(B_r(c)\cap Q^+)$ lies in a ball $B'$ of radius $2K_\Omega r$, and $|B'|\le1$, so that \eqref{GG2} on $B'$ gives the claim for $t\in[1,|B'|^{-1}]$; for $|B'|^{-1}\le t\le|B_r|^{-1}=(2K_\Omega)^n|B'|^{-1}$ one uses the monotonicity of $G^{-1}(x,\cdot)$ together with $G^{-1}(x,st)\le s^{\frac{1}{g^-}}G^{-1}(x,t)$ for $s\ge1$ (a consequence of \eqref{D22}).
\item \textit{The field $\hat A$.} One has $\hat A(y,s,0)=0$, $K_\Omega^{-1}|\zeta|\le|J^T\zeta|\le K_\Omega|\zeta|$, $K_\Omega^{-n}\le a\le K_\Omega^n$, and $a,J$ are bounded and $\beta_3$-H\"older continuous with constants depending only on $n$ and $K_\Omega$. Moreover $\partial_\zeta\hat A=a\,J\,(\partial_\xi A)(\Psi^{-1}(y),s,J^T\zeta)\,J^T$. By \eqref{D22diff}, the function $\varphi(t)=\frac{g(x,t)}{t}$ satisfies $\lambda^{g^+-2}\le\frac{\varphi(\lambda t)}{\varphi(t)}\le\lambda^{g^--2}$ for $\lambda\le1$ and $\lambda^{g^--2}\le\frac{\varphi(\lambda t)}{\varphi(t)}\le\lambda^{g^+-2}$ for $\lambda\ge1$; hence $\frac{\varphi(|J^T\zeta|)}{\varphi(|\zeta|)}\in[K_\Omega^{-m},K_\Omega^{m}]$ with $m:=\max\{|g^--2|,|g^+-2|\}$. Therefore \eqref{6.31}--\eqref{7.411} hold for $\hat A$ and $\hat G$ with $\hat\lambda_0:=K_\Omega^{-(n+m+2)}\lambda_0$ and $\hat\Lambda_0:=C(n)K_\Omega^{n+m+2}\Lambda_0$. As for \eqref{7.41}, let $y_1,y_2\in Q^+$, $x_i:=\Psi^{-1}(y_i)$, $a_i:=a(y_i)$, $J_i:=J(y_i)$, and write
\[
\begin{aligned}
\hat A(y_1,s_1,\zeta)-\hat A(y_2,s_2,\zeta)
={}&a_1J_1\bigl[A(x_1,s_1,J_1^T\zeta)-A(x_2,s_2,J_1^T\zeta)\bigr]\\
&+a_1J_1\bigl[A(x_2,s_2,J_1^T\zeta)-A(x_2,s_2,J_2^T\zeta)\bigr]\\
&+(a_1J_1-a_2J_2)\,A(x_2,s_2,J_2^T\zeta).
\end{aligned}
\]
The first bracket is estimated by \eqref{7.41}, using $|x_1-x_2|\le K_\Omega|y_1-y_2|$, $g(x,|J^T\zeta|)\le K_\Omega^{g^+-1}g(x,|\zeta|)$ and $\ln(e+K_\Omega|\zeta|)\le C(1+\ln(e+|\zeta|))$; the third term by $|a_1J_1-a_2J_2|\le C|y_1-y_2|^{\beta_3}$ and $|A(x,s,\xi)|\le\Lambda_1g(x,|\xi|)$ (Proposition~\ref{prop1.112}). For the second term, if $C|y_1-y_2|^{\beta_3}\le(2K_\Omega)^{-1}$ the segment joining $J_2^T\zeta$ to $J_1^T\zeta$ stays in $\{|\xi|\ge\frac{|\zeta|}{2K_\Omega}\}$, so that \eqref{6.41} and the mean value theorem give the bound $C\,g(x_2,|\zeta|)\,|y_1-y_2|^{\beta_3}$; in the opposite case $|y_1-y_2|^{\beta_3}$ is bounded below and the bound follows from $|A|\le\Lambda_1g$. Since $|y_1-y_2|$ is bounded, \eqref{7.41} holds for $\hat A$ with $\beta_1$ replaced by $\min\{\beta_1,\beta_3\}$ and $\Lambda$ replaced by $C(n,K_\Omega,g^\pm)\Lambda$.
\item \textit{The field $\hat B$.} By \eqref{6.51},
\[
|\hat B(y,s,\zeta)|\le K_\Omega^{n}\bigl(b_1M(x,|J^T\zeta|)+b_2h_1(x,|s|)+b_3\bigr)
\]
with $x=\Psi^{-1}(y)$, and $M(x,|J^T\zeta|)\le K_\Omega^{g^+}M(x,|\zeta|)$, because $G(x,Kt)\le K^{g^+}G(x,t)$ for $K\ge1$ and the inverse of the complementary function of $H_1$ is concave, so that $\widetilde{H_1}^{-1}(x,\lambda\,\cdot)\le\lambda\,\widetilde{H_1}^{-1}(x,\cdot)$ for $\lambda\ge1$. Hence \eqref{6.51} holds for $\hat B$, $\hat G$ and $\hat H_1(y,t):=H_1(\Psi^{-1}(y),t)$ with $b_1,b_2,b_3$ multiplied by $K_\Omega^{n+g^+}$, and Lemma~\ref{lem:B} holds for $\hat B$ (with constants depending in addition on $K_\Omega$).
\item \textit{The datum $\hat\phi$ (case~(D)).} We replace $M_\phi$ by a larger constant depending only on $M_\phi$, $n$ and $K_\Omega$ (and keep the letter $M_\phi$) such that $|\hat\phi|+|D\hat\phi|\le M_\phi$ in $Q(0,2r_\Omega)$ and $\|\hat\phi\|_{C^{1,\beta_3}(Q^0(0,2r_\Omega))}\le M_\phi$ (recall that $\Psi^{-1}|_{Q^0}$ is a $C^{1,\beta_3}$ parametrization of $\partial\Omega$ with norm $\le K_\Omega$).
\item \textit{The datum $\hat C$ (case~(N)).} $|\hat C(y',s)|\le K_\Omega^{n-1}\bigl(c_1h_2(\gamma(y'),|s|)+c_2\bigr)$, and $|\hat C(y',s)|\le K_\Omega^{n-1}\widetilde M$ for $|s|\le M$. Moreover, for $|s_i|\le M$,
\[
|\hat C(y_1',s_1)-\hat C(y_2',s_2)|\le K_\Omega^{n-1}\bigl|C(\gamma(y_1'),s_1)-C(\gamma(y_2'),s_2)\bigr|+|\mathcal J_\gamma(y_1')-\mathcal J_\gamma(y_2')|\,|C(\gamma(y_2'),s_2)|,
\]
so that \eqref{C1} holds for $\hat C$ with $\beta_1$ replaced by $\min\{\beta_1,\beta_3\}$ and $\Lambda$ replaced by $C(n,K_\Omega,\widetilde M)\Lambda$ (recall $|\gamma(y_1')-\gamma(y_2')|\le K_\Omega|y_1'-y_2'|$).
\end{enumerate}

\medskip
\noindent
\textbf{Convention.}
From now on, when a point $z\in\partial\Omega$ has been fixed we work in the flattened coordinates $y=\Psi_z(x)$; we denote the flattened variable again by $x$ and keep the letters $u,A,B,C,G,g,H_1,\phi$ for the transformed objects $\hat u,\hat A,\hat B,\hat C,\hat G,\hat g,\hat H_1,\hat\phi$. The structure constants are replaced by the new ones given in (a)--(e): $\beta_1$ by $\min\{\beta_1,\beta_3\}$, $L_0$ by $2L_0$, $R_0$ by $\widetilde R_0$, $\mu$ by $\widetilde\mu$, $\widetilde M$ (case~(N)) by $K_\Omega^{n-1}\widetilde M$, and $\lambda_0,\Lambda_0,b_i,M_\phi$ (case~(D)) and the constants of Lemma~\ref{lem:B} by the corresponding ones; we keep the same symbols, and $L_2$ denotes a H\"older constant of $u$ in the flattened coordinates (it is at most $K_\Omega L_2$, where $L_2$ is the constant in the original coordinates). Accordingly, ``depending only on $\mathrm{data}_\partial$'' means depending only on the original quantities of the problem in $\Omega$. For a ball $B_\rho(y)\subset Q(0,2r_\Omega)$ we write $\Omega_\rho(y):=B_\rho(y)\cap\{x_n>0\}$ (it may be empty); if $y_n\ge0$ then $|\Omega_\rho(y)|\ge\frac{1}{2}|B_\rho|$. All the results of Section~\ref{sec:prelim} and of Subsection~\ref{subsec:proof1.1} that only involve the structure of $G$, $A$ and $B$ (Propositions~\ref{prop1.112}, \ref{Prop2.2}, \ref{prop:SP}, Lemmas~\ref{lem:B}, \ref{lem:comp-factor}, Corollary~\ref{cor:g-comp}) hold for the flattened objects with the new constants. In particular, the interior Lemmas~\ref{Lem1.1}--\ref{Lem1.4} hold, with the same proofs, for the flattened solution in every ball contained in $Q^+(0,2r_\Omega)$, because test functions with compact support in $Q^+(0,2r_\Omega)$ are admissible in both \eqref{eq:flat-weak} and \eqref{eq:flat-weak-N}.

\begin{lemma}\label{Lem2.1}
Let $u$ be as above, let $z\in\partial\Omega$ and work in the flattened coordinates of $\Psi_z$ (Convention); in case~(N) we put $\phi:\equiv0$, and in both cases $w:=u-\phi$. There are $r_*\in(0,\frac{r_\Omega}{2}]$, $c\ge1$, $c_0\ge1$ and $\delta_{\mathrm{hi}}\in(0,1]$, depending only on $\mathrm{data}_\partial$, with $r_*$ and $\delta_{\mathrm{hi}}$ not larger than the constants of Lemma~\ref{Lem1.1}, such that for every $y_0\in\mathbb{R}^n$ and $R\in(0,r_*]$ with $B_{2R}(y_0)\subset Q(0,r_\Omega)$:
\begin{enumerate}[(i)]
\item (Caccioppoli inequality) for every $\omega\in\mathbb{R}$ such that, in case~(D), $\omega=0$ or $B_{2R}(y_0)\subset Q^+$, and, in case~(N), $|\omega|\le M$,
\begin{equation}\label{eq4.1a}
\int_{\Omega_R(y_0)}G(x,|Dw|)\, \mathrm{d}x\le c\int_{\Omega_{2R}(y_0)}G\Big(x,\frac{|w-\omega|}{R}\Big)\, \mathrm{d}x+c|B_{2R}|;
\end{equation}
\item (higher integrability) if $(y_0)_n\ge0$, then for every $\delta\in(0,\delta_{\mathrm{hi}}]$
\begin{equation}\label{eq4.1}
\left( \fint_{\Omega_R(y_0)} G(x,|Du|)^{(1+\delta)}\, \mathrm{d}x \right)^{\frac{1}{1+\delta}} \leq c_0 \left( 1 + \fint_{\Omega_{2R}(y_0)} G(x,|Du|)\, \mathrm{d}x \right),
\end{equation}
and consequently
\begin{equation}\label{eq4.1b}
\int_{\Omega_R(y_0)}\bigl(1+G(x,|Du|)\bigr)^{1+\delta}\, \mathrm{d}x\le c\,R^{-n\delta}\,\Theta^{1+\delta},\qquad \Theta:=\int_{\Omega_{2R}(y_0)}\bigl(1+G(x,|Du|)\bigr)\, \mathrm{d}x .
\end{equation}
\end{enumerate}
\end{lemma}

\begin{proof}
\textbf{Step 0: preliminaries.} We have $|w|\le M+M_\phi=:M_1$ (with $M_\phi:=0$ in case~(N)). Let
\[
\overline G(x,t):=G\bigl((x',|x_n|),t\bigr),\qquad x\in Q(0,2r_\Omega),
\]
be the even reflection of $G$. The function $\overline G$ satisfies \eqref{D22}, \eqref{GG1} and \eqref{GG2} in $Q(0,2r_\Omega)$ with the constants of the flattened problem: if $B_r(c)\subset Q(0,2r_\Omega)$ and $\bar c:=(c',|c_n|)$, then $x\mapsto(x',|x_n|)$ maps $B_r(c)$ into $B_r(\bar c)\cap\overline{Q^+}$, and $B_r(\bar c)\subset Q(0,2r_\Omega)$ by the symmetry of the cube, so that \eqref{GG2} on $B_r(\bar c)$ (item (a) above) gives \eqref{GG2} for $\overline G$ on $B_r(c)$. Hence Proposition~\ref{prop:SP} applies to $\overline G$ in every ball contained in $Q(0,2r_\Omega)$. Let $r_0$, $\varepsilon\in(0,1)$ and $c$ be the constants of Proposition~\ref{prop:SP} for $\overline G$, put $s:=1+\varepsilon$, and let $r_*^{(1.1)}$, $\delta_{\mathrm{hi}}^{(1.1)}$ be the constants of Lemma~\ref{Lem1.1}. We choose
\[
r_*:=\min\Big\{\frac{r_\Omega}{2},\ \frac{r_0}{4},\ r_*^{(1.1)},\ \frac{\widetilde R_0}{4}\Big\}.
\]
We also fix the extension $\overline w$ of $w$ to $Q(0,2r_\Omega)$ used below. In case~(D), $\overline w:=w$ in $Q^+$ and $\overline w:=0$ in $Q^-:=Q(0,2r_\Omega)\cap\{x_n<0\}$; since $\eta w\in W_0^{1,G}(Q^+)$ for every $\eta\in C_c^\infty(Q(0,2r_\Omega))$, we have $\overline w\in W^{1,\overline G}(B)$ for every ball $B\subset Q(0,r_\Omega)$ and $D\overline w=\mathbf 1_{Q^+}Dw$. In case~(N), $\overline w(x):=u(x',|x_n|)$ is the even reflection of $u$; then $\overline w\in W^{1,\overline G}(Q(0,2r_\Omega))$ and $|D\overline w|(x)=|Du|(x',|x_n|)$. All constants below depend only on $\mathrm{data}_\partial$.

\textbf{Step 1: Caccioppoli inequalities.} Let $B_{2R}(y_0)\subset Q(0,r_\Omega)$ and $R\le r_*$.

\textit{Case~(D).} Let $\varrho\in\{\frac{3R}{2},2R\}$, let $\xi\in C_c^\infty(B_\varrho(y_0))$ with $0\le\xi\le1$, $\xi=1$ on $B_R(y_0)$, $|D\xi|\le\frac{4}{R}$, and let $\omega\in\mathbb{R}$ be such that either $\omega=0$ or $\operatorname{supp}\xi\subset Q^+$. Then $\vartheta:=\xi^{g^+}(w-\omega)\in W_0^{1,G}(Q^+)\cap L^\infty$ vanishes outside a compact subset of $Q(0,2r_\Omega)$, hence it is admissible in \eqref{eq:flat-weak}. Writing $A=A(x,u,Du)$ and using $D\vartheta=\xi^{g^+}Dw+g^+\xi^{g^+-1}(w-\omega)D\xi$ we get
\[
\int_{Q^+}\xi^{g^+}A\cdot Dw\, \mathrm{d}x=-g^+\int_{Q^+}\xi^{g^+-1}(w-\omega)\,A\cdot D\xi\, \mathrm{d}x+\int_{Q^+}B(x,u,Du)\,\vartheta\, \mathrm{d}x .
\]
By \eqref{eqc2}, \eqref{eqc1} and $|D\phi|\le M_\phi$,
\[
A\cdot Dw=A\cdot Du-A\cdot D\phi\ge\lambda_1G(x,|Du|)-\Lambda_1M_\phi\,g(x,|Du|).
\]
By Young's inequality \eqref{Yi}, Proposition~\ref{zoo}(1),(3) and \eqref{L1}, for $\varepsilon_2\in(0,1)$,
\[
g(x,t)M_\phi\le\widetilde G\bigl(x,\varepsilon_2g(x,t)\bigr)+G\Big(x,\frac{M_\phi}{\varepsilon_2}\Big)\le\varepsilon_2^{\widetilde g^-}(g^+-1)\,G(x,t)+C_{\varepsilon_2},
\]
and we choose $\varepsilon_2$ so that $\Lambda_1\varepsilon_2^{\widetilde g^-}(g^+-1)\le\frac{\lambda_1}{8}$. The two remaining terms are estimated exactly as in the proof of Lemma~\ref{Lem1.1} (estimates \eqref{Eqa10} and \eqref{Eqa11}, with $u-\omega$ replaced by $w-\omega$, $|w-\omega|\le2M_1$, and Lemma~\ref{lem:B}), the small parameters being chosen so that each of the three absorbed terms is at most $\frac{\lambda_1}{8}\int_{\Omega_{2R}(y_0)}\xi^{g^+}G(x,|Du|)\, \mathrm{d}x$. This yields
\[
\frac{\lambda_1}{2}\int_{Q^+}\xi^{g^+}G(x,|Du|)\, \mathrm{d}x\le c\int_{B_\varrho(y_0)\cap Q^+}G\Big(x,\frac{|w-\omega|}{R}\Big)\, \mathrm{d}x+c|B_{2R}| ,
\]
and, since $G(x,|Dw|)\le2^{g^+-1}\bigl(G(x,|Du|)+G(x,M_\phi)\bigr)$ and $G(x,M_\phi)\le F\max\{1,M_\phi^{g^+}\}$,
\begin{equation}\label{eq:cacc-D}
\int_{\Omega_R(y_0)}G(x,|Dw|)\, \mathrm{d}x\le c\int_{B_\varrho(y_0)\cap Q^+}G\Big(x,\frac{|w-\omega|}{R}\Big)\, \mathrm{d}x+c|B_{2R}| .
\end{equation}
For $\varrho=2R$ this is \eqref{eq4.1a}, which is therefore valid for every position of $y_0$ (for $\omega=0$).

\textit{Case~(N).} Let $\xi\in C_c^\infty(B_{2R}(y_0))$ with $0\le\xi\le1$, $\xi=1$ on $B_R(y_0)$, $|D\xi|\le\frac{4}{R}$, let $\omega\in[-M,M]$ and $\vartheta:=\xi^{g^+}(u-\omega)$. Then $\vartheta\in W^{1,G}(Q^+)\cap L^\infty$ vanishes outside a compact subset of $Q(0,2r_\Omega)$, so it is admissible in \eqref{eq:flat-weak-N}. Writing $A=A(x,u,Du)$ and $D\vartheta=\xi^{g^+}Du+g^+\xi^{g^+-1}(u-\omega)D\xi$,
\[
\int_{Q^+}\xi^{g^+}A\cdot Du\, \mathrm{d}x=-g^+\int_{Q^+}\xi^{g^+-1}(u-\omega)\,A\cdot D\xi\, \mathrm{d}x+\int_{Q^+}B\,\vartheta\, \mathrm{d}x+\int_\Sigma C(y',u)\,\vartheta\, \mathrm{d}y' .
\]
The first two terms on the right-hand side are estimated exactly as in the proof of Lemma~\ref{Lem1.1} (estimates \eqref{Eqa10} and \eqref{Eqa11}, with $|u-\omega|\le2M$ and Lemma~\ref{lem:B}), the small parameters being chosen so that each absorbed term is at most $\frac{\lambda_1}{4}\int_{\Omega_{2R}(y_0)}\xi^{g^+}G(x,|Du|)\, \mathrm{d}x$. For the boundary term, since $\vartheta$ vanishes for large $x_n$, the fundamental theorem of calculus in the $x_n$-direction gives $\int_\Sigma|\vartheta|\, \mathrm{d}y'\le\int_{Q^+}|D\vartheta|\, \mathrm{d}x$, hence
\[
\Big|\int_\Sigma C(y',u)\vartheta\, \mathrm{d}y'\Big|\le\widetilde M\int_{Q^+}\xi^{g^+}|Du|\, \mathrm{d}x+\widetilde Mg^+\int_{Q^+}\xi^{g^+-1}|u-\omega||D\xi|\, \mathrm{d}x .
\]
Since $G(x,t)\ge F^{-1}t^{g^-}$ for $t\ge1$ and $g^->1$, there is $C_1=C_1(\widetilde M,\lambda_1,F,g^-)$ with $\widetilde Mt\le\frac{\lambda_1}{4}G(x,t)+C_1$ for all $t\ge0$. Therefore the first integral is bounded by
\[
\frac{\lambda_1}{4}\int_{\Omega_{2R}(y_0)}\xi^{g^+}G(x,|Du|)\, \mathrm{d}x+C_1|B_{2R}|,
\]
and the second one, taking $t=\frac{4g^+|u-\omega|}{R}$ and using Proposition~\ref{zoo}(1), by
\[
\frac{\lambda_1}{4}(4g^+)^{g^+}\int_{\Omega_{2R}(y_0)}G\bigl(x,\frac{|u-\omega|}{R}\bigr)\, \mathrm{d}x+C_1|B_{2R}|.
\]
Absorbing the three terms in the left-hand side we obtain \eqref{eq4.1a} (recall $w=u$ in case~(N)).

\textbf{Step 2: a reverse H\"older inequality.} We claim that, for $B_{2R}(y_0)\subset Q(0,r_\Omega)$ and $R\le r_*$,
\begin{equation}\label{eq4.4}
\fint_{B_R(y_0)}\overline G(x,|D\overline w|)\, \mathrm{d}x\le c+c\left(\fint_{B_{2R}(y_0)}\overline G(x,|D\overline w|)^{\frac{1}{1+\varepsilon}}\, \mathrm{d}x\right)^{1+\varepsilon}.
\end{equation}

\textit{Case~(D).} Recall that $\overline G(x,|D\overline w|)=G(x,|Dw|)\mathbf 1_{Q^+}$.

\textit{Case A: $(y_0)_n\ge\frac{3R}{2}$.} Then $B_{\frac{3R}{2}}(y_0)\subset Q^+$. We apply \eqref{eq:cacc-D} with $\varrho=\frac{3R}{2}$ and $\omega:=\fint_{B_{\frac{3R}{2}}(y_0)}w\, \mathrm{d}x$, and then Proposition~\ref{prop:SP} in $B_{\frac{3R}{2}}(y_0)$ (whose radius is at most $r_0$), together with $G\bigl(x,\frac{a}{R}\bigr)\le\bigl(\frac{3}{2}\bigr)^{g^+}G\bigl(x,\frac{a}{\frac{3R}{2}}\bigr)$ (Proposition~\ref{zoo}(1)):
\[
\fint_{B_R(y_0)}G(x,|Dw|)\, \mathrm{d}x\le c\fint_{B_{\frac{3R}{2}}(y_0)}G\Big(x,\frac{|w-\omega|}{R}\Big)\, \mathrm{d}x+c\le c+c\Big(\fint_{B_{\frac{3R}{2}}(y_0)}G(x,|Dw|)^{\frac{1}{1+\varepsilon}}\, \mathrm{d}x\Big)^{1+\varepsilon}.
\]
Since $B_{\frac{3R}{2}}(y_0)\subset B_{2R}(y_0)$, \eqref{eq4.4} follows.

\textit{Case B: $(y_0)_n<\frac{3R}{2}$.} The function $\overline w$ vanishes on $B_{2R}(y_0)\cap Q^-$, a set of measure at least $\sigma|B_{2R}|$ with $\sigma:=8^{-n}$: if $(y_0)_n\le0$ it contains half of the ball, and if $0<(y_0)_n<\frac{3R}{2}$ it contains the ball $B_{\frac{R}{4}}\Bigl(\bigl(y_0',-\frac{R}{4}\bigr)\Bigr)$. We apply \eqref{eq:cacc-D} with $\varrho=2R$ and $\omega=0$, and Proposition~\ref{prop:SP} for $\overline G$ and $\overline w$ in $B_{2R}(y_0)$ (radius $\le r_0$), in its version for functions vanishing on a set of measure $\ge\sigma|B_{2R}|$, together with $\overline G\bigl(x,\frac{a}{R}\bigr)\le2^{g^+}\overline G\bigl(x,\frac{a}{2R}\bigr)$:
\[
\fint_{B_R(y_0)}\overline G(x,|D\overline w|)\, \mathrm{d}x\le c\fint_{B_{2R}(y_0)}\overline G\Big(x,\frac{|\overline w|}{R}\Big)\, \mathrm{d}x+c\le c+c\Big(\fint_{B_{2R}(y_0)}\overline G(x,|D\overline w|)^{\frac{1}{1+\varepsilon}}\, \mathrm{d}x\Big)^{1+\varepsilon}.
\]
This proves \eqref{eq4.4}.

\textit{Case~(N).} Let $\bar y_0:=(y_0',-(y_0)_n)$ (then $B_{2R}(\bar y_0)\subset Q(0,r_\Omega)$ as well) and $\omega:=(\overline w)_{B_{2R}(y_0)}=(\overline w)_{B_{2R}(\bar y_0)}\in[-M,M]$. The set $B_R(y_0)\cap\{x_n<0\}$ is the mirror image of $\Omega_R(\bar y_0)$, and similarly for $B_{2R}$, so that
\[
\int_{B_R(y_0)}\overline G(x,|D\overline w|)\, \mathrm{d}x=\int_{\Omega_R(y_0)}G(x,|Du|)\, \mathrm{d}x+\int_{\Omega_R(\bar y_0)}G(x,|Du|)\, \mathrm{d}x,
\]
and the same identity holds for $\int_{B_{2R}(y_0)}\overline G\bigl(x,\frac{|\overline w-\omega|}{R}\bigr)\, \mathrm{d}x$ with $\Omega_{2R}(y_0)$, $\Omega_{2R}(\bar y_0)$. Applying \eqref{eq4.1a} at $y_0$ and at $\bar y_0$ (with the same $\omega$) and adding,
\[
\int_{B_R(y_0)}\overline G(x,|D\overline w|)\, \mathrm{d}x\le c\int_{B_{2R}(y_0)}\overline G\Big(x,\frac{|\overline w-\omega|}{R}\Big)\, \mathrm{d}x+c|B_{2R}| .
\]
By Proposition~\ref{prop:SP} for $\overline G$ in $B_{2R}(y_0)$ (radius $2R\le r_0$) and $\overline G\bigl(x,\frac{a}{R}\bigr)\le2^{g^+}\overline G\bigl(x,\frac{a}{2R}\bigr)$ we obtain \eqref{eq4.4}.

\textbf{Step 3: Gehring's lemma.} Fix $y_0$ and $R\le r_*$ with $B_{2R}(y_0)\subset Q(0,r_\Omega)$. By Step~2, \eqref{eq4.4} holds for every pair of concentric balls $B_\rho(y)\subset B_{2\rho}(y)\subset B_{2R}(y_0)$ (note $\rho\le R\le r_*$ and $B_{2\rho}(y)\subset Q(0,r_\Omega)$). Proposition~\ref{prop:Gehring}, applied with $\kappa=2$, $Q=B_{2R}(y_0)$, $s=1+\varepsilon$, $f=\overline G(\cdot,|D\overline w|)^{\frac{1}{s}}$, $\varphi\equiv c^{\frac{1}{s}}$ and $b=c$, yields $\delta_G=\delta_G(n,1+\varepsilon,c,2)>0$; we put $\delta_{\mathrm{hi}}:=\min\{1,\delta_{\mathrm{hi}}^{(1.1)},\delta_G\}$. For $\delta=\delta_{\mathrm{hi}}$,
\[
\Big(\fint_{B_R(y_0)}\overline G(x,|D\overline w|)^{1+\delta}\, \mathrm{d}x\Big)^{\frac{1}{1+\delta}}\le c\Big(1+\fint_{B_{2R}(y_0)}\overline G(x,|D\overline w|)\, \mathrm{d}x\Big),
\]
and for smaller $\delta$ the same inequality follows from Jensen's inequality.

\textbf{Step 4: conclusion.} \textit{Case~(D).} Let $(y_0)_n\ge0$, so that $|\Omega_\rho(y_0)|\ge\frac{1}{2}|B_\rho|$ for $\rho=R,2R$. Since
\[
G(x,|Du|)\le2^{g^+-1}\bigl(G(x,|Dw|)+F\max\{1,M_\phi^{g^+}\}\bigr)
\]
and $\overline G(x,|D\overline w|)=G(x,|Dw|)\mathbf 1_{Q^+}$, Step~3 gives
\[
\begin{aligned}
\Big(\fint_{\Omega_R(y_0)}G(x,|Du|)^{1+\delta}\, \mathrm{d}x\Big)^{\frac{1}{1+\delta}}
&\le C\Big(\frac{|B_R|}{|\Omega_R(y_0)|}\Big)^{\frac{1}{1+\delta}}\Big(\fint_{B_R(y_0)}\overline G(x,|D\overline w|)^{1+\delta}\, \mathrm{d}x\Big)^{\frac{1}{1+\delta}}+C\\
&\le C\Big(1+\fint_{B_{2R}(y_0)}\overline G(x,|D\overline w|)\, \mathrm{d}x\Big),
\end{aligned}
\]
and
\[
\fint_{B_{2R}(y_0)}\overline G(x,|D\overline w|)\, \mathrm{d}x\le\frac{C}{|B_{2R}|}\int_{\Omega_{2R}(y_0)}\bigl(1+G(x,|Du|)\bigr)\, \mathrm{d}x\le C\bigl(1+\fint_{\Omega_{2R}(y_0)}G(x,|Du|)\, \mathrm{d}x\bigr).
\]
This is \eqref{eq4.1}.

\textit{Case~(N).} Let $(y_0)_n\ge0$ and $\bar y_0:=(y_0',-(y_0)_n)$. If $x\in\Omega_{2R}(\bar y_0)$, then $x_n>0\ge-(y_0)_n$, hence $|x-y_0|\le|x-\bar y_0|<2R$; thus $\Omega_{2R}(\bar y_0)\subset\Omega_{2R}(y_0)$ and
\[
\int_{B_{2R}(y_0)}\overline G(x,|D\overline w|)\, \mathrm{d}x\le2\int_{\Omega_{2R}(y_0)}G(x,|Du|)\, \mathrm{d}x.
\]
Since $\Omega_R(y_0)\subset B_R(y_0)$ and $|\Omega_R(y_0)|\ge\frac{1}{2}|B_R|$, Step~3 yields \eqref{eq4.1}.

Finally, \eqref{eq4.1b} follows from \eqref{eq4.1} as in the proof of Lemma~\ref{Lem1.1}: $(1+G)^{1+\delta}\le2^\delta(1+G^{1+\delta})$, $|\Omega_\rho(y_0)|\ge\frac{1}{2}|B_\rho|$ and $1+\fint_{\Omega_{2R}(y_0)}G(x,|Du|)\, \mathrm{d}x=\frac{1}{|\Omega_{2R}(y_0)|}\Theta$.
\end{proof}

\medskip\noindent
\textbf{Standing choices.}
By Proposition~\ref{prop1.1}(2) in case~(D) and Proposition~\ref{prop1.1}(3) in case~(N), $u\in C^{0,\alpha_1}(\overline\Omega)$ (we assume $\alpha_1<1$); in case~(D), moreover, $u=\phi$ pointwise on $\partial\Omega$, because $u\in C(\overline\Omega)$ and $u-\phi\in W_0^{1,G}(\Omega)$. We write $L_2$ for a $C^{0,\alpha_1}$-seminorm bound of $u$ in the flattened coordinates (by the Convention, $L_2\le K_\Omega\|u\|_{C^{0,\alpha_1}(\overline\Omega)}$, a number depending only on $\mathrm{data}_\partial$ and $\Omega$), and we put
\[
\beta:=\min\{\beta_1,\beta_3,\alpha_1\beta_2\},\qquad
\theta_*:=\min\Big\{1,\frac{g^-}{2}\Big\},\qquad
\tau_0:=\min\{(1-\alpha_1)g^+,\,n\}\in(0,n].
\]
Let $r_*$ and $\delta_{\mathrm{hi}}$ be the constants of Lemma~\ref{Lem2.1}, let $\delta_*^\partial=\delta_*^\partial(n,g^\pm,F,\lambda_0,\Lambda_0,\widetilde M)>0$ be the Gehring exponent of the frozen problem defined in Step~4 of the proof of Lemma~\ref{Lem2.21} (in case~(D) it does not depend on $\widetilde M$), and put
\[
\delta':=\min\Big\{\frac{\delta_{\mathrm{hi}}}{4},\delta_*^\partial\Big\},\qquad R_\omega:=\min\Big\{r_*,\ \exp\Big(-\frac{L_0}{g^-\delta'}\Big)\Big\},\qquad R_1:=\frac{1}{4}\min\Big\{R_\omega,\ \frac{r_\Omega}{4K_\Omega}\Big\},
\]
where $L_0$ is the constant of the flattened problem, so that $\omega(R)=\frac{L_0}{|\ln R|}\le g^-\delta'$ for $R\le R_\omega$. All these numbers depend only on $\mathrm{data}_\partial$ (and $L_2$ for those involving $L_2$); in particular they do not depend on the point $z\in\partial\Omega$ below.

Points of $\overline\Omega$ at distance $\ge \frac{R_1}{2}$ from $\partial\Omega$ are handled by Theorem~\ref{Thm1.1} (see the conclusion below). Let now $x_0\in\overline\Omega$ with $\mathrm{dist}(x_0,\partial\Omega)<2R_1$ and let $z\in\partial\Omega$ with $|x_0-z|=\mathrm{dist}(x_0,\partial\Omega)$. We work in the flattened coordinates of $\Psi_z$ (Convention), and we denote $\Psi_z(x_0)$ again by $x_0$. Then $(x_0)_n\ge0$ and $|x_0|\le K_\Omega|x_0-z|<2K_\Omega R_1$ (the segment $[x_0,z]$ lies in $B_{r_\Omega}(z)\subset U_z$), so that
\begin{equation}\label{eq:star}
B_s(y)\subset Q(0,r_\Omega)\qquad\text{whenever }|y-x_0|\le3R_1\text{ and }s\le6R_1 ,
\end{equation}
since $|y|+s<(3+2K_\Omega+6)R_1\le11K_\Omega R_1<r_\Omega$. The constants $L_2,\beta$, $r_*,\delta_{\mathrm{hi}},R_\omega,R_1$ are those of the flattened problem (Convention); by \eqref{eq:star} Lemma~\ref{Lem2.1} applies to all balls $B_{2R}(y_0)\subset B_{2R_1}(x_0)$ with $R\le r_*$. For $y\in\overline{Q^+}$ and $r>0$ we put
\[
\psi_y(r):=\int_{\Omega_r(y)}\bigl(1+G(x,|Du|)\bigr)\, \mathrm{d}x,\qquad\varphi_y(r):=\int_{\Omega_r(y)}\bigl(1+G(y,|Du|)\bigr)\, \mathrm{d}x .
\]

\begin{lemma}\label{Lem2.0}
For every $y\in\overline{Q^+}$ with $|y-x_0|\le R_1$ and every $\rho\in(0,R_1]$ we have $\psi_y(\rho)\le C\rho^{\,n-\tau_0}$, where $C$ depends only on $\mathrm{data}_\partial$ and $L_2$.
\end{lemma}

\begin{proof}
Put $w:=u-\phi$ (with $\phi:\equiv0$ in case~(N)). By \eqref{eq:star}, $B_{2\rho}(y)\subset Q(0,r_\Omega)$, and $\rho\le R_1\le\frac{r_*}{4}$, so that Lemma~\ref{Lem2.1}(i) applies in $B_{2\rho}(y)$. We choose $\omega$ as follows.

\textit{Case~(N).} We take $\omega:=u(y)\in[-M,M]$; by the H\"older continuity of $u$, $|w-\omega|\le L_2(2\rho)^{\alpha_1}$ in $\Omega_{2\rho}(y)$.

\textit{Case~(D), $y_n\ge2\rho$.} Then $B_{2\rho}(y)\subset Q^+$, and we take $\omega:=w(y)$. For $x\in B_{2\rho}(y)$,
\[
|w(x)-\omega|\le|u(x)-u(y)|+|\phi(x)-\phi(y)|\le L_2(2\rho)^{\alpha_1}+2M_\phi\rho\le C\rho^{\alpha_1}.
\]

\textit{Case~(D), $y_n<2\rho$.} We take $\omega=0$. For $x\in\Omega_{2\rho}(y)$ let $\hat x:=(x',0)\in Q^0(0,2r_\Omega)$; then $|x-\hat x|=x_n<4\rho$ and $w(\hat x)=0$, so that
\[
|w(x)|=|w(x)-w(\hat x)|\le L_2(4\rho)^{\alpha_1}+4M_\phi\rho\le C\rho^{\alpha_1}.
\]

In all cases $|w-\omega|\le C\rho^{\alpha_1}$ in $\Omega_{2\rho}(y)$, and Proposition~\ref{zoo}(2) together with $\rho\le1$ gives
\[
G\Big(x,\frac{|w-\omega|}{\rho}\Big)\le F\max\bigl\{1,(C\rho^{\alpha_1-1})^{g^+}\bigr\}\le C\rho^{-\tau_0^\sharp}\ \text{ in }\Omega_{2\rho}(y),\qquad\tau_0^\sharp:=(1-\alpha_1)g^+ .
\]
By \eqref{eq4.1a}, $\int_{\Omega_\rho(y)}G(x,|Dw|)\, \mathrm{d}x\le C\rho^{n-\tau_0^\sharp}$. Finally, since
\[
G(x,|Du|)\le2^{g^+-1}\bigl(G(x,|Dw|)+G(x,M_\phi)\bigr),\qquad G(x,M_\phi)\le F\max\{1,M_\phi^{g^+}\},
\]
we obtain $\psi_y(\rho)\le C\rho^{n-\tau_0^\sharp}+C|B_\rho|\le C\rho^{n-\tau_0^\sharp}$ for every $\rho\in(0,R_1]$. If $\tau_0^\sharp<n$, then $\tau_0=\tau_0^\sharp$ and the claim follows. If $\tau_0^\sharp\ge n$, then $\tau_0=n$ and, $\psi_y$ being nondecreasing, $\psi_y(\rho)\le\psi_y(R_1)\le CR_1^{\,n-\tau_0^\sharp}\le C$.
\end{proof}

\medskip\noindent
\textbf{The frozen problems.}
Let $x_c\in\mathbb{R}^n$ with $h:=(x_c)_n\ge0$ and $R\in(0,R_\omega]$ be such that $B_{2R}(x_c)\subset B_{2R_1}(x_0)$ (so that $B_{2R}(x_c)\subset Q(0,r_\Omega)$ by \eqref{eq:star}). For $\rho>0$ let $D_\rho=D_\rho(x_c)$ be the cylinder
\[
D_\rho:=B'_{\frac{\rho}{2}}(x_c')\times\Bigl(\max\Bigl\{0,h-\frac{\rho}{2}\Bigr\},\,h+\frac{\rho}{2}\Bigr),
\]
and let $\Lambda_\rho:=B'_{\frac{\rho}{2}}(x_c')\times\{0\}$ if $h\le\frac{\rho}{2}$ and $\Lambda_\rho:=\emptyset$ otherwise (so that $\Lambda_\rho$ is the part of $\partial D_\rho$ lying on $\{x_n=0\}$, at the bottom of $D_\rho$). In case~(N) we put $\Sigma_\rho:=\Lambda_\rho$ (the part of the boundary carrying the conormal condition), in case~(D) $\Sigma_\rho:=\emptyset$, and in both cases $\Gamma_\rho:=\partial D_\rho\setminus\overline{\Sigma_\rho}$. The following elementary facts are used repeatedly: $D_\rho$ is convex and increasing in $\rho$, $\operatorname{dist}\bigl(D_{\frac{R}{2}},\partial D_R\setminus\overline{\Lambda_R}\bigr)\ge\frac{R}{4}$, and
\begin{equation}\label{eq:cyl}
\Omega_{\frac{\rho}{2}}(x_c)\subset D_\rho\subset\Omega_{\rho}(x_c),\qquad\text{in particular}\quad\Omega_\rho(x_c)\subset D_{2\rho}.
\end{equation}
Indeed, if $x\in D_\rho$ then $|x'-x_c'|<\frac{\rho}{2}$, $|x_n-h|<\frac{\rho}{2}$ and $x_n>0$, so that $|x-x_c|<\frac{\rho}{\sqrt2}$; conversely, if $|x-x_c|<\frac{\rho}{2}$ and $x_n>0$, then $x\in D_\rho$. If $\Sigma_R\ne\emptyset$ we call
\[
D^{\mathrm{sym}}_R:=\{x:(x',|x_n|)\in D_R\}=B'_{\frac{R}{2}}(x_c')\times\bigl(-h-\frac{R}{2},h+\frac{R}{2}\bigr)
\]
the reflected cylinder; if $\Sigma_R=\emptyset$ we put $D_R^{\mathrm{sym}}:=D_R$. In both cases $D_R^{\mathrm{sym}}$ is convex and contained in $B_{\frac{3R}{2}}(z_c)$, where $z_c:=(x_c',0)$ if $\Sigma_R\ne\emptyset$ (then $|x-z_c|^2<\frac{R^2}{4}+R^2$ on $D_R^{\mathrm{sym}}$) and $z_c:=x_c$ otherwise.

We freeze at the fixed centre $x_c$ (and, in case~(N), at its projection $\hat x_c:=(x_c',0)$ for the datum $C$):
\[
\begin{gathered}
\overline A(\eta):=A\bigl(x_c,u(x_c),\eta\bigr),\qquad G_c:=G(x_c,\cdot),\qquad g_c:=\partial_tG_c,\\[2pt]
C_*:=\begin{cases}C\bigl(\hat x_c,u(\hat x_c)\bigr)&\text{in case~(N)},\\[2pt] 0&\text{in case~(D)},\end{cases}
\end{gathered}
\]
so that $\overline A$ and $G_c$ satisfy \eqref{6.31}--\eqref{6.41} and the properties listed before \eqref{FP}, and $|C_*|\le\widetilde M$ in case~(N). Let
\[
\mathcal W:=\{\psi\in W^{1,G_c}(D_R):\ \psi=0\text{ on }\Gamma_R\text{ in the sense of traces}\}
\]
(in case~(D) this is $W_0^{1,G_c}(D_R)$). The frozen problem is: find $v\in u+\mathcal W$ with
\begin{equation}\label{mixed-BVP}
\int_{D_R}\overline A(Dv)\cdot D\psi\, \mathrm{d}x=\int_{\Sigma_R}C_*\,\psi\, \mathrm{d}y'\qquad\forall\,\psi\in\mathcal W,
\tag{$\mathcal{FP}_2$}
\end{equation}
that is, $-\operatorname{div}\overline A(Dv)=0$ in $D_R$ and $v=u$ on $\Gamma_R$, together with $\overline A(Dv)\cdot\nu=C_*$ on $\Sigma_R$ ($\nu=-e_n$) in case~(N). If $\Sigma_R=\emptyset$ (in particular in case~(D)), \eqref{mixed-BVP} is a Dirichlet problem in $D_R$. In case~(N), since every $\psi\in\mathcal W$ vanishes on the top face of $D_R$, the fundamental theorem of calculus in the $x_n$-direction gives
\begin{equation}\label{eq:ftc}
\int_{\Sigma_R}|\psi|\, \mathrm{d}y'\le\int_{D_R}|D\psi|\, \mathrm{d}x\qquad\forall\,\psi\in\mathcal W.
\end{equation}
The following lemma is the boundary counterpart of Lemmas~\ref{Lem1.2} and~\ref{Lem1.4a}; the regularity statements are adapted from Lieberman \cite[Lemmas 5.1 and 5.2]{Li1991} and, near the flat portion $\Lambda_R$ of $\partial D_R$, from \cite[Theorems 1.5 and 1.6]{Li1993} in case~(D) (Dirichlet datum of class $C^{1,\beta_3}$) and from Antonini \cite[Theorem 1.3]{ACAN} in case~(N) (conormal problem on a flat boundary portion with a H\"older continuous, here constant, datum); see also \cite[Theorems 1.1 and 1.2]{ACAN}. The scaling invariance of the problem is used to make the constants independent of $R$.

\begin{lemma}\label{Lem2.21}
Let $x_c,R,D_R,G_c,\overline A,C_*$ be as above. Then $u\in W^{1,G_c}(D_R)\cap C(\overline{D_R})$ and \eqref{mixed-BVP} has a unique solution $v$; moreover $v\in C^1(\overline{D_{\frac{R}{2}}})$, and in case~(D) also $v\in C(\overline{D_R})$ with $v=u$ on $\partial D_R$. There are $c_1,c_2,c_3,c_4\ge1$, $C_L\ge1$ and $\lambda_L\in(0,1)$, depending only on $n,g^\pm,F,\lambda_0,\Lambda_0$ and on $M_\phi$ (case~(D)) or $\widetilde M$ (case~(N)) (the constants of the flattened problem, hence only on $\mathrm{data}_\partial$), and independent of $x_c$, $R$, $u$, such that, with $\lambda_3:=\min\{\lambda_Lg^-,\frac{1}{2}\}$ and $D_\rho=D_\rho(x_c)$:
\begin{gather}
\sup_{D_{\frac{R}{2}}}G_c(|Dv|)\le c_1\Big(R^{-n}\int_{D_R}G_c(|Dv|)\, \mathrm{d}x+1\Big),\label{eq5.9}\\
\operatorname*{osc}_{D_\rho}Dv\le C_L\Big(\frac{\rho}{R}\Big)^{\lambda_L}\Big(\sup_{D_{\frac{R}{2}}}|Dv|+1\Big),\qquad0<\rho\le\frac{R}{2},\label{eq5.10a}\\
\int_{D_\rho}G_c\bigl(|Dv-(Dv)_{D_\rho}|\bigr)\, \mathrm{d}x\le c_2\rho^n\Big(\frac{\rho}{R}\Big)^{\lambda_3}\Big(R^{-n}\int_{D_R}G_c(|Dv|)\, \mathrm{d}x+1\Big),\qquad0<\rho\le\frac{R}{2},\label{eq5.10}\\
\int_{D_R}G_c(|Dv|)\, \mathrm{d}x\le c_3\int_{D_R}\bigl(1+G_c(|Du|)\bigr)\, \mathrm{d}x,\label{eq5.11}\\
\sup_{D_R}|u-v|\le c_4\,\bigl(\operatorname*{osc}_{D_R}u+R\bigr).\label{eq5.12}
\end{gather}
Finally, with $\delta'$ as in the standing choices,
\begin{equation}\label{eq5.13}
\int_{D_R}G_c(|Dv|)^{1+\delta'}\, \mathrm{d}x\le C\int_{D_R}\bigl(1+G_c(|Du|)\bigr)^{1+\delta'}\, \mathrm{d}x ,
\end{equation}
$v\in W^{1,G}(D_R)\cap L^\infty(D_R)$, and the extension of $u-v$ by zero to $Q^+$ belongs to $W^{1,G}(Q^+)\cap L^\infty(Q^+)$ (to $W_0^{1,G}(Q^+)$ in case~(D)) and vanishes outside the compact set $\overline{D_R}\subset Q(0,2r_\Omega)$; in particular it is an admissible test function in \eqref{eq:flat-weak} (case~(D)) or \eqref{eq:flat-weak-N} (case~(N)), and its restriction to $D_R$ is admissible in \eqref{mixed-BVP}.
\end{lemma}

\begin{proof}
\textbf{Step 1: well-posedness.} For $x\in D_R\subset\Omega_R(x_c)$ the points $x_c,x$ lie in $\overline{B_R(x_c)}\cap\overline{Q^+}$, $\omega(R)\le g^-\delta'$ and $R<\widetilde R_0$; hence Lemma~\ref{lem:comp-factor} (with $\kappa=0$) gives $G_c(|Du|)\le C(1+G(x,|Du|))^{1+\delta'}$ in $D_R$, and the right-hand side is integrable by Lemma~\ref{Lem2.1}(ii) (with $y_0=x_c$; note $\delta'\le\delta_{\mathrm{hi}}$). Thus $u\in W^{1,G_c}(D_R)$, and $u\in C(\overline{D_R})$ by the H\"older continuity of $u$. In case~(N), by \eqref{eq:ftc}, $\psi\mapsto\int_{\Sigma_R}C_*\psi\, \mathrm{d}y'$ is a bounded linear functional on $\mathcal W$. In both cases the operator $\overline A$ is bounded, coercive and strictly monotone by Proposition~\ref{Prop2.2} (applied with $G_c$ in place of $G$), and the Browder--Minty theorem gives existence and uniqueness of $v\in u+\mathcal W$. In case~(D) the convex set $D_R$ satisfies a uniform exterior cone condition; hence, by \cite[Lemma 5.1]{Li1991}, $v\in C(\overline{D_R})$ and $v=u$ on $\partial D_R$.

\textbf{Step 2: \eqref{eq5.11} and \eqref{eq5.12}.} Testing \eqref{mixed-BVP} with $v-u\in\mathcal W$ and using \eqref{eq:ftc} (in case~(D) the last term below is absent, since $C_*=0$),
\[
\int_{D_R}\overline A(Dv)\cdot Dv\, \mathrm{d}x\le\int_{D_R}\overline A(Dv)\cdot Du\, \mathrm{d}x+\widetilde M\int_{D_R}|Dv-Du|\, \mathrm{d}x ;
\]
since $\overline A(\xi)\cdot\xi\ge\lambda_0G_c(|\xi|)$, Young's inequality \eqref{Yi}, \eqref{L1} and $\widetilde Mt\le\varepsilon_3G_c(t)+C_{\varepsilon_3}$ give \eqref{eq5.11}. For \eqref{eq5.12} let $l\ge k:=\sup_{D_R}u$ and $A(l):=\{x\in D_R:v>l\}$. The function $(v-l)^+$ belongs to $\mathcal W$, because its trace on $\Gamma_R$ is $(\operatorname{tr}v-l)^+=(\operatorname{tr}u-l)^+=0$. Testing \eqref{mixed-BVP} with it and using \eqref{eq:ftc},
\[
\lambda_0\int_{A(l)}G_c(|Dv|)\, \mathrm{d}x\le\int_{\Sigma_R}C_*(v-l)^+\, \mathrm{d}y'\le\widetilde M\int_{A(l)}|Dv|\, \mathrm{d}x\le\frac{\lambda_0}{2}\int_{A(l)}G_c(|Dv|)\, \mathrm{d}x+C|A(l)| ,
\]
so that $\int_{D_R}|D(v-l)^+|\, \mathrm{d}x\le C|A(l)|$ (use $t\le1+FG_c(t)$). The function $(v-l)^+$ vanishes on $\Gamma_R$; if $\Sigma_R\ne\emptyset$ its even reflection across $\{x_n=0\}$ belongs to $W_0^{1,1}(D^{\mathrm{sym}}_R)$, and in any case the Sobolev inequality gives
\[
\|(v-l)^+\|_{L^{\frac{n}{n-1}}(D_R)}\le C(n)\int_{D_R}|D(v-l)^+|\, \mathrm{d}x.
\]
Hence $(m-l)|A(m)|^{\frac{n-1}{n}}\le C|A(l)|$ for $m>l\ge k$, and Stampacchia's lemma (see \cite[Ch.~II, Lemma~4.1]{OL}) yields $|A(k+d)|=0$ with $d\le C|A(k)|^{\frac{1}{n}}\le C|D_R|^{\frac{1}{n}}\le CR$. Thus $v\le\sup_{D_R}u+CR$; applying the same argument to $-v$ (which solves the problem with $\overline A(\xi)$ replaced by $-\overline A(-\xi)$ and $C_*$ by $-C_*$) we get $v\ge\inf_{D_R}u-CR$, which is \eqref{eq5.12} (if $\Sigma_R=\emptyset$, in particular in case~(D), then $C_*=0$ plays no role and the term $CR$ can be dropped).

\textbf{Step 3: gradient estimates.} The structure $(\overline A,G_c)$ does not depend on $x$; hence the problem is invariant under translations and under $v\mapsto R^{-1}\bigl(v(x_c+R\,\cdot)-k_c\bigr)$. In case~(N) this leaves $\overline A$ and the conormal datum $C_*$ unchanged (the flux $\overline A(Dv)$ is invariant because $Dv$ is). In case~(D), with $k_c:=\phi(x_c',0)$, it maps the boundary datum $\phi$ on the flat part to $\phi_R(y):=R^{-1}(\phi(x_c+Ry)-k_c)$, which satisfies $|D\phi_R|\le M_\phi$ and $[D\phi_R]_{\beta_3}\le R^{\beta_3}M_\phi\le M_\phi$. Since $\operatorname{dist}(D_{\frac{R}{2}},\partial D_R\setminus\overline{\Lambda_R})\ge\frac{R}{4}$, only interior estimates and estimates near the flat portion $\Lambda_R$ are needed. The interior estimates are \cite[Lemma 5.2]{Li1991} and \cite[Theorem 1.1]{ACAN}; near $\Lambda_R$ we use the local boundary estimates for a flat boundary portion carrying $C^{1,\beta_3}$ Dirichlet data \cite[Theorems 1.5 and 1.6]{Li1993} (see also \cite[Theorem 1.2]{ACAN}) in case~(D), and the local conormal estimates for a flat boundary portion with constant (hence H\"older continuous) datum \cite[Theorem 1.3]{ACAN} in case~(N). Combining them by a standard covering argument with balls of radius comparable to $R$, in the rescaled variables, we obtain \eqref{eq5.9} and \eqref{eq5.10a}, with constants depending only on $n,g^\pm,\lambda_0,\Lambda_0,F$ and $M_\phi$ resp. $\widetilde M$, and independent of $x_c$, $R$ and of the position of $x_c$ relative to $\{x_n=0\}$ (the additive $1$ accounts for the datum: $|D\phi_R|\le M_\phi$ in case~(D), $|C_*|\le\widetilde M$ in case~(N); if $\Lambda_R=\emptyset$ these are the interior estimates \eqref{eq3.10} and \eqref{eq:holder-grad} of Lemma~\ref{Lem1.2}). Estimate \eqref{eq5.10} follows from \eqref{eq5.9} and \eqref{eq5.10a} as \eqref{eq3.11} follows from \eqref{eq3.10} and \eqref{eq:holder-grad} in Lemma~\ref{Lem1.2}: since $|Dv-(Dv)_{D_\rho}|\le\operatorname*{osc}_{D_\rho}Dv$ in $D_\rho$, Proposition~\ref{zoo}(1) and \eqref{eq5.10a}, with $S:=\sup_{D_{\frac{R}{2}}}|Dv|$, give
\[
\int_{D_\rho}G_c\bigl(|Dv-(Dv)_{D_\rho}|\bigr)\, \mathrm{d}x
\le|D_\rho|\,G_c\Big(C_L\Big(\frac{\rho}{R}\Big)^{\lambda_L}(S+1)\Big)
\le C|D_\rho|\Big(\frac{\rho}{R}\Big)^{\lambda_Lg^-}\bigl(G_c(S)+F\bigr),
\]
and \eqref{eq5.10} follows from \eqref{eq5.9}, because $|D_\rho|\le C\rho^n$, $\lambda_3\le\lambda_Lg^-$ and $\frac{\rho}{R}\le1$.

\textbf{Step 4: \eqref{eq5.13} and admissibility.} If $\Sigma_R\ne\emptyset$ let $\widetilde v,\widetilde u$ be the even reflections of $v,u$ across $\{x_n=0\}$ to $D^{\mathrm{sym}}_R$, and $\widetilde w:=\widetilde v-\widetilde u$ extended by zero outside $D_R^{\mathrm{sym}}$; if $\Sigma_R=\emptyset$ we put $\widetilde v=v$, $\widetilde u=u$ and $\widetilde w=v-u$ extended by zero. Since $v-u\in \mathcal W$, $\widetilde w\in W^{1,G_c}(\mathbb{R}^n)$. Let $r_0$ and $\varepsilon$ be the constants of Proposition~\ref{prop:SP} (for the class of $N$-functions of the flattened problem, as in Step~0 of Lemma~\ref{Lem2.1}; since $G_c$ does not depend on $x$, the additive constant in that proposition is not needed), $s:=1+\varepsilon$, put $B^\sharp:=B_{3R}(z_c)$ (it contains the $R$-neighbourhood of $D^{\mathrm{sym}}_R$) and define, on $B^\sharp$, $f:=G_c(|D\widetilde v|)^{\frac{1}{s}}$ in $D^{\mathrm{sym}}_R$, $f:=0$ elsewhere, and $\varphi:=(1+G_c(|D\widetilde u|))^{\frac{1}{s}}$ in $D^{\mathrm{sym}}_R$, $\varphi:=1$ elsewhere. By Lemma~\ref{lem:comp-factor} (with $\kappa=0$) and $\omega(R)\le g^-\delta'$ we have
\[
(1+G_c(|Du|))^{1+\delta'}\le C(1+G(x,|Du|))^{(1+\delta')^2}\le C(1+G(x,|Du|))^{1+3\delta'}
\]
in $D_R$, which is integrable by Lemma~\ref{Lem2.1}(ii) because $3\delta'\le\frac{3}{4}\delta_{\mathrm{hi}}$; hence $f\in L^s(B^\sharp)$ and $\varphi\in L^{s(1+\delta')}(B^\sharp)$. We claim that \eqref{eq:RH-frozen} holds, with a constant $b_*^\partial$ depending only on $n,g^\pm,F,\lambda_0,\Lambda_0$ and, in case~(N), $\widetilde M$, whenever $B_{4r}(z)\subset B^\sharp$ (then $4r\le3R\le3r_*\le r_0$). Let $z\in\mathbb{R}^n$, $\bar z:=(z',-z_n)$ and $\zeta\in C_c^\infty(B_{2r}(z))$ with $0\le\zeta\le1$, $\zeta=1$ on $B_r(z)$, $|D\zeta|\le\frac{2}{r}$; let $\bar\zeta(x):=\zeta(x',-x_n)$.
\begin{itemize}
\item If $B_{2r}(z)\cap D^{\mathrm{sym}}_R=\emptyset$, then $f=0$ on $B_r(z)$ and the claim is trivial.
\item If $B_{2r}(z)\subset D^{\mathrm{sym}}_R$, put $\omega:=(\widetilde v)_{B_{2r}(z)}$. If $\Sigma_R=\emptyset$ we test the equation of $v$ with $\zeta^{g^+}(v-\omega)\in W_0^{1,G_c}(D_R)$, as in case~(ii) of Lemma~\ref{Lem1.4a}. If $\Sigma_R\ne\emptyset$, the restrictions to $D_R$ of $\zeta^{g^+}(\widetilde v-\omega)$ and of $\bar\zeta^{g^+}(\widetilde v-\omega)$ belong to $\mathcal W$ (their supports stay away from $\Gamma_R$) and are admissible in \eqref{mixed-BVP}; the boundary terms are bounded, by \eqref{eq:ftc}, by $\widetilde M\int_{D_R\cap B_{2r}(z)}|D(\cdot)|\, \mathrm{d}x$ and absorbed by Young's inequality, as in Step~1 of Lemma~\ref{Lem2.1}. Since $B_r(z)\cap\{x_n<0\}$ is the mirror image of $B_r(\bar z)\cap\{x_n>0\}$ (and similarly for $B_{2r}$), adding the two inequalities gives
\[
\int_{B_r(z)}f^s\, \mathrm{d}x\le C\int_{B_{2r}(z)}G_c\bigl(\frac{|\widetilde v-\omega|}{r}\bigr)\, \mathrm{d}x+C|B_{2r}|,
\]
and Proposition~\ref{prop:SP} in $B_{2r}(z)$ gives \eqref{eq:RH-frozen}.
\item Otherwise $B_{2r}(z)$ contains a point $z'\in\partial D^{\mathrm{sym}}_R$. We test with $\zeta^{g^+}(v-u)$ and, if $\Sigma_R\ne\emptyset$, also with $\bar\zeta^{g^+}(v-u)$ (admissible: $v-u\in \mathcal W$). The same computation as in the previous item has the additional term
\[
\Lambda_1\int_{D_R\cap B_{2r}(z)}\zeta^{g^+}g_c(|Dv|)\,|Du|\, \mathrm{d}x ,
\]
which is absorbed by Young's inequality; with the boundary terms treated as above, it gives
\[
\int_{B_r(z)}f^s\, \mathrm{d}x\le C\int_{B_{2r}(z)}G_c\Big(\frac{|\widetilde w|}{r}\Big)\, \mathrm{d}x+C\int_{B_{2r}(z)\cap D^{\mathrm{sym}}_R}G_c(|D\widetilde u|)\, \mathrm{d}x+C|B_{2r}|.
\]
Since $D^{\mathrm{sym}}_R$ is convex, there is a unit vector $\nu$ with $D^{\mathrm{sym}}_R\subset\{(x-z')\cdot\nu<0\}$, so that $\widetilde w$ vanishes on the half ball $B_{2r}(z')\cap\{(x-z')\cdot\nu>0\}\subset B_{4r}(z)$, of measure $2^{-n-1}|B_{4r}|\ge4^{-n}|B_{4r}|$. Proposition~\ref{prop:SP} for $G_c$ in $B_{4r}(z)$ (version for functions vanishing on a set of measure $\ge4^{-n}|B_{4r}|$) and Jensen's inequality conclude as in Lemma~\ref{Lem1.4a}.
\end{itemize}
Put $\delta_*^\partial:=\delta_G(n,s,b_*^\partial,4)$, the exponent of Proposition~\ref{prop:Gehring} for $\kappa=4$; it depends only on $n$, $g^\pm$, $F$, $\lambda_0$, $\Lambda_0$ and, in case~(N), $\widetilde M$. Since $\delta'\le\delta_*^\partial$, Proposition~\ref{prop:Gehring} applied in $B^\sharp$ with $\kappa=4$, $\delta_1=\delta'$ and the pairs $B_{\frac{R}{4}}(z)\subset B_R(z)\subset B^\sharp$, $z\in D^{\mathrm{sym}}_R$, and a covering of $D^{\mathrm{sym}}_R$ by $N(n)$ balls $B_{\frac{R}{4}}(z_i)$, $z_i\in D_R^{\mathrm{sym}}$, give, as in Step~2 of Lemma~\ref{Lem1.4a} (with \eqref{eq5.11} in place of \eqref{eq3.12}; the energy of $\widetilde v$ is twice that of $v$ if $\Sigma_R\ne\emptyset$), the estimate \eqref{eq5.13}. By Lemma~\ref{lem:comp-factor} and $\omega(R)\le g^-\delta'$, $G(x,|Dv|)\le C(1+G_c(|Dv|))^{1+\delta'}$ in $D_R$, which is integrable by \eqref{eq5.13} and Step~1; as $v$ is bounded by \eqref{eq5.12}, $v\in W^{1,G}(D_R)\cap L^\infty(D_R)$. Finally, $u-v\in W^{1,G}(D_R)\cap L^\infty(D_R)$. In case~(N) it has zero trace on $\Gamma_R$, so its extension by zero to $Q^+$ has weak gradient $\mathbf 1_{D_R}D(u-v)\in L^G$ and belongs to $W^{1,G}(Q^+)\cap L^\infty(Q^+)$. In case~(D), $w:=v-u\in W^{1,G}(D_R)\cap C(\overline{D_R})$ vanishes on $\partial D_R$ and is bounded by \eqref{eq5.12}; the functions $w_k:=\bigl(w-\frac{1}{k}\bigr)^+-\bigl(w+\frac{1}{k}\bigr)^-$ have their support in $\{|w|\ge\frac{1}{k}\}$, a compact subset of $D_R$, hence belong to $W_0^{1,G}(D_R)$, and $w_k\to w$ in $W^{1,G}(D_R)$ exactly as in Step~3 of the proof of Lemma~\ref{Lem1.4a}. Therefore $v-u\in W_0^{1,G}(D_R)\cap L^\infty(D_R)$, and its extension by zero belongs to $W_0^{1,G}(Q^+)\cap L^\infty(Q^+)$.
\end{proof}

The comparison estimate and the Morrey estimates are obtained as in the interior case, using Lemmas~\ref{Lem2.1} and~\ref{Lem2.21} instead of Lemmas~\ref{Lem1.1}, \ref{Lem1.2} and~\ref{Lem1.4a}; as before we freeze at the fixed centre $x_c$ and use Lemma~\ref{lem:comp-factor} and Corollary~\ref{cor:g-comp} with this fixed centre.

\begin{lemma}\label{Lem2.2}
Let $v$ be as in Lemma~\ref{Lem2.21}, and let $\beta$ and $\theta_*$ be as above. For every $\delta\in\bigl(0,\frac{\delta_{\mathrm{hi}}}{2}\bigr]$ there are $R_2=R_2(\delta)\in(0,\min\{R_1,R_\omega\}]$ and $c>0$, depending only on $\mathrm{data}_\partial$, $\delta$ and $L_2$, such that, for every $x_c$ with $(x_c)_n\ge0$ and every $R\le R_2$ with $B_{2R}(x_c)\subset B_{2R_1}(x_0)$,
\begin{equation}\label{eq5.14}
\int_{D_R}G_c(|Du-Dv|)\, \mathrm{d}x\le cR^{\beta\theta_*-n\frac{\delta}{2}}\,\Theta_R^{1+\frac{\delta}{2}},\qquad\Theta_R:=\int_{\Omega_{2R}(x_c)}\bigl(1+G(x,|Du|)\bigr)\, \mathrm{d}x .
\end{equation}
\end{lemma}

\begin{proof}
We repeat the proof of Lemma~\ref{Lem1.3} on $D_R\subset\Omega_R(x_c)$ with the same $X:=R^{-n\frac{\delta}{2}}\Theta_R^{1+\frac{\delta}{2}}$ and the same numbers $\sigma,\varepsilon_\delta,R_2$. Lemma~\ref{Lem2.1}(ii) (estimate \eqref{eq4.1b} with $y_0=x_c$ and $\frac{\delta}{2}$ in place of $\delta$) replaces Lemma~\ref{Lem1.1}, so that $\int_{D_R}(1+G(x,|Du|))^{1+\frac{\delta}{2}}\, \mathrm{d}x\le CX$; moreover $X\ge c\,\Theta_R$, because $\Theta_R\ge|\Omega_{2R}(x_c)|\ge\frac{1}{2}|B_{2R}|$. By Lemma~\ref{Lem2.21}, $u-v$ (extended by zero) is admissible in \eqref{eq:flat-weak} (case~(D)) or \eqref{eq:flat-weak-N} (case~(N)) and, restricted to $D_R$, in \eqref{mixed-BVP}. Subtracting the two identities (in case~(N) the boundary integral in \eqref{eq:flat-weak-N} reduces to $\Sigma_R$, the only part of $\Sigma$ where $\operatorname{tr}(u-v)\neq0$) we get
\[
I:=\int_{D_R}\bigl(\overline A(Du)-\overline A(Dv)\bigr)\cdot(Du-Dv)\, \mathrm{d}x=I_1+I_2+I_3,
\]
with
\[
\begin{aligned}
I_1&:=\int_{D_R}(\overline A(Du)-A(x,u,Du))\cdot(Du-Dv)\, \mathrm{d}x,\\
I_2&:=\int_{D_R}B(x,u,Du)(u-v)\, \mathrm{d}x,\\
I_3&:=\int_{\Sigma_R}\bigl(C(y',u)-C_*\bigr)(u-v)\, \mathrm{d}y'
\end{aligned}
\]
(so that $I_3=0$ in case~(D), where $\Sigma_R=\emptyset$).

For $I_1$ we use \eqref{7.41} for the flattened field $A$ (with $\beta_1$ replaced by $\min\{\beta_1,\beta_3\}$, see item (b) of the invariance of the structure), $|x-x_c|<R$ and $|u(x)-u(x_c)|\le L_2R^{\alpha_1}$ in $D_R$, Corollary~\ref{cor:g-comp} for the pair $(x_c,x)$ and Lemma~\ref{lem:comp-factor}; the estimates of the two resulting integrals are literally those of Lemma~\ref{Lem1.3}, with \eqref{eq5.11} in place of \eqref{eq3.12}, and give $|I_1|\le CR^\beta X$. For $I_2$ we use \eqref{eq5.12} in place of \eqref{eq3.13} (note $\operatorname*{osc}_{D_R}u+R\le CR^{\alpha_1}$) and Lemma~\ref{lem:B}: $|I_2|\le CR^{\alpha_1}\Theta_R\le CR^\beta X$. For $I_3$ (case~(N)) we use \eqref{C1} for the flattened datum and the H\"older continuity of $u$: for $y'\in\Sigma_R$, $|(y',0)-\hat x_c|<R$ and $|u(y',0)-u(\hat x_c)|\le L_2R^{\alpha_1}$, hence $|C(y',u(y',0))-C_*|\le CR^\beta$. By \eqref{eq:ftc} (applied to $u-v\in\mathcal W$) and $t\le1+FG_c(t)$,
\[
|I_3|\le CR^\beta\int_{D_R}|Du-Dv|\, \mathrm{d}x\le CR^\beta\Big(|D_R|+\int_{D_R}G_c(|Du|+|Dv|)\, \mathrm{d}x\Big)\le CR^\beta X ,
\]
where the last step uses \eqref{eq5.11}, the bound $\int_{D_R}G_c(|Du|)\, \mathrm{d}x\le CX$ obtained in the proof of Lemma~\ref{Lem1.3} and $|D_R|\le CR^n\le CX$. Hence $I\le CR^\beta X$ (no loss in the exponent $\beta$ occurs), and the conclusion follows from Lemma~\ref{lem:orlicz-ineq}, Proposition~\ref{Prop2.2} and H\"older's inequality as in Lemma~\ref{Lem1.3}. In case~(D) the datum $\phi$ does not enter these estimates: it affects only the boundary values of $u$ on the flat part, which are encoded in the H\"older continuity of $u$ and in the fact that $u-v\in W_0^{1,G}(D_R)$.
\end{proof}

Let $\delta:=\min\bigl\{1,\frac{\delta_{\mathrm{hi}}}{2},\frac{\beta\theta_*}{2\tau_0}\bigr\}$, a number depending only on $\mathrm{data}_\partial$, and let $R_2=R_2(\delta)$ be the radius given by Lemma~\ref{Lem2.2} for this $\delta$; thus $\tau_0\delta\le\frac{\beta\theta_*}{2}$.

\begin{lemma}\label{Lem2.3}
For every $\tau\in(0,n)$ there are $R_\tau\in\bigl(0,\min\{R_2,\frac{R_1}{2}\}\bigr]$ and $C_\tau>0$, depending only on $\mathrm{data}_\partial$, $\tau$ and $L_2$, such that
\begin{equation}\label{eq4.17}
\psi_y(\rho)+\varphi_y(\rho)\le C_\tau\rho^{\,n-\tau}\qquad\forall y\in\overline{Q^+}\cap B_{\frac{R_1}{2}}(x_0),\ \forall\rho\in(0,R_\tau].
\end{equation}
\end{lemma}

\begin{proof}
We apply Lemma~\ref{Lem:boot} with
\[
\mathcal Y:=\overline{Q^+}\cap B_{\frac{R_1}{2}}(x_0),\qquad E_r(y):=\Omega_r(y),\qquad R_\flat:=\min\Big\{R_2,\frac{R_1}{2}\Big\},\qquad\chi_0:=\beta\theta_*,
\]
with the numbers $\tau_0$, $\delta$ fixed above and $\delta'$ of the standing choices; the functions $\psi_y,\varphi_y$ are those of the standing choices. We have $R_\flat\le R_2\le R_\omega$, hence $\omega(r)\le g^-\delta'$ for $0<r\le R_\flat$, and $\tau_0\delta\le\frac{\chi_0}{2}$. Let $y\in\mathcal Y$. For $0<r\le R_\flat$ we have $2r\le R_1$, so that $B_{2r}(y)\subset B_{2R_1}(x_0)\subset Q(0,r_\Omega)$ by \eqref{eq:star}, and the hypotheses of Lemma~\ref{Lem:boot} are verified as follows.
\begin{enumerate}[(a)]
\item Lemma~\ref{lem:comp-factor}, for the flattened objects (Convention), applies to the points of $B_r(y)\cap\overline{Q^+}$, because $r\le R_\omega\le r_*\le\frac{\widetilde R_0}{4}$.
\item This is Lemma~\ref{Lem2.0} (note that $|y-x_0|\le\frac{R_1}{2}$ and $r\le R_1$).
\item This is Lemma~\ref{Lem2.1}(ii), estimate \eqref{eq4.1b} with $y_0=y$, at the scale $r\le\frac{R_\flat}{2}\le r_*$, because $\delta\le\frac{\delta_{\mathrm{hi}}}{2}$ and $\delta'\le\frac{\delta_{\mathrm{hi}}}{4}$.
\item Let $0<R\le\frac{R_\flat}{2}$ and $0<\rho\le\frac{R}{4}$, and put $x_c:=y$, $G_c:=G(y,\cdot)$. Since $2R\le R_\flat\le\min\{R_2,\frac{R_1}{2}\}$, we have $B_{2R}(y)\subset B_{R_1}(x_0)$ and $R\le R_2\le R_\omega$, so that Lemmas~\ref{Lem2.21} and~\ref{Lem2.2} apply; let $v$ be the solution of \eqref{mixed-BVP} in $D_R=D_R(y)$. By \eqref{eq:cyl}, $\Omega_\rho(y)\subset D_{2\rho}(y)\subset D_{\frac{R}{2}}(y)$ and $D_R\subset\Omega_R(y)$; by convexity, the $\Delta_2$-condition, \eqref{eq5.9}, \eqref{eq5.11} and \eqref{eq5.14} (recall $\Theta_R=\psi_y(2R)$),
\begin{align*}
\int_{\Omega_\rho(y)}G_c(|Du|)\, \mathrm{d}x
&\le C\rho^n\sup_{D_{\frac{R}{2}}}G_c(|Dv|)+C\int_{D_R}G_c(|Du-Dv|)\, \mathrm{d}x\\
&\le C\Big(\frac{\rho}{R}\Big)^n\int_{\Omega_R(y)}\bigl(1+G_c(|Du|)\bigr)\, \mathrm{d}x+C\rho^n+CR^{\beta\theta_*-\frac{n\delta}{2}}\psi_y(2R)^{1+\frac{\delta}{2}},
\end{align*}
the term $C\rho^n$ coming from the additive $1$ in \eqref{eq5.9}.
\end{enumerate}
Lemma~\ref{Lem:boot} now gives \eqref{eq4.17}, with some $R_\tau\le R_\flat$ and constants depending only on $\mathrm{data}_\partial$, $\tau$ and $L_2$.
\end{proof}

\begin{proof}[\textbf{Conclusion of the Proofs of Theorems~\ref{Thm1.2} and~\ref{Thm1.3}}]
\mbox{}\par\noindent
\textbf{Step 1: points far from the boundary.} The function $u$ is a bounded generalized solution of the equation in $\Omega$, satisfying \eqref{M} (in case~(N) because \eqref{FNV} contains \eqref{FDV}). By Theorem~\ref{Thm1.1} applied with
\[
\Omega_0:=\Bigl\{x\in\Omega:\mathrm{dist}(x,\partial\Omega)\ge\frac{R_1}{2}\Bigr\},
\]
we have $u\in C^{1,\alpha_{\mathrm{int}}}(\overline{\Omega_0})$, where $\alpha_{\mathrm{int}}$ depends only on $\mathrm{data}$ and $\|u\|_{C^{1,\alpha_{\mathrm{int}}}(\overline{\Omega_0})}$ only on $\mathrm{data}$ and $R_1$, hence only on $\mathrm{data}_\partial$.

\medskip
\noindent
\textbf{Step 2: constants and scales near the boundary.} Let $x_0\in\overline\Omega$ with $\mathrm{dist}(x_0,\partial\Omega)<2R_1$; we use the flattened coordinates and the notation of the standing choices. We fix the following numbers.
\begin{itemize}
\item Let $\beta,\theta_*$ be as above, $\lambda_3$ as in Lemma~\ref{Lem2.21}, and let $\tau,\theta,\varepsilon$ be the numbers given by Lemma~\ref{Lem:expo} for $\chi_0:=\beta\theta_*$, $\lambda:=\lambda_3$ and $\kappa_0:=4$; put $\alpha_\partial:=\frac{\varepsilon}{g^+}$.
\item Let $\delta$ and $R_2=R_2(\delta)$ be the numbers fixed before Lemma~\ref{Lem2.3}, and let $R_\tau\le\min\{R_2,\frac{R_1}{2}\}$ be given by Lemma~\ref{Lem2.3} for this $\tau$.
\item Let $x_c\in\overline{Q^+}$ with $|x_c-x_0|\le\frac{R_1}{4}$, let $0<\rho<\rho_0:=\bigl(\frac{R_\tau}{8}\bigr)^{1+\theta}$ and
\[
R:=(4\rho)^{\frac{1}{1+\theta}} .
\]
\end{itemize}
Then $\rho=\frac{1}{4}R^{1+\theta}$, $4\rho\le R$ (because $R\le1$), $2R<R_\tau$, $B_{2R}(x_c)\subset B_{2R_1}(x_0)$ and $R\le R_2\le R_\omega$, so that Lemmas~\ref{Lem2.21} and~\ref{Lem2.2} apply; let $v$ be the solution of \eqref{mixed-BVP} in $D_R=D_R(x_c)$ and $G_c=G(x_c,\cdot)$. By \eqref{eq:cyl},
\[
\Omega_\rho(x_c)\subset D_{2\rho}\subset D_{\frac{R}{2}}\subset D_R .
\]

\medskip
\noindent
\textbf{Step 3: decay of the oscillation of $Du$.} For every constant $c\in\mathbb{R}^n$,
\[
|Du-(Du)_{\Omega_\rho}|\le|Du-c|+\fint_{\Omega_\rho}|Du-c|\, \mathrm{d}x;
\]
choosing $c:=(Dv)_{D_{2\rho}}$ and using Jensen's inequality, convexity and the $\Delta_2$-condition,
\[
\int_{\Omega_\rho(x_c)}G_c\bigl(|Du-(Du)_{\Omega_\rho}|\bigr)\, \mathrm{d}x\le C\Bigg[\int_{D_{2\rho}}G_c\bigl(|Dv-(Dv)_{D_{2\rho}}|\bigr)\, \mathrm{d}x+\int_{D_R}G_c(|Du-Dv|)\, \mathrm{d}x\Bigg].
\]

Since $2\rho\le\frac{R}{2}$, \eqref{eq5.10}, \eqref{eq5.11} and Lemma~\ref{Lem2.3} (which gives $\int_{D_R}(1+G_c(|Du|))\, \mathrm{d}x\le\varphi_{x_c}(R)\le C_\tau R^{n-\tau}$) yield, using $R\le1$,
\[
\begin{aligned}
\int_{D_{2\rho}}G_c\bigl(|Dv-(Dv)_{D_{2\rho}}|\bigr)\, \mathrm{d}x
&\le C\rho^n\Big(\frac{\rho}{R}\Big)^{\lambda_3}\bigl(R^{-n}\varphi_{x_c}(R)+1\bigr)\\
&\le C\rho^n\Big(\frac{\rho}{R}\Big)^{\lambda_3}R^{-\tau}.
\end{aligned}
\]

By \eqref{eq5.14} and Lemma~\ref{Lem2.3} (which gives $\Theta_R=\psi_{x_c}(2R)\le C_\tau(2R)^{n-\tau}$), since $\tau\delta\le\tau<\frac{\beta\theta_*}{4}$,
\[
\begin{aligned}
\int_{D_R}G_c(|Du-Dv|)\, \mathrm{d}x
&\le CR^{\beta\theta_*-\frac{n\delta}{2}+(n-\tau)(1+\frac{\delta}{2})}\\
&=CR^{\,n-\tau+\beta\theta_*-\frac{\tau\delta}{2}}\\
&\le CR^{\,n-\tau+\frac{\beta\theta_*}{2}} .
\end{aligned}
\]

By \eqref{eq:expo} (with $\kappa_0=4$) both right-hand sides are bounded by $C\rho^nR^\tau=C\rho^n(4\rho)^\varepsilon$; hence
\[
\int_{\Omega_\rho(x_c)}G_c(|Du-(Du)_{\Omega_\rho}|)\, \mathrm{d}x\le C\rho^{n+\varepsilon}.
\]

Applying \eqref{eq:conv} (with $E=\Omega_\rho(x_c)\subset B_\rho(x_c)$, $h=|Du-(Du)_{\Omega_\rho}|$ and $M=C$) we get
\[
\int_{\Omega_\rho(x_c)}|Du-(Du)_{\Omega_\rho}|^{g^-}\, \mathrm{d}x\le C\rho^{\,n+\frac{\varepsilon g^-}{g^+}}=C\rho^{\,n+g^-\alpha_\partial}
\]
for all $x_c\in\overline{Q^+}$ with $|x_c-x_0|\le\frac{R_1}{4}$ and all $\rho<\rho_0$. Moreover $Du\in L^{g^-}(\Omega_{R_1}(x_0))$ by Lemma~\ref{Lem2.0} and Proposition~\ref{zoo}(2).

\medskip
\noindent
\textbf{Step 4: H\"older continuity of $Du$ in the flattened coordinates.} Let $U:=Q^+\cap B_{\frac{R_1}{8}}(x_0)$, a convex, hence Lipschitz, domain. For $y\in\overline U$ and $\rho\le\rho_0$ we have $B_\rho(y)\cap U\subset\Omega_\rho(y)$, and for every constant $c$,
\[
\int_{B_\rho(y)\cap U}|Du-\{Du\}_{y,\rho}|^{g^-}\, \mathrm{d}x\le2^{g^-}\int_{B_\rho(y)\cap U}|Du-c|^{g^-}\, \mathrm{d}x,
\]
where $\{Du\}_{y,\rho}$ is the mean over $B_\rho(y)\cap U$. Choosing $c:=(Du)_{\Omega_\rho(y)}$, we see that the hypothesis of Proposition~\ref{prop:Camp} holds in $U$ with $s=g^-$, $\alpha=\alpha_\partial$ and radii $\rho\le\frac{\rho_0}{2}$. Therefore $Du\in C^{0,\alpha_\partial}(\overline U)$, that is $u\in C^{1,\alpha_\partial}(\overline U)$ in the flattened coordinates.

\medskip
\noindent
\textbf{Step 5: back to the original coordinates and conclusion.} Let $\xi_0\in\overline\Omega$ denote the point $x_0$ of the original coordinates, so that the flattened point is $x_0=\Psi_z(\xi_0)$. We have $\alpha_\partial<\tau\le\frac{\beta\theta_*}{20}\le\beta_3$ (Lemma~\ref{Lem:expo}). Since
\[
\Psi_z\Bigl(\overline\Omega\cap B_{\frac{R_1}{8K_\Omega}}(\xi_0)\Bigr)\subset\overline U
\]
(recall $|\Psi_z(\xi)-\Psi_z(\xi_0)|\le K_\Omega|\xi-\xi_0|$) and $Du=(D\Psi_z)^T\,(D\hat u)\circ\Psi_z$ with $D\Psi_z\in C^{0,\beta_3}$ and $\Psi_z$ Lipschitz, we get
\[
u\in C^{1,\alpha_\partial}\Bigl(\overline\Omega\cap B_{\frac{R_1}{8K_\Omega}}(\xi_0)\Bigr),
\]
with a norm bounded in terms of $\mathrm{data}_\partial$ and $L_2$.

Set $\alpha:=\min\{\alpha_{\mathrm{int}},\alpha_\partial\}$. The compact set $\overline\Omega$ is covered by
\[
\Omega_0^\circ:=\Bigl\{x\in\Omega:\mathrm{dist}(x,\partial\Omega)>\frac{R_1}{2}\Bigr\}
\]
and by the balls $B_{\frac{R_1}{8K_\Omega}}(\xi_0)$, $\xi_0\in\overline\Omega$, $\mathrm{dist}(\xi_0,\partial\Omega)\le \frac{R_1}{2}$; extracting a finite subcover and using, for pairs of points not contained in a common member of the subcover (whose distance is bounded below by a Lebesgue number), the trivial bound $|Du(x)-Du(y)|\le2\sup|Du|$, we obtain $u\in C^{1,\alpha}(\overline\Omega)$.

\medskip
\noindent
\textit{Dependence of $\alpha$ and of the norm.}
\begin{itemize}
\item The exponent $\alpha_\partial=\frac{\varepsilon}{g^+}$ is a function of $n$, $g^\pm$, $\beta=\min\{\beta_1,\beta_3,\alpha_1\beta_2\}$, $\theta_*$ and $\lambda_3$ only.
\item $\alpha_1$ depends only on $\mathrm{data}_\partial$ (Proposition~\ref{prop1.1}(2) in case~(D), \ref{prop1.1}(3) in case~(N)), and $\lambda_3$ only on $n,g^\pm,\lambda_0,\Lambda_0,F$, on $M_\phi$ (case~(D)) or $\widetilde M$ (case~(N)) and on $K_\Omega$ (Lemma~\ref{Lem2.21}).
\item Hence $\alpha$ depends only on $\mathrm{data}_\partial$, and on $\Omega$ only through $\mathrm{data}_\Omega$.
\item The radii $R_1$, $R_\tau$ and the constants of the Campanato estimates depend in addition on $L_2=\|u\|_{C^{0,\alpha_1}(\overline\Omega)}$, that is on $\Omega$; so does the number of balls in the finite covering.
\end{itemize}
This gives the bound for $\|u\|_{C^{1,\alpha}(\overline\Omega)}$ and completes the proofs of Theorems~\ref{Thm1.2} and~\ref{Thm1.3}.
\end{proof}

\begin{remark}
In Theorems~\ref{Thm1.2} and~\ref{Thm1.3} the boundary condition is prescribed on the whole boundary $\partial\Omega$. The arguments of Subsection~\ref{subsec:proof-bdry} are local near each boundary point (they use the structure conditions, the regularity of $\partial\Omega$ and of the data only in a neighbourhood of the point, together with the H\"older continuity of $u$ there). Consequently, if the boundary condition is prescribed only on a relatively open portion $\Sigma\subset\partial\Omega$ and $u$ is known to be locally H\"older continuous on $\Omega\cup\Sigma$, the same proofs give $u\in C^{1,\alpha}_{\mathrm{loc}}(\Omega\cup\Sigma)$, with $\alpha$ and the local norms depending on the local data; we do not detail this extension.
\end{remark}

\ \\
\textbf{Acknowledgment:}
The author thanks Prof. Carlo Alberto Antonini for valuable comments and corrections on an earlier arXiv version of this manuscript.
\ \\
\textbf{Author Contributions:} The author confirms sole responsibility for this work.\\
\ \\
\textbf{Data Availability:} Data sharing is not applicable to this article as no new data were created or analyzed in this study.\\
\ \\
\textbf{Declarations}\\
\ \\
\textbf{Ethical Approval:} Not applicable.\\
\ \\
\textbf{Conflict of interest:} The author declares that he has no Conflict of interest.

\end{document}